\documentclass{aims}

\usepackage{amsmath}
\usepackage{amssymb,amsmath,amsthm}
\usepackage{xcolor}

\usepackage[margin={3cm,2cm}]{geometry}
  \usepackage{paralist}
  \usepackage{graphics} %% add this and next lines if pictures should be in esp format
  \usepackage{epsfig} %For pictures: screened artwork should be set up with an 85 or 100 line screen
\usepackage{graphicx}
 \usepackage[colorlinks=true]{hyperref}
\hypersetup{urlcolor=blue, citecolor=red}

\def\currentvolume{X}
 \def\currentissue{X}
  \def\currentyear{200X}
   \def\currentmonth{XX}
    \def\ppages{X--XX}

\newtheorem{theorem}{Theorem}[section]

\theoremstyle{definition}

\newcommand{\bit}{\bibitem}

\newcommand{\um}{u^m}
\newcommand{\umm}{u^{m+1}}
\newcommand{\ummd}{u^{m+1/2}}
\newcommand{\vm}{v^m}
\newcommand{\vmm}{v^{m+1}}
\newcommand{\vmmd}{v^{m+1/2}}
\newcommand{\wm}{w^m}
\newcommand{\wmm}{w^{m+1}}
\newcommand{\wmmd}{w^{m+1/2}}

\newcommand{\grad}{\nabla}
\renewcommand{\div}{\nabla\cdot}
\newcommand{\nn}{\mathbf{n}}

\title[Chemotactically induced search and defense strategies] % Running head is the full title or shortened version of the full title. This will appear at the top of odd pages. Please make sure it fits within the width limit.
      {Chemotactically induced search and defense strategies in a tritrophic system} % Only the first word and proper nouns should be capitalized

\author[Nestor Anaya, Manuel Falconi and Guilmer Gonz\'alez]{}
 %reaccion difusion pde en biologia dinamica de poblaciones

 \email{nestoranaya@hotmail.com}
 \email{falconi@unam.mx}
 \email{guilmerg@ciencias.unam.mx}

\thanks{The first author is supported by NSF grant xx-xxxx}

\thanks{$^*$Corresponding author: Nestor Anaya}

\begin{document}
\maketitle

% Enter the first author's name and address:
\centerline{\scshape Nestor Anaya$^*$}
\medskip
{\footnotesize
% Enter the address of the first author
 \centerline{Departamento de Matem\'aticas}
   \centerline{Facultad de Ciencias}
   \centerline{Ciudad Universitaria, CDMX 04510 M\'exico}
} % Do not forget to end {\footnotesize with the sign }

\medskip

\centerline{\scshape Manuel Falconi and Guilmer Gonz\'alez}
\medskip
{\footnotesize
 % Enter the address of the second and third authors
 \centerline{Departamento de Matem\'aticas}
   \centerline{Facultad de Ciencias}
   \centerline{Ciudad Universitaria, CDMX 04510 M\'exico}
}

\bigskip

% The name of the associate editor will be entered by AIMS editorial staff.
% "Communicated by the associate editor name" is not needed for special issue.
 \centerline{(Communicated by the associate editor name)}

%The abstract of your paper
\begin{abstract}
In this paper we study the question of the survival of a predator which in a static scenario  vanishes. we analyze the role of migration on the coexistence of three species interacting through a intraguild relationship.
\end{abstract}

\section{Introduction}
Individual movement regulated by concentrations of chemical substances is a very frequent natural phenomenon; known as {\it Chemotaxis} is an important mechanism, for instance, of bacterial populations in search of nutrients or to establish symbiotic relationships (see \cite{jbr}). Chemical components has been observed as a defense strategy of several species. To get some insight about this process, in \cite{per} were studied the chemical defense of  two species of brown alga Dictyota menstrualis and Dictyota mertensii used against  the limited mobility herbivores, the amphipod Parhyale hawaiensis and the crab Pachygrapsus transversus.  In fact, natural defense against predation is very well documented and it is present in both invertebrate and vertebrate species,  see  \cite{dum},  \cite{eis}, \cite{fat}, \cite{mat},\cite{bru} . On the other hand, the study about the relationship of organism dispersal and community structure of interacting species has a long history. Since the works of Kolmogorov \cite{kol} and Skellam \cite{ske}, mathematical modeling of diffusion and random walk has been widely applied in the study of the effect of  individual movement on the dynamic properties of different kinds of species interaction. Among the recent works on this topic it is \cite{yan} where the authors consider a tritrophic food chain with predators and one resource; the existence and boundedness of solutions and stability of equilibrium solutions are analyzed. Stability and Turing patterns of a diffusive predator-prey model have been analyzed in \cite{dan}. Diffusion and delay effect has been incorporated in an intraguild predation model in \cite{ren}, where the authors studied how the delay on the conversion rate of mesopredator induces spatiotemporal patterns. About diffusion in predator-prey context see \cite{ezi}, \cite{ai}.
In this work we analyzed how the emission of chemical substances which attract predators of consumers of a resource impact the spatial distribution of species. A laboratory  study on this topic is \cite{kes} where Kessler and Baldwin have found that volatile emissions from {\it Nicotiana attenuata} could reduce the number of herbivores up to $90\%$.

In this work, we consider an intraguild predation model of one resource and two predators; the importance of this interaction for population ecology has been explained by Polis and Holt in \cite{Polis1992}. We consider that meso predator feed on a resource which grows acoording to a logistic growth law and it is consumed by a top predator; functional responses of meso and top predators are of Holling II type. Predators and preys difusse in a connected bounded region $\Omega \subset \mathbb{R}^2$ of the plane. We consider two cases: in the first case, the model is
\begin{eqnarray}   \label{mod1f} \label{modelo}
\frac{\partial u}{\partial t} &=&d_{0}\Delta u+\alpha u\left( 1-\frac{u}{K}%
\right) -\frac{buv}{u+a},  \nonumber \\
\frac{\partial v}{\partial t} &=&d_{1}\Delta v+\gamma \frac{buv}{u+a}-\frac{%
cvw}{v+d}-\mu v,   \\
\frac{\partial w}{\partial t} &=&d_{2}\Delta w+\beta \frac{cvw}{v+d}-\nu
w-\nabla \cdot \left( \chi _{1}\left( v,w\right) \nabla v\right), 
\nonumber
\end{eqnarray}
the random dispersal of top predators is tempered by a certain tendency  to move up the gradient of meso predators.

In the second case, as a chemotactic defense mechanism of the prey is considered, the resource population attracts top predators which feeds on mesopredator; this kind of indirect defense against predators has been reported in \cite{alj}, see also \cite{bru}; 2) top predator  in search of food moves towards areas where the mesopredator population is increasing   The model is given by    %
\begin{eqnarray} \label{mod2f}
\frac{\partial u}{\partial t} & =&  d_0\Delta u +  \alpha u  ( 1 - \frac{u }{K })- \frac{b u v }{u + a} ,\nonumber \\
\frac{\partial v}{\partial t} & = &  d_1\Delta v + \gamma \frac{b u v  }{u + a} -  \frac{c v w}{v+d} -\mu v,  \\
\frac{\partial w}{\partial t} &= &  d_2\Delta w +\beta \frac{c v w }{v + d}  -\nu w -  \div ( \chi_2(u,w) \grad u) ,\nonumber
\end{eqnarray} 
in this model the random movement is regulated by the gradient of  population density  of the resource.
The carrying capacity $K=K(x,y)$ is non-negative function defined in $\Omega$ and describes the different suitability of the niches for the resource species. Niche suitability and size population has been addressed in \cite{lui}.
It is assumed that the flux vanishes in the boundary of $\Omega$,
\begin{equation}
\frac{\partial u}{\partial \eta }\left( x,t\right) =\frac{\partial v}{%
\partial \eta }\left( x,t\right) =\frac{\partial w}{\partial \eta }\left(
x,t\right) =0 , x\in \partial \Omega t>0
\label{con frontera}
\end{equation}
where $\partial
/\partial \eta =\eta \cdot \nabla $, and $\eta $ is  the normal vector to  $\partial \Omega $.

%The carrying capacity $K=K(x,y)$ is non-negative function defined in $\Omega$ and describes the different suitability of the niches for the resource species. Niche suitability and size population has been addressed in \cite{lui}.
%It is assumed that the flux vanishes in the boundary of $\Omega$,
%\begin{equation}
%\frac{\partial u}{\partial \eta }\left( x,t\right) =\frac{\partial v}{%
%\partial \eta }\left( x,t\right) =\frac{\partial w}{\partial \eta }\left(
%x,t\right) =0 , x\in \partial \Omega \mbox{ and }t>0
%\label{con frontera}
%\end{equation}
%where $\partial
%/\partial \eta =\eta \cdot \nabla $, and $\eta $ is  the normal vector to  $\partial \Omega $.\\

The carrying capacity is denoted by $K$, $\alpha$ is the intrinsic growth of the resource $u$
; $b$ y $c$ are the mortality rate by predation of $u$ and $v$, respectively.
The conversion rate of biomass captured by $v$ and $w$ are  $\gamma $ and  $
\beta $, respectively. In Model \ref{mod1f} it is assumed that the regulating mechanism against of random dispersal of $w$ depends on a volatile substance is generated by $u$; in Model \ref{mod2f} is generated by $v$.   
Two predators which  feed on a common resource subject to a Lotka-Volterra interaction was considered in \cite{wan}; diffusive movement of predators is controlled by the prey density gradient. In \cite{tel} was analyzed a predator-prey model where predator moves toward the gradient of a chemical released by prey.\\
The underlying ordinary differential system corresponding to models \ref{mod2f} and \ref{mod1f} is given by
\begin{eqnarray}   \label{mod3f} 
u' &=&\alpha u\left( 1-\frac{u}{K}
\right) -\frac{buv}{u+a},  \nonumber \\
v'&=&\gamma \frac{buv}{u+a}-\frac{
cvw}{v+d}-\mu v,   \\
w'&=&\beta \frac{cvw}{v+d}-\nu
w. 
\nonumber
\end{eqnarray}
The system (\ref{mod3f}) has the following equilibrium points

\begin{enumerate}
\item[i)] $P_{1}\left( 0,0,0\right) $
\item [ii)] $P_{2}\left( K,0,0\right) $
\item[iii)] $P_{3}\left( \frac{a\mu }{b\gamma -\mu },\frac{a\alpha \gamma \left(
b\gamma K-\mu (a + K)\right) }{K\left( b\gamma -\mu \right) ^{2}},0\right) .$
\end{enumerate}
Under appropriate conditions, this system posses two equilibrium points $P_{4}\left( u_{1},v_{1},w_{1}\right) $ and $P_{5}\left(
u_{2},v_{2},w_{2}\right) $ with positive coordinates given by

\begin{eqnarray*}
u_{1} & = &\frac{1}{2}\left( -a+K-\sqrt{\frac{c\alpha \beta \left( a+K\right)
^{2}-\left( 4bdK+\left( a+K\right) ^{2}\alpha \right) \nu }{\left( c\beta
-\nu \right) \alpha }}\right) \\
v_{1} &=&\frac{d\nu }{c\beta -\nu } \\
w_{1}&=&\frac{\left( d+v_{1}\right) \left( b\gamma u_{1}-\left(
a+u_{1}\right) v_{1}\mu \right) }{c\left( a+u_{1}\right) } \\
u_{2}&=&\frac{1}{2}\left( -a+K+\sqrt{\frac{c\alpha \beta \left( a+K\right)
^{2}-\left( 4bdK+\left( a+K\right) ^{2}\alpha \right) \nu }{\left( c\beta
-\nu \right) \alpha }}\right) \\
v_{2}&=&\frac{d\nu }{c\beta -\nu } \\
w_{2}&=&\frac{\left( d+v_{2}\right) \left( b\gamma u_{2}-\left(
a+u_{2}\right) v_{2}\mu \right) }{c\left( a+u_{2}\right) }.
\end{eqnarray*}

The local dynamics around the equilibrium points of this system is depicted in Appendix A.
 
\section{Existence of positive solution}
In this section we provide conditions for the existence of positive solutions of systems \ref{mod1f} and \ref{mod2f} for the initial conditions

\begin{equation}
t=0 \text{: } u=u_{0}\left( x\right) \text{, } v=v_{0}\left( x\right) \text{, }%
w=w_{0}\left( x\right) \text{, }x\in \Omega  \label{condiciones iniciales}
\end{equation}
 and the boundary conditions given by (\ref{con frontera}).
Let $p>n\geq 1$; then $W^{1,p}\left( \Omega ,%
%TCIMACRO{\U{211d} }%
%BeginExpansion
\mathbb{R}
%EndExpansion
^{n}\right) $ is continuously embedded in the continuous function space $%
C\left( \Omega ;%
%TCIMACRO{\U{211d} }%
%BeginExpansion
\mathbb{R}
%EndExpansion
^{n}\right) $. Let
\[
X:=\{y\in W^{1,p}\left( \Omega ,%
%TCIMACRO{\U{211d} }%
%BeginExpansion
\mathbb{R}
%EndExpansion
^{3}\right) |\eta \cdot \nabla y_{|_{\partial \Omega }}=0\}.
\]%

\begin{theorem} \label{Teo1}
If $\left( u_{0},v_{0},w_{0}\right) \in X $, then

\begin{itemize}
\item There exists $T=T_{\max }\in \lbrack 0,\infty )$, which depends on the 
initial conditions (\ref{condiciones iniciales}) such that the problem (\ref{mod1f}),(\ref{con frontera})-(\ref{condiciones iniciales}) has a unique maximal solution $\left( u,\text{ }v,\text{ 
}w\right) $ on $\Omega \times \lbrack 0,T_{\max })$ and  $\left(
u\left( \cdot ,t\right) ,\text{ }v\left( \cdot ,t\right) ,\text{ }w\left(
\cdot ,t\right) \right) \in C\left( \left( 0,\text{ }T_{\max }\right)
,\Omega \right) $, $\left( u,v,w\right) \in C^{2,1}\left( \left( 0,T_{\max
}\right) \times \overline{\Omega },%
%TCIMACRO{\U{211d} }%
%BeginExpansion
\mathbb{R}
%EndExpansion
^{3}\right) ;$

\item If $u_{0},$ $v_{0},$ $w_{0}\geq 0$ on $\overline{\Omega }$,
then $u,$ $v,$ $w\geq 0$ on $\Omega \times \lbrack 0,T_{\max })$;

\item If $\left\Vert \left( u,w,w\right) \left( \cdot ,t\right) \right\Vert
_{L^{\infty }\left( \Omega \right) }$ is bounded for all $t\in
\lbrack 0,T_{\max })$, then $T_{\max }=+\infty $; equivalently,  $\left(
u,v,w\right) $ is a global solution.
\end{itemize}
\end{theorem}

\begin{proof}
Let $z=\left( u,v,w\right) ^{3}.$ Then, (\ref{modelo}),( \ref{con frontera}) and (\ref{condiciones iniciales}) can be written as 
\begin{eqnarray}
z_{t} &=&\nabla \cdot \left( A\left( z\right) \nabla z\right) +F\left(
z\right) \text{ on }\Omega \times \lbrack 0,\infty )
\nonumber \\
B_{z} &=&\frac{\partial }{\partial \nu }z=0\text{ on }\partial \Omega
\times \lbrack 0,\infty )   \label{modelo reescrito} \\
z\left( \cdot ,0\right) &=&\left( u_{0},v_{0},w_{0}\right) \text{ en }\Omega 
,  \nonumber
\end{eqnarray}%
where
\[
A\left[ z\right] =\left[ 
\begin{array}{ccc}
d_{0} & 0 & 0 \\ 
0 & d_{1} & 0 \\ 
0 & -\chi _{1} & d_{2}%
\end{array}%
\right] 
\]%
and
\[
F\left( z\right) =\left[ 
\begin{array}{c}
u\left( \alpha \left( 1-\frac{u}{K}\right) -\frac{bv}{u+a}\right) \\ 
v\left( \gamma \frac{bu}{u+a}-\frac{cw}{v+d}-\mu \right) \\ 
w\left( \beta \frac{cv}{v+d}-\nu \right)%
\end{array}%
\right] 
\]

The result follows from  \cite{Haskell}.
\end{proof}

According to the above theorem, to prove the existence of global solutions it is necessary to show that $u$, $v$ and $w$
are uniformly bounded in $L^{\infty }\left( \Omega \right) $.

\begin{theorem}
If $\left( u_{0},v_{0},w_{0}\right) \in X $, then the solutions of the system (\ref{modelo}), (\ref{con frontera}) and (\ref{condiciones iniciales}) are bounded.
\end{theorem}

\begin{proof}
Let $W\left( x,t\right) =u+\frac{1}{\gamma }v+\frac{1}{\gamma \beta }w$, so
\begin{eqnarray*}
\frac{d}{dt}\int_{\Omega }\left( W\left( x,t\right) \right) &=&\int_{\Omega
}\left( d_{0}\Delta u+\alpha u\left( 1-\frac{u}{K}\right) -\frac{buv}{u+a}%
\right) dx+\int_{\Omega }\left( \frac{1}{\gamma }\left( d_{1}\Delta v+\gamma 
\frac{buv}{u+a}-\frac{cvw}{v+d}-\mu v\right) \right) dx \\
&&+\int_{\Omega }\left( \frac{1}{\gamma \beta }\left( d_{2}\Delta w+\beta 
\frac{cvw}{v+d}-vw-\nabla \cdot \left( \chi_1 \left( v,w\right) \nabla
v\right) \right) \right) dx \\
&=&\int_{\Omega }\left( d_{0}\Delta u+\frac{1}{\gamma }d_{1}\Delta v+\frac{1%
}{\gamma \beta }d_{2}\Delta w\right) dx \\
&&+\int_{\Omega }\left( \alpha u\left( 1-\frac{u}{K}\right) -\frac{buv}{u+a}+%
\frac{buv}{u+a}-\frac{c}{\gamma }\frac{vw}{v+d}-\frac{\mu }{\gamma }v+\frac{c%
}{\gamma }\frac{cvw}{v+d}-\frac{\nu }{\gamma \beta }w\right) dx \\
&\leq &\int_{\Omega }\left( \alpha u\left( 1-\frac{u}{K}\right) -\frac{\mu }{%
\gamma }v-\frac{\nu }{\gamma \beta }\right) dx
\end{eqnarray*}%
It follows that
\begin{equation}
\frac{d}{dt}\int_{\Omega }Wdx+\int_{\Omega }\left( \frac{\mu }{\gamma }v+%
\frac{\nu }{\gamma \beta }w\right) dx\leq \int_{\Omega }\alpha u\left( 1-%
\frac{u}{K}\right) dx. \label{des1}
\end{equation}

On the other hand,  let $\mu _{0}=\min \{\mu ,\nu \}$ that implies
\begin{equation}
\frac{d}{dt}\int_{\Omega }Wdx+\mu _{0}\int_{\Omega }\left( \frac{1}{\gamma }%
v+\frac{1}{\gamma \beta }w\right) dx\leq \frac{d}{dt}\int_{\Omega
}Wdx+\int_{\Omega }\left( \frac{\mu }{\gamma }v+\frac{\nu }{\gamma \beta }%
w\right) dx.  \label{des2}
\end{equation}%
From (\ref{des1}) and (\ref{des2}), we obtain that 
\begin{equation}
\frac{d}{dt}\int_{\Omega }Wdx+\mu _{0}\int_{\Omega }\left( u+\frac{1}{\gamma 
}v+\frac{1}{\gamma \beta }w\right) dx\leq \int_{\Omega }\left( \alpha
u\left( 1-\frac{u}{K}\right) +\mu _{0}u\right) dx  \label{des3}
\end{equation}

Note that 
\begin{eqnarray}
\int_{\Omega }\left( \left( \alpha +\mu _{0}\right) u-\frac{\alpha u^{2}}{K}%
\right) \; dx &\leq &\int_{\Omega }\frac{1}{4}\frac{K\left( \alpha +\mu
_{0}\right) ^{2}}{\alpha } \; dx\label{des4} \\
&=&\frac{1}{4}\frac{K\left( \alpha +\mu _{0}\right) ^{2}}{\alpha }|\Omega | 
\nonumber
\end{eqnarray}%
Now, let $K_{0}=\frac{1}{4}\frac{K\left( \alpha +\mu _{0}\right) ^{2}}{\alpha 
}|\Omega |$, then from (\ref{des3}) and (\ref{des4}) we have that 
\[
\frac{d}{dt}\int_{\Omega }Wdx+\mu _{0}\int_{\Omega }\left( u+\frac{1}{\gamma 
}v+\frac{1}{\gamma \beta }w\right) dx\leq K_{0} 
\]%
from this, is clearly evident that
\[
\int_{\Omega }\left( u+\frac{1}{\gamma }v+\frac{1}{\gamma \beta }w\right)
dx\leq K_{0}+ce^{-t} 
\]
It follows that solutions are bounded, since $u, v, w$ are nonegative.
\end{proof}

The proof of the following theorem is similar to those of Theorem \ref{Teo1}.
\begin{theorem}
Let $\left( u_{0},v_{0},w_{0}\right) \in X$.

\begin{itemize} \label{Teo3}
\item There exists $T=T_{\max }\in \lbrack 0,\infty )$,  which depends on the 
initial conditions (\ref{condiciones iniciales}) such that the problem (\ref{mod2f}),(\ref{con frontera}) and (\ref{condiciones iniciales}) has a unique maximal solution $\left( u,\text{ }v,\text{ 
}w\right) $ on $\Omega \times \lbrack 0,T_{\max })$ and  $\left(
u\left( \cdot ,t\right) ,\text{ }v\left( \cdot ,t\right) ,\text{ }w\left(
\cdot ,t\right) \right) \in C\left( \left( 0,\text{ }T_{\max }\right)
,\Omega \right) $, $\left( u,v,w\right) \in C^{2,1}\left( \left( 0,T_{\max
}\right) \times \overline{\Omega },%
%TCIMACRO{\U{211d} }%
%BeginExpansion
\mathbb{R}
%EndExpansion
^{3}\right) ;$ 

\item If $u_{0},$ $v_{0},$ $w_{0}\geq 0$ on $\overline{\Omega }$,
then $u,$ $v,$ $w\geq 0$ on $\Omega \times \lbrack 0,T_{\max })$;

\item If $\left\Vert \left( u,w,w\right) \left( \cdot ,t\right) \right\Vert
_{L^{\infty }\left( \Omega \right) }$ is bounded for all  $t\in
\lbrack 0,T_{\max })$, then $T_{\max }=+\infty $; i.e., $\left(
u,v,w\right) $ is a globally bounded solution.
\end{itemize}
\end{theorem}

Note that $v$ and $w$ vanish if $\gamma b\leq \mu $
and $\beta c\leq \nu $, respectively. From now on, we assume that $\gamma b>\mu $ $\ $ and $\beta c>\nu $.

Let $Y=\left\{ U=\left( u,v,w\right) \in \left[ C^{1}\left( \overline{\Omega 
}\right) \right] ^{3}|\partial _{\nu }u(x)=0,\mbox{ x}\in \partial \Omega
\right\} $, and let $\left\{ \phi _{i,j}\mbox{, }j=1,2,...,\dim \left( E\left( \mu
_{i}\right) \right) \right\} $ be a orthonormal basis of $E\left( \mu
_{i}\right) $ and $Y_{ij}=\left\{ C\cdot \phi _{ij}|C\in 
%TCIMACRO{\U{211d} }%
%BeginExpansion
\mathbb{R}
%EndExpansion
^{3}\right\} $. Then,  $Y_{i}=\oplus _{j=1}^{\dim \left( E\left( \mu
_{i}\right) \right) }Y_{ij},$ $Y=\oplus _{i=1}^{\infty }Y_{i}.$

\begin{theorem}
If $bK\gamma -a\mu -K\mu <0$ then the equilibrium point $P_{2}$ $\ $ of system \ref{mod1f} is locally stable.
\end{theorem}

\begin{proof}
Let $A\left[ z\right] =\left( 
\begin{array}{ccc}
d_{0} & 0 & 0 \\ 
0 & d_{1} & 0 \\ 
0 & -\chi _{1} & d_{2}%
\end{array}%
\right) $ as in theorem \ref{Teo1} and $L=A\left[ z\right] \Delta
+J_{1}$ where $J_{1}$ is the Jacobian matrix of the system without diffusion evaluated at
 $P_{2}$; i.e. 
\[
J_{1}=\left( 
\begin{array}{ccc}
-\alpha  & -\frac{bK}{a+K} & 0 \\ 
0 & \frac{bK\gamma }{a+K}-\mu  & 0 \\ 
0 & 0 & -\nu 
\end{array}%
\right) .
\]%
The linearization of the system at $P_{2}$ is $U_{t}=LU$ . $Y_{i}$ is invariant with respect to operator $L$ for all $i\geq 1$;  $\lambda $ is an eigenvalue of $L$
restricted to $Y_{i}$ if and only if is an eigenvalue of matrix $-\mu _{i}A%
\left[ z\right] \Delta +J_{1}$.

The characteristic polynomial of  $\mu _{i}A\left[ z\right] \Delta +J_{1}$
is
\[
\varphi _{i}\left( \lambda \right) =\left( \lambda +\mu _{i}d_{1}+\alpha
\right) \left( \lambda +\mu _{i}d_{2}-\frac{bK\gamma }{a+K}+\mu \right)
\left( \lambda +\mu _{i}d_{2}+\nu \right) 
\]%
whose roots are $\varphi _{i}\left( \lambda \right) $
 $-\mu _{i}d_{1}-\alpha $, $-\mu _{i}d_{2}+\frac{bK\gamma }{a+K}-\mu $ and 
$-\mu _{i}d_{2}-\nu $. Therefore, the point-spectrum of $L$ consists of eigenvalues that satisfy
 $\{\mbox{Re}\, {\lambda} \leq -\left( 1/2\right) \max \{\alpha
,-\frac{bK\gamma }{a+K}+\mu ,\nu \}\}$ whenever $bK\gamma -a\mu -K\mu <0$; from which stability around  
 $P_{2}$ follows, [\cite{Henry},Th. 5.1.1 ].
\end{proof}
The following section describes the spatial discretization that we apply to perform some numerical simulations of the previous models.

%%%%%%%%%%%%%%%%%%%%%

\section{Spatial discretization }

\subsection{Variational formulation }
We consider a general reaction-diffusion problem with Neumann boundary conditions

\begin{eqnarray}
-\Delta u + \mu u & = & f \qquad \mbox{en }\Omega \label{fv1} \\
u(x,0) & = & u_0(x) \qquad \mbox{en }\Omega  \label{fv2}\\
\partial_nu(x,t) & = & 0 \qquad \mbox{en }\partial\Omega  \label{fv3}
\end{eqnarray}
where function $f \in C^0(\Omega)$ is regular, $\mu \in \mathbb{R}$.  As it is usual  $\partial_n u =\nabla u \cdot \nn$,  where $\nn$ is the exterior normal vector to $\partial \Omega$. 

A classic solution of the above problem  (\ref{fv1})--(\ref{fv3}) is a function $u:\bar{\Omega}\mapsto \mathbb{R}$, $u\in C^2(\bar{\Omega})$ which satisfies (\ref{fv1})--(\ref{fv3}). In order to facilitate the search of $u$ we reformulate the problem to find a equivalent solution.

Let $v\in X:= C^2(\bar{\Omega})$.  Multiplying  (\ref{fv1}) by $v$ it is obtained

$$
-v\Delta u + \mu u v = f v
$$
Integrating on  $\Omega$

\begin{equation}
-\int_\Omega v\Delta u \; d\Omega + \mu \int_\Omega uv\; d\Omega  = \int_\Omega f v\; d\Omega 
\end{equation}

Applying the Green Theorem
\begin{equation}
\int_\Omega \nabla u\cdot \nabla v \; d\Omega - \int_{\partial\Omega} (\nabla u\cdot \nn) v \; dS + \mu \int_\Omega uv\; d\Omega  = \int_\Omega f v\; d\Omega.
\end{equation}
Since $\partial_n u =0$ for $x\in \partial\Omega$, we have

\begin{equation}
\int_\Omega \nabla u\cdot \nabla v \; d\Omega + \mu \int_\Omega uv\; d\Omega =  \int_\Omega f v\; d\Omega. \label{variacional}
\end{equation}

This expression is known as variational formulation of the problem (\ref{fv1})--(\ref{fv3}),  see \cite{Thomee2006}.  Notice that in (\ref{variacional}) it is only required that $u,v\in C^1(\bar{\Omega})$.  Furthermore, they can even be just continuous. 

\subsection{Discretization Finite Element Method  }
Let $H^k(\Omega)$ a Sobolev space and  $C^1(0,T,X)$ is  the space of continuously differentiable functions from  $[0, T]$ on $X$.  $\Omega_h$ is a polygonal approximation of $\Omega$.  We consider a mesh $T_h$ of $\Omega_h$ consisting of convex elements  $E_i\in T_h$,  $i\in I$ , $I\subset \mathbb{N}$.

Let $\{\varphi_j(x,y)\}_{1\leq j\leq N}$ be a base of $V_h$
\begin{eqnarray*}
u_h(x,y,t) & = & \sum_{j=1}^N u_i(t) \varphi_j(x,y) \\
v_h(x,y,t) & = &  \sum_{j=1}^N v_i(t) \varphi_j(x,y) \\
w_h(x,y,t) & = &   \sum_{j=1}^N w_j(t) \varphi_j(x,y) 
\end{eqnarray*}
$x, y\in \Omega$, $0\leq t\leq T$.  The basis $ \varphi_j(x,y)$ are compact support functions and we use the usual  linear elements $P1$ defined on triangles.

Parameter $h$ represents the size of element  $E_i$ of mesh $T_h$ and  is defined as
$$
h = \max_{E_i \in T_h} \mbox{diam}(E_i),
$$
as $h\mapsto 0$, space $V_h$ is closer to  $H^k(\Omega)$.

\section{Semi-discretization of time}

Let 
$$0=t_0<t_1< \cdots t_N = T,$$
a partition of the interval $[0,  T]$ with constant step $dt= t_{m+1}-t_m$ for all $m\in\{0,\ldots,N-1\}$.  The derivative with respect to time is approximated using forward finite differences 

$$
u_t =  \frac{\umm - \um}{dt}, \quad
v_t  =  \frac{\vmm - \vm}{dt}, \quad
w_t  =  \frac{\wmm - \wm}{dt}
$$
where  $u^m=u(x,  t_m), v^m=v(x, t_m), w^m=w(x ,t_m)$. 

By substituting the above approximation in model  (\ref{modelo}) we obtain that 

%Para cada paso del tiempo $m\ge 0$, usando las anteriores aproximaciones y sustituyendo en el modelo (\ref{modelo}), calculamos $\umm$, $\vmm$ y $\wmm$ de manera que 

 \begin{eqnarray}  \label{mod34wDd} 
\umm   & =&   \um + dt\cdot d_0\Delta \umm + dt\cdot \alpha \umm  ( 1 - \frac{\umm }{K(x,y) })- dt\cdot \frac{b \umm \vmm }{\umm + a} ,\nonumber \\
\vmm  & = &  \vm + dt\cdot d_1\Delta \vmm +dt\cdot  \gamma \frac{b \umm \vmm  }{\umm + a} - dt\cdot \frac{c \vmm \wmm}{\vmm+d} -\mu \vmm   \\
\wmm &= &  \wm + dt\cdot d_2\Delta \wmm + dt\cdot \beta \frac{c \vmm \wmm }{\vmm + d}  \nonumber  \\
           &   &-dt\cdot \nu \wmm -  dt\cdot\div ( \chi_2(\vmm, \wmm) \grad \vmm). \nonumber
\end{eqnarray} 
This is the Implicit Euler Method which depends on  both  $(x,y)\in \Omega$ for each element $E_i$ and the boundary conditions 
%Este es el m\'etodo Euler impl\'\i cito que s\'olo depende de $(x,y)\in \Omega$ para cada $K_i$ de la partici\'on y las condiciones de frontera para cada inc\'ognita

\begin{equation}
\grad \umm \cdot \nn=0,\, \grad \vmm \cdot \nn=0,\, \grad \wmm \cdot \nn=0, \quad m\ge 0.
\label{condd}
\end{equation} 

From the initial values  $u_0, v_0$,  and  $w_0$,  we compute the next iterations $(u_1, v_1, w_1), \ldots, (u_N,v_N, w_N)$.  The system (\ref{mod34wDd})  is solved by FEM,  assuming that $u_0, v_0, w_0\in C^2(\bar{\Omega})$,  see \cite{Douglas1970}.  To avoid some complications which arise from the nonlinearity involved in (\ref{mod34wDd}),  the terms corresponding to temporal variation are solved using a semi-implicit Runge-Kutta method of second order. The two steps of this computational process are depicted in the following.  First,  the right side of equations (\ref{modelo}) are rewritten as 

\begin{eqnarray}
F(u,v,w) & =&   d_0\Delta u + \alpha u  ( 1 - \frac{u }{K(x,y) })- \frac{b u v }{u + a} ,\nonumber \\
G(u,v,w) & = & d_1\Delta v + \gamma \frac{b u v  }{u + a} - \frac{c v w}{v+d} -\mu v,  \\
H(u,v,w) &= &  d_2\Delta w +\beta \frac{c v w }{v + d}  -\nu w -  \div ( \chi_2(v,w) \grad v) .\nonumber
\end{eqnarray} 
The first step of the RK--method of second order consists in an one Euler step computed at central point of each time interval. 

\begin{eqnarray}
\ummd & = & \um + \frac{dt}{2}\cdot  F(\um,\vm,\wm) \\
\vmmd & = & \vm + \frac{dt}{2}\cdot  G(\um,\vm,\wm) \\
\wmmd & = & \wm + \frac{dt}{2}\cdot  H(\um,\vm,\wm) 
\end{eqnarray}

In the second step,  computations are made at time $m+1$ like

\begin{eqnarray}
\umm & = & \um + dt\cdot F(\ummd,\vmmd,\wmmd) \\
\vmm & = & \vm + dt\cdot  G(\ummd,\vmmd,\wmmd) \\
\wmm & = & \wm +dt\cdot  H(\ummd,\vmmd,\wmmd) 
\end{eqnarray}
Now we considered  the diffusion in an implicit form,  then the schema becomes a semi-implicit one.  For each step,  the equations are solved by applying the FEM Galerkin-Ritz method described above.

\section{Numerical simulations} 
In this section, some numerical simulations are carried out in order to obtain some knowledge about the effect on the population density of the indirect defense mechanism of the resource against the meso--predator, which consists on the attraction of the main predator towards the resource. This will be contrasted with the results of the corresponding simulations of model \ref{modelo}, in which the random diffusion of the main predator is regulated by a tendency to move towards the gradient of the meso--predator; this is the case of predators actively searching for prey,  see \cite{chr}, \cite{cod} and the references cited there. In all this section,  we assume that  $\alpha = 5,\, a=2.0, \, b=5, \, c=.01,\, d=2.0,\, \beta =1.0,\, \gamma=1,\, \mu=.05,\, \nu =.05$. For this parameter values, the equilibrium points of  system (4) are $P_{0}=(0,0,0),$ $P_{1}=(K,0,0),$ $P_{2}=(\frac 2 
{99},\frac{200 (99 K-2)}{9801 K},0),$ $P_{3}=(k-2,2,\frac{2 (99 K-200)}{K})$. Existence and stablity properties of these equilibrium points are described in Table  \ref{tb1}.
\begin{table}[htp] 
\caption{}
\label{tb1}
\begin{center}
\begin{tabular}{|cccc|}
Point &  Existence Interval & Stable  & Unstable\\
\hline
$P_{0}$ & $K > 0 $  &  & $K>0$ \\
& & & \\
$P_{1}$ & $K>0$ & $K < \frac{2}{99}$ &$K > \frac{2}{99}$ \\
& & & \\
$P_{2}$ & $K>\frac{2}{99}$ &  $\frac{2}{99}< K < \frac{200}{99}$& $K> \frac{200}{99}$\\
& & & \\
$P_3$ & $K > \frac{200}{99}$ & $\frac{200}{99} <  K< 2.02063$ & 
\end{tabular}
\end{center}
\label{tabla1}
\end{table}
For the numerical computations we assume that $\Omega=[-1,1]\times[-1,1]$ and we have used the {\tt FreeFem++} software  \cite{Hecht2012}.

\subsection{Model \ref{modelo}: active-search hunting.}
In the following we consider Model  (\ref{modelo}) where the top predator is an active-search hunter.   We take 
$$\chi_1(v,w)=e_1 w-e_2v.$$
Therefore the top predator move towards the gradient of mesopredator only if its population density is large enough compared to that of the mesopredator. The ratio $\frac{e_2}{e_1}$ measures the defensive capacity of the mesopredator in terms of its population size; the larger this ratio, the greater the density of the predator required to advance towards the prey.
The parameter values are given by $\alpha=5, \, a= 2.0,\, b=5.0,\, c=0.1,\, d=2.0,\, \beta=1.0,\, \gamma=1.0,\, \mu=0.05,\, \nu=0.05,$  $d_0=0.1$, $d_1=1, d_2=1$. \\
Initial conditions for the spatial distribution of the resource, the meso-predator and top predator are considered as 
\begin{eqnarray*}
u_0(x,y) &= & 2\exp(-10(x^2+ (y-.9)^2))(1-x^2)^2(1 - y^2)^2; \\
v_0(x,y) &= & 2\exp(-(x + .9)^2 - (y + .9)^2)(1 - x^2)^2(1 - y^2)^2;\\
w_0(x,y) &= & 1.5
\end{eqnarray*}
for all $x,y\in \Omega$. 
In contrast with the meso predator and the resource, the top predator is initially uniformly distributed, 
(see Fig \ref{figCondInit1}).
\begin{figure}[hbt] 
\centerline{
\begin{tabular}{ccc}
\includegraphics[scale=0.15]{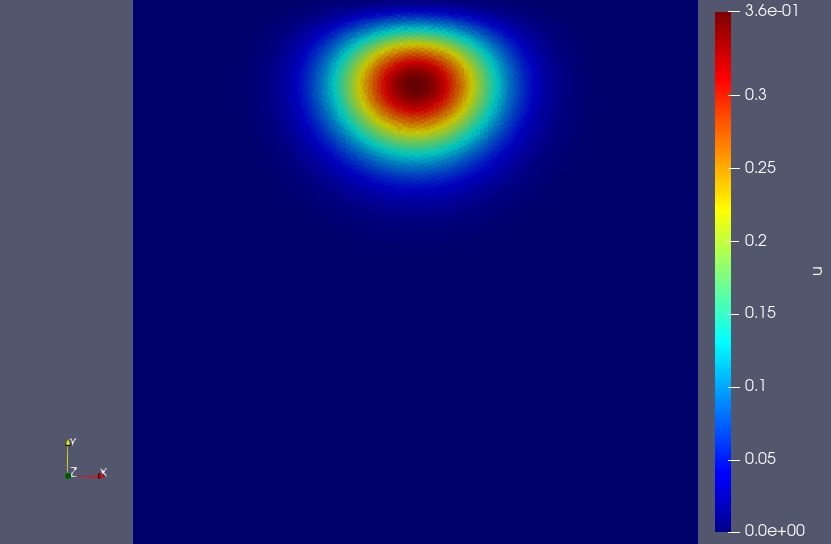} &
\includegraphics[scale=0.15]{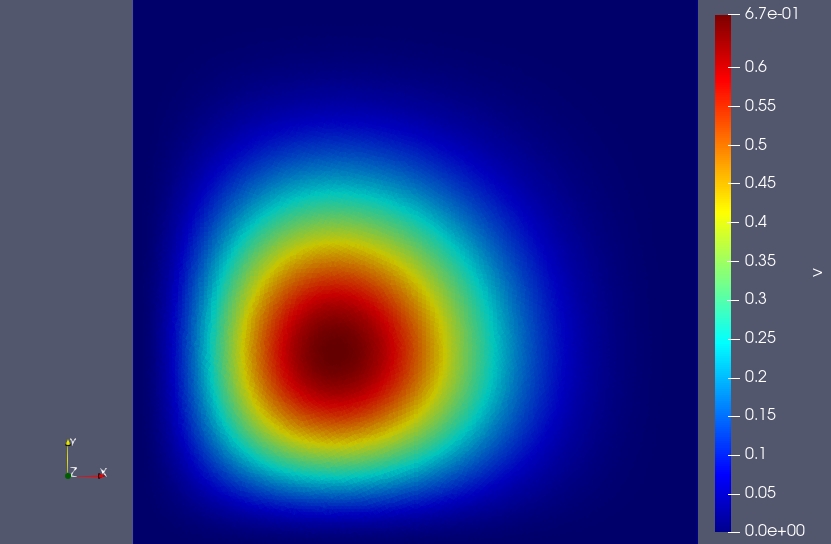} &
\includegraphics[scale=0.15]{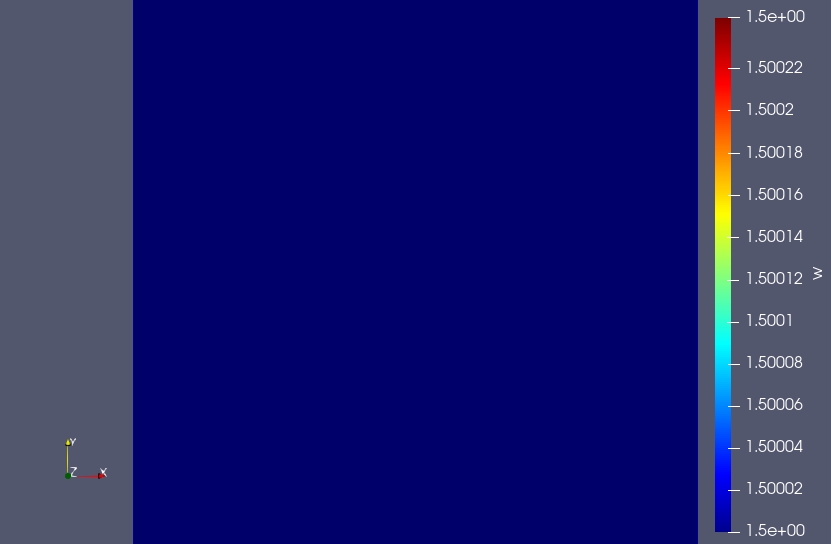} \\
($a_1$) $u,\quad t=0$ & ($b_1$) $v,\quad t=0$ & ($c_1$) $w,\quad t=0$ 
\end{tabular}}
\caption{{\em Contour plots} of time evolution of the resource $u$,  mesopredador $v$ and top predador $w$ at different times.} 
\label{figCondInit1}
\end{figure}

\subsubsection{Defensive capacity and species distribution}
We consider five different defensive capacities of the prey.
The suitability of the habitat of the resource is given by
\begin{eqnarray*}
K(x,y) &=&  2\exp(-5((x+.75)^2+(y-.75)^2))+2\exp(-5((x-.75)^2+(y+.75)^2)),   \\
          &  &  +2 \exp(-5((x+.75)^2+(y+.75)^2))+2\exp(-5((x-.75)^2+(y-.75)^2)).
\end{eqnarray*}
Notice that the range of $K$ in $\Omega$ is contained in the interval $(\frac{2}{99},\frac{200}{99})$. Therefore, according to Table \ref{tb1} system \ref{mod3f}  without diffusion does not have the coexistence point $P_{3}$ and  point $P_{2}$ is asymptotically stable. Thus, without  diffusion the  top predator $w$ would become extinct.

First, let  $ e_1=1.0,\, e_2=1.0.$ In this case, the  defensive capacity of the prey is neutral. Top predator move towards mesopredator whenever its density be greater than the one of the mesopredator

\begin{figure}[hbt]
\begin{tabular}{ccc}

\includegraphics[scale=0.125]{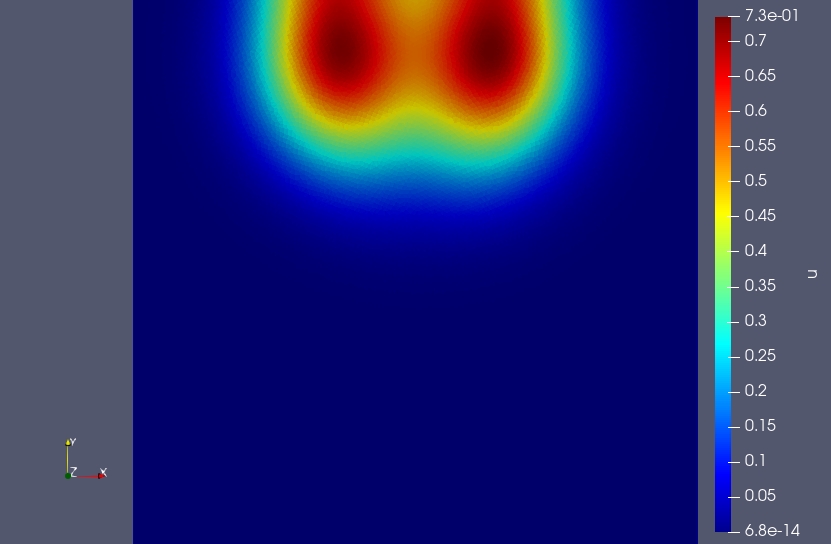} &
\includegraphics[scale=0.125]{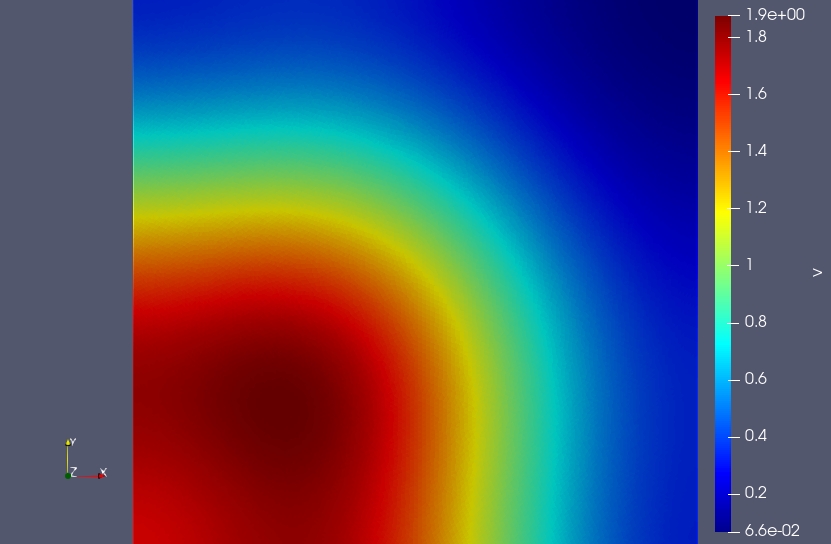} &
\includegraphics[scale=0.125]{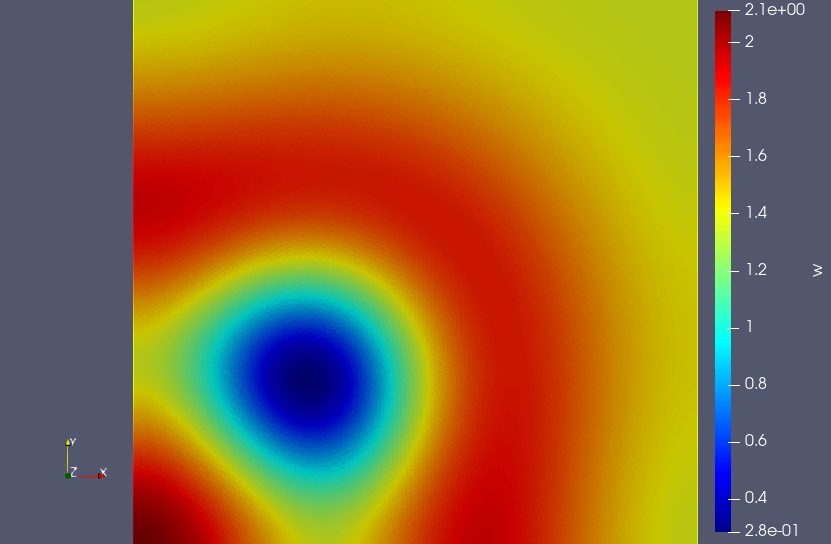} \\
($a_2$) $u,\quad t=0.1$ & ($b_2$) $v,\quad t=0.1$ & ($c_2$) $w,\quad t=0.1$ \\  

\includegraphics[scale=0.125]{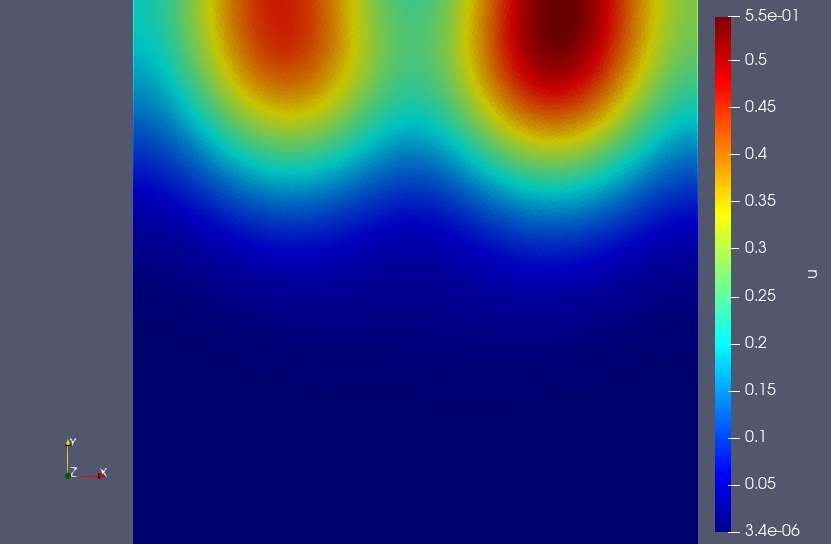} &
\includegraphics[scale=0.125]{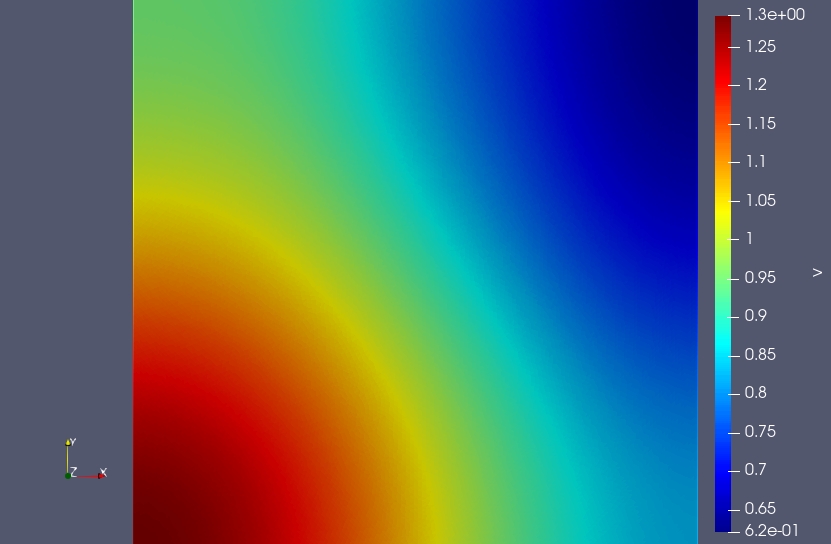} &
\includegraphics[scale=0.125]{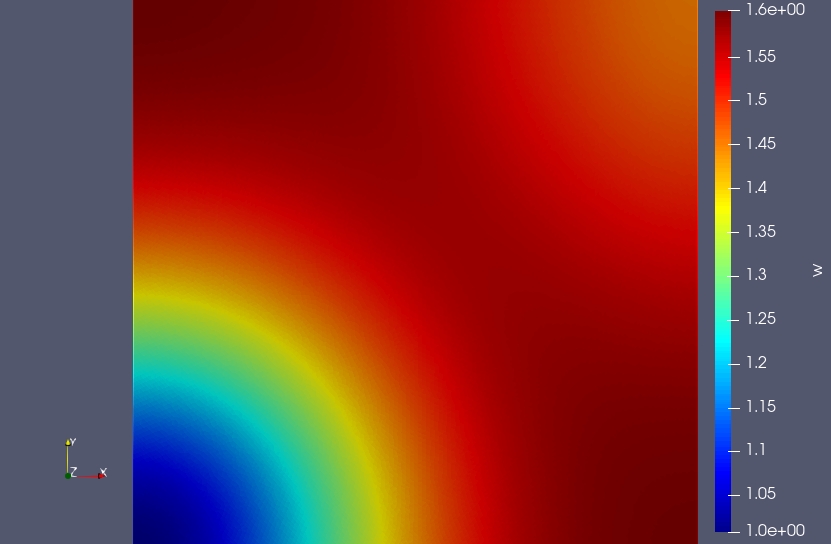} \\
($a_3$) $u,\quad t=0.5$ & ($b_3$) $v,\quad t=0.5$ & ($c_3$) $w,\quad t=0.5$ \\  
%\end{tabular}
%\caption{{\em Contour plots} of time evolution of the resource $u$,  mesopredador $v$ and top predador $w$ at different times.} 
%\end{figure}
%\newpage
%\begin{figure}[hbt]
%\begin{tabular}{ccc}
\includegraphics[scale=0.125]{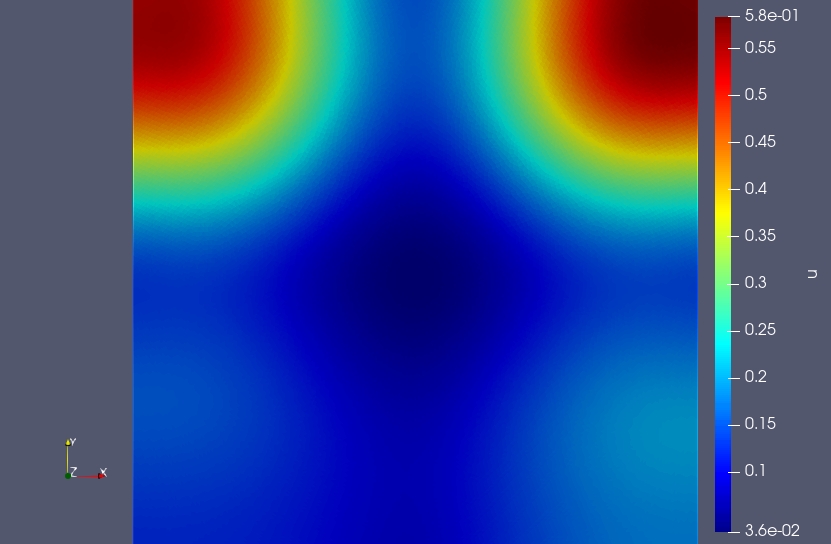} &
\includegraphics[scale=0.125]{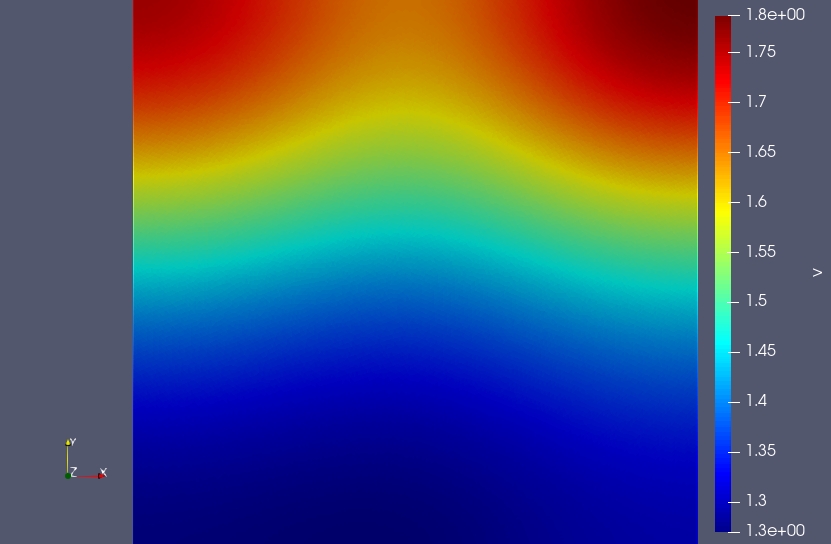} & 
\includegraphics[scale=0.125]{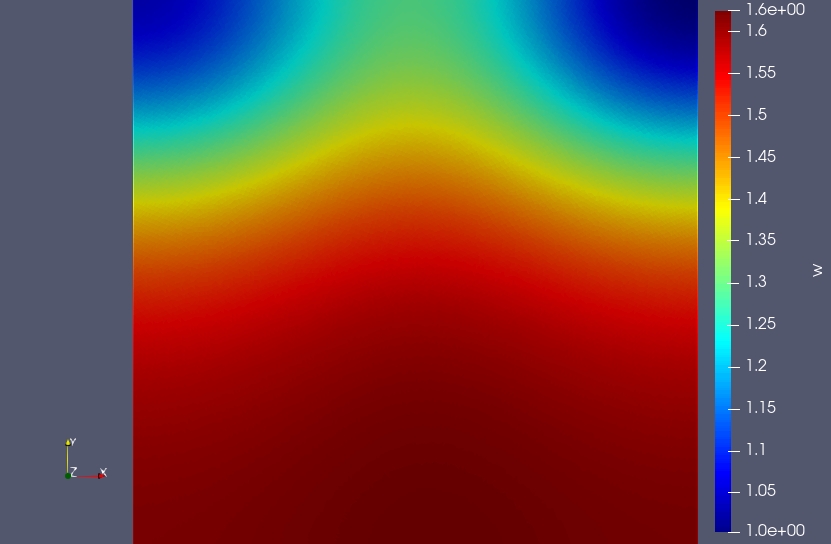} \\
($a_4$) $u,\quad t=2.0$ & ($b_4$) $v,\quad t=2.0$ & ($c_4$) $w,\quad t=2.0$ \\
%\end{tabular}
%\caption{{\em Contour plots} of time evolution of the resource $u$,  mesopredador $v$ and top predador $w$ at different times.} 
%\end{figure}

%\begin{figure}[hbt] 
%\begin{tabular}{ccc}
\includegraphics[scale=0.125]{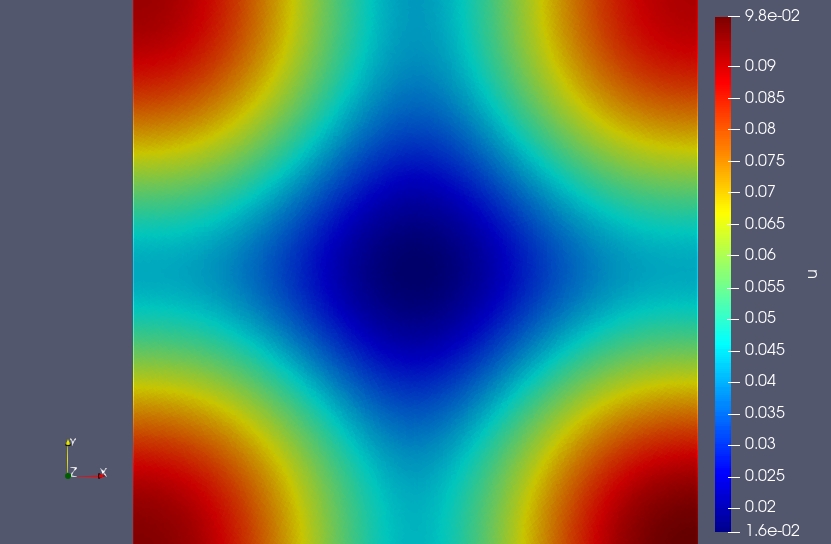} &
\includegraphics[scale=0.125]{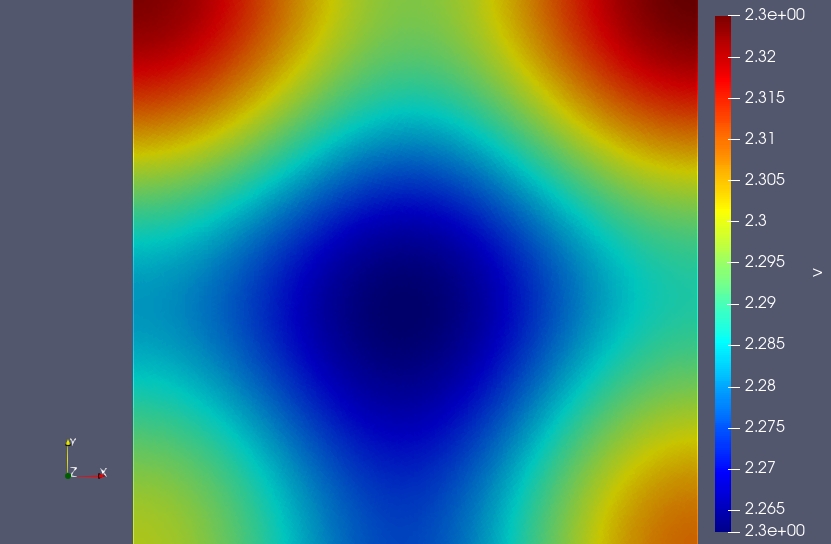} & 
\includegraphics[scale=0.125]{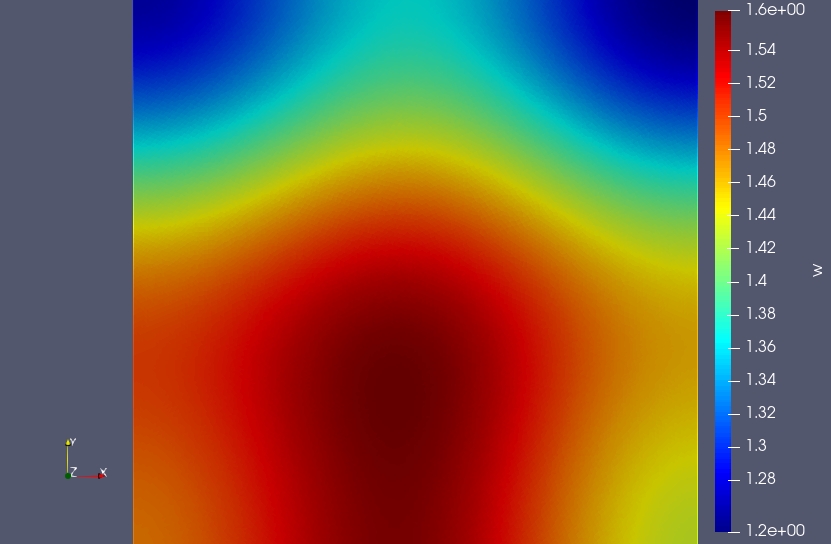} \\
($a_{5}$) $u,\quad t=4.0$ & ($b_{5}$) $v,\quad t=4.0$ & ($c_{5}$) $w,\quad t=4.0$ \\
\includegraphics[scale=0.125]{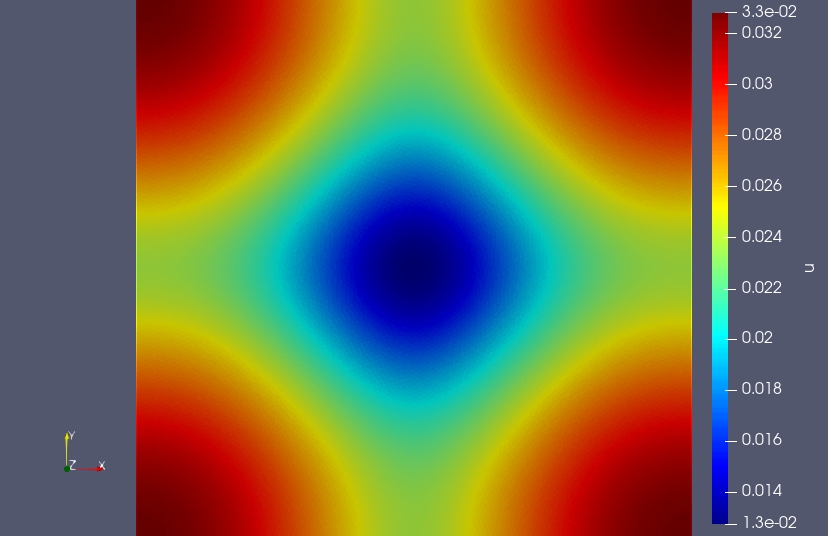} &
\includegraphics[scale=0.125]{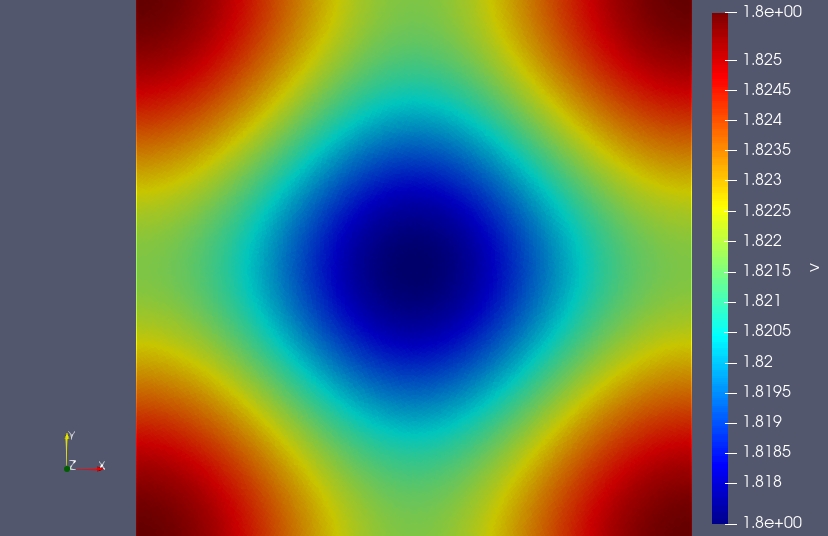} & 
\includegraphics[scale=0.125]{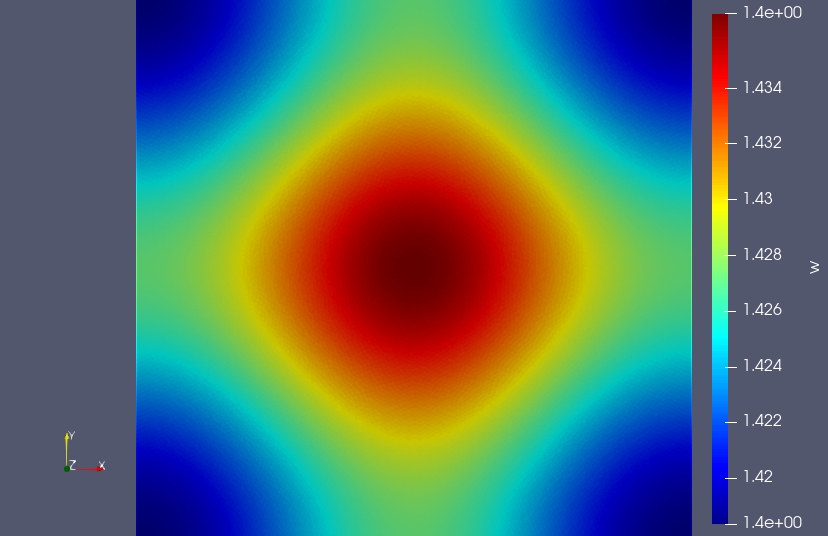} \\
($a_{6}$) $u,\quad t=20.0$ & ($b_{6}$) $v,\quad t=20.0$ & ($c_{6}$) $w,\quad t=20.0$ 

\end{tabular}
\caption{Evolution of the spatial distribution of the three species. $e_1=1.0, e_2=1.0$}
\label{ff2}
\end{figure}

\clearpage
Second, let  $ e_1=1.0, e_2=0.5$ In this case, the mesopredator defense against of top predator is lesser than the above case. Thus, we observe that predators are closer to the mesopredators than in the first case (see figures \ref{ff2} and \ref{Figu3b}). 
\begin{figure}[hbt] 
\begin{tabular}{ccc}
\includegraphics[scale=0.125]{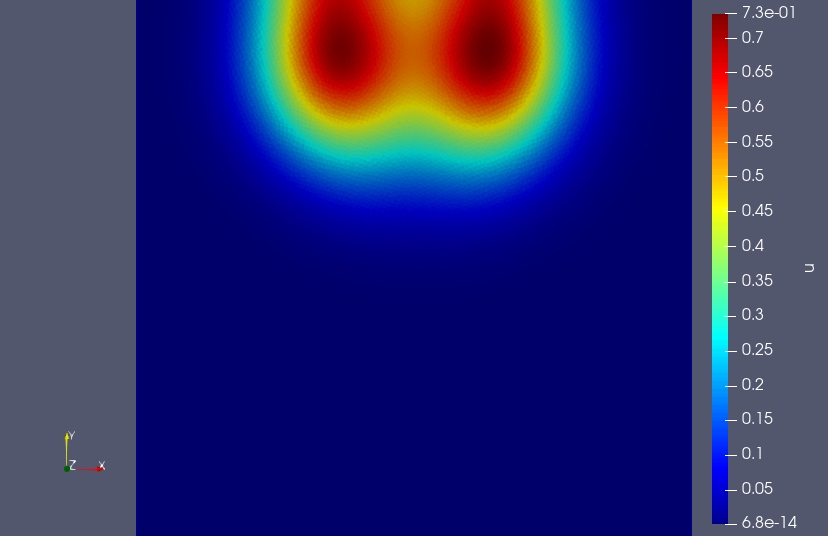} &
\includegraphics[scale=0.125]{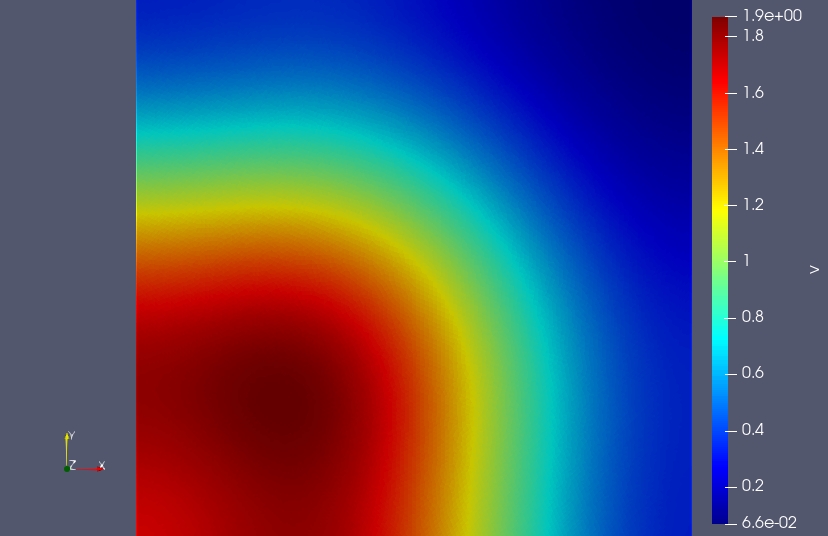} &
\includegraphics[scale=0.125]{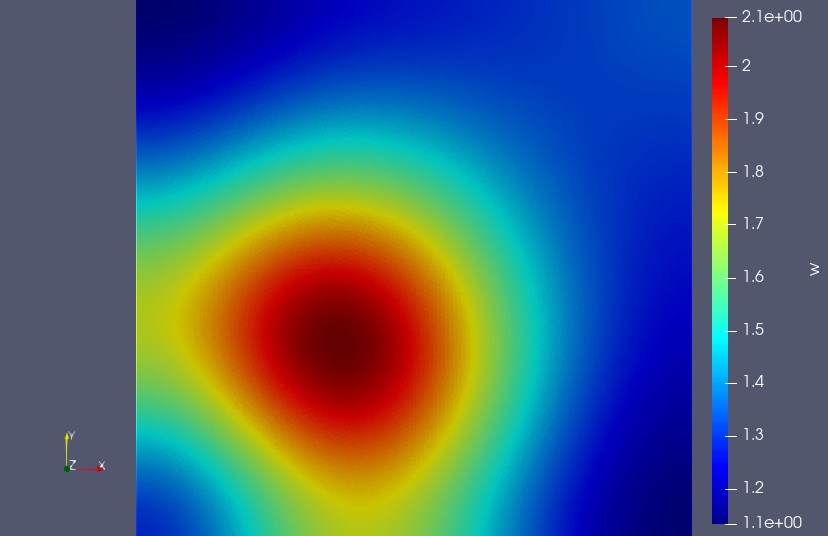} \\
($a_1$) $u,\quad t=0.1$ &  $v,\quad t=0.1$ &  $w,\quad t=0.1$  \\
%\end{tabular}
%\caption{ $e_1=1.0, e_2=0.5$}  \label{Figu1b}
%\end{figure}
%\begin{figure}[hbt]
%\begin{tabular}{ccc}
\includegraphics[scale=0.125]{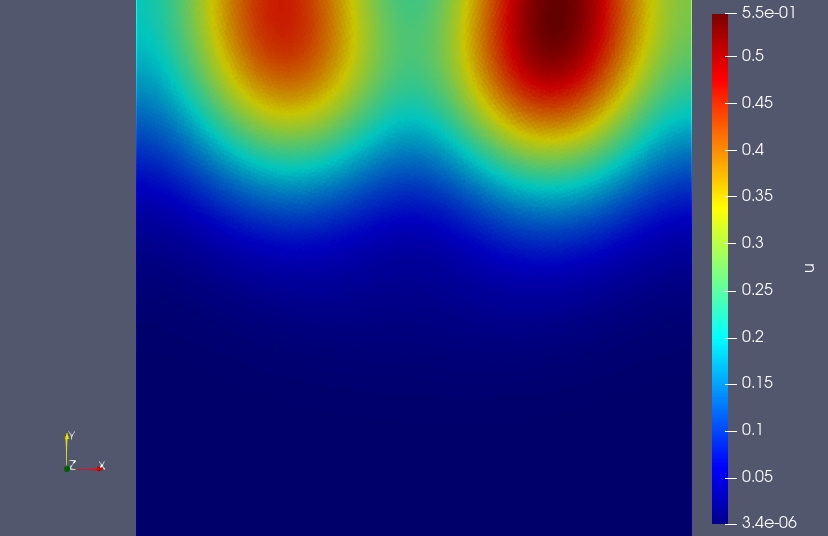} &
\includegraphics[scale=0.125]{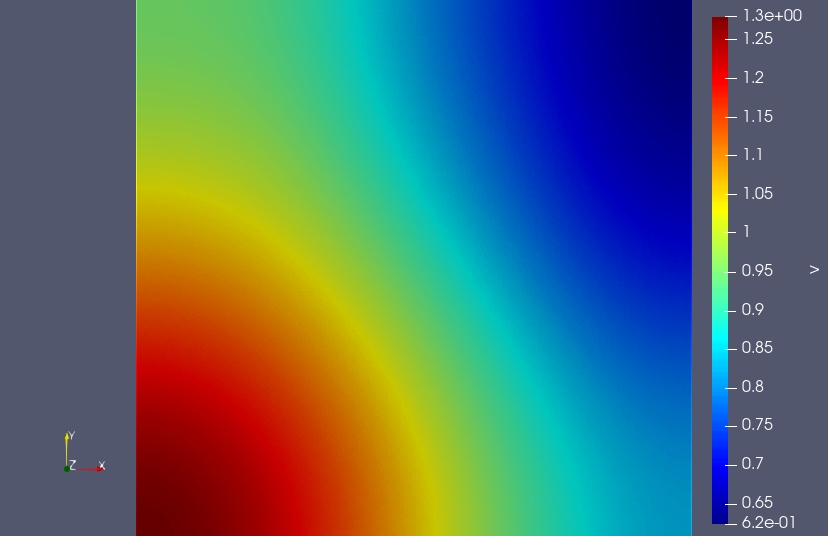} & 
\includegraphics[scale=0.125]{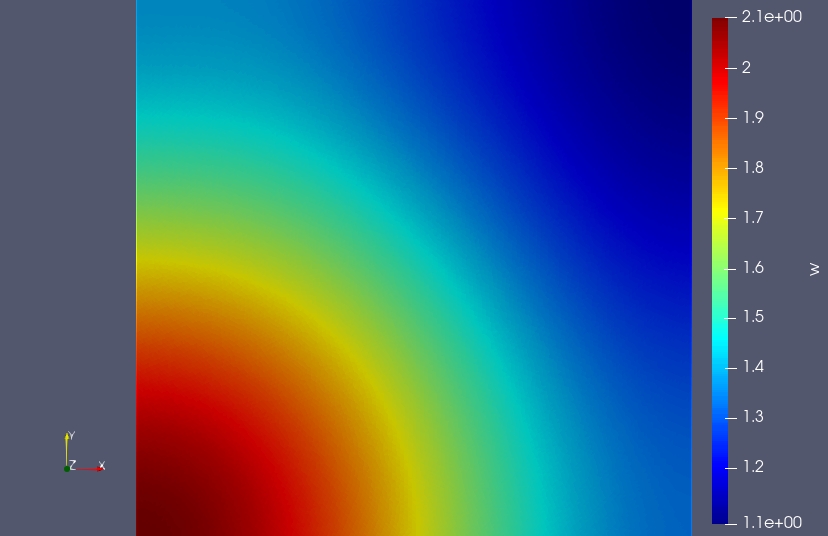} \\
($a_2$) $u,\quad t=0.5$ &  $v,\quad t=0.5$ &  $w,\quad t=0.5$  \\
%\end{tabular}
\includegraphics[scale=0.125]{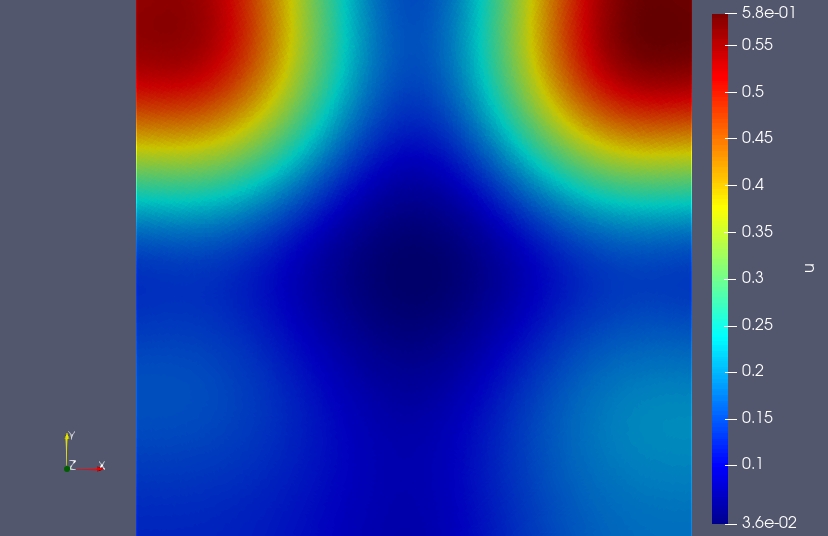} &
\includegraphics[scale=0.125]{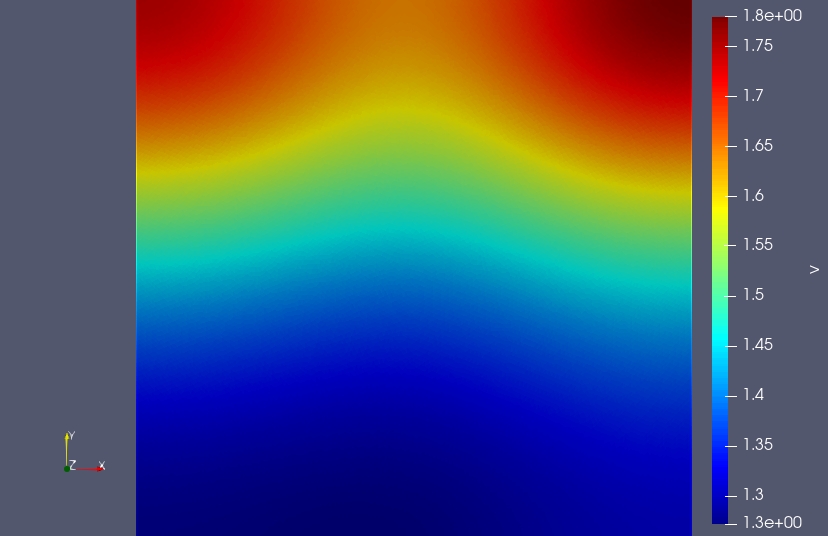} & 
\includegraphics[scale=0.125]{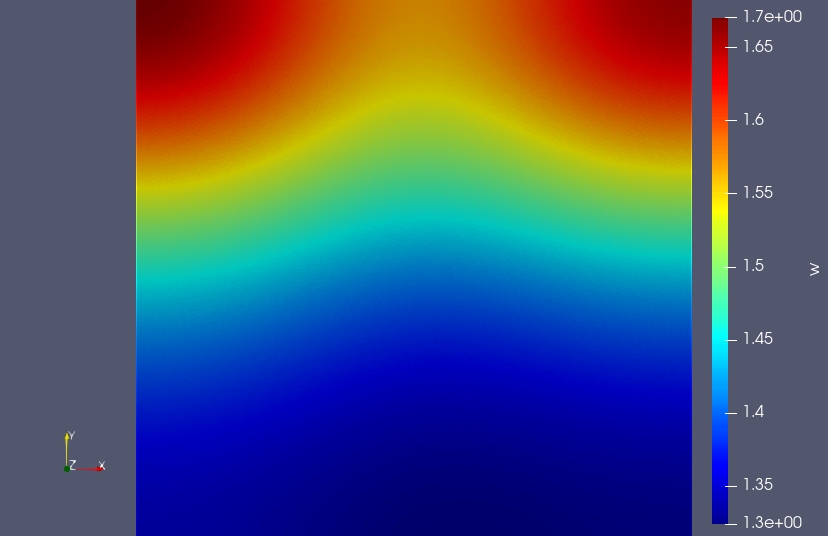} \\
($a_3$) $u,\quad t=2.0$ &  $v,\quad t=2.0$ &  $w,\quad t=2.0$  \\
%\caption{Evolution of the spatial distribution of the three species.} \label{Figu2b}
%\end{figure}
%\begin{figure}[hbt]
%\begin{tabular}{ccc}
%\includegraphics[scale=0.075]{./figuras/34MV_Caso1brp5/vt2.jpeg} &
 %\includegraphics[scale=0.075]{./figuras/34MV_Caso1brp5/vt2.jpeg} & 
%\includegraphics[scale=0.075]{./figuras/34MV_Caso1brp5/wt2.jpeg} \\
%($a_3$) $u,\quad t=2.0$ &  $v,\quad t=2.0$ &  $w,\quad t=2.0$ \\
%\end{tabular}
%\caption{Evolution of the spatial distribution of the three species.} \label{Figu22}
%\end{figure}
%\begin{figure}[hbt] 
%\begin{tabular}{ccc}
\includegraphics[scale=0.125]{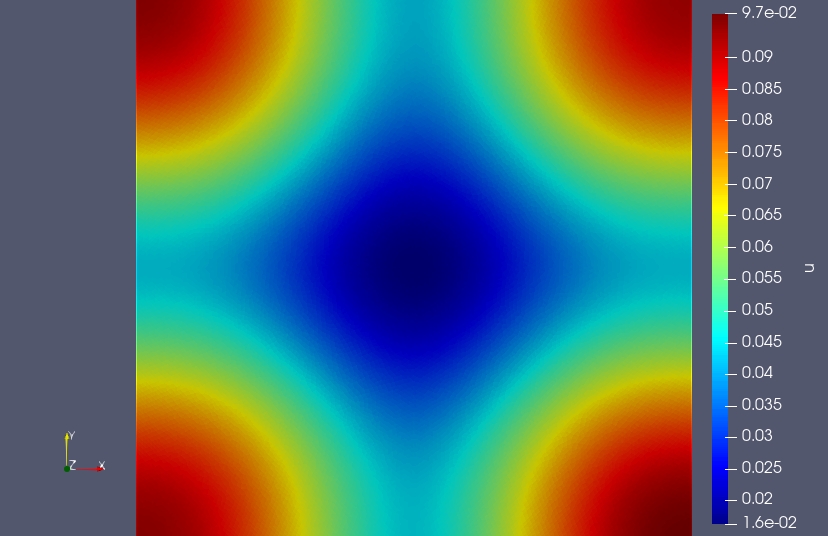} &
\includegraphics[scale=0.125]{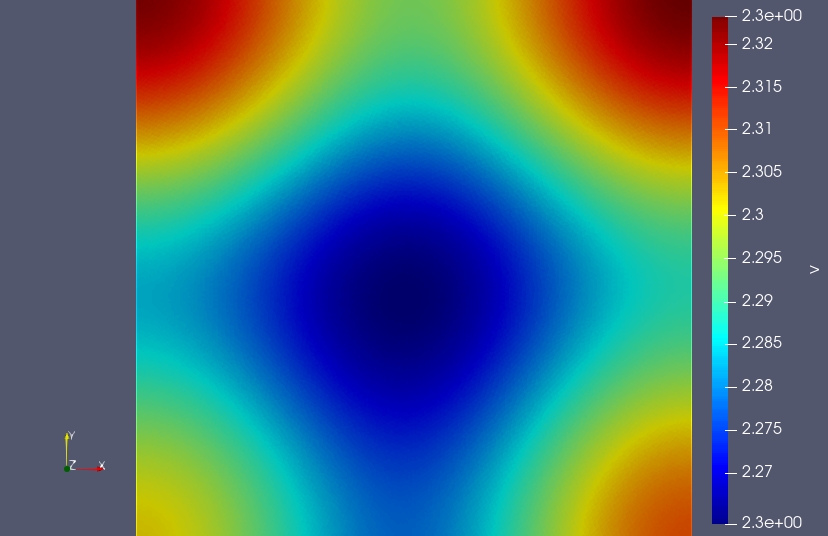} & 
\includegraphics[scale=0.125]{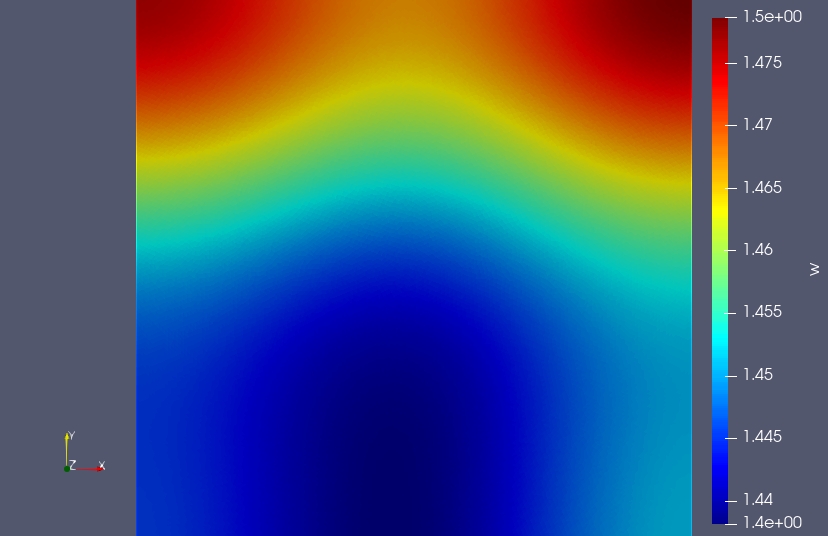} \\
($a_{4}$) $u,\quad t=4.0$ & $v,\quad t=4.0$ &  $w,\quad t=4.0$  \\

\includegraphics[scale=0.125]{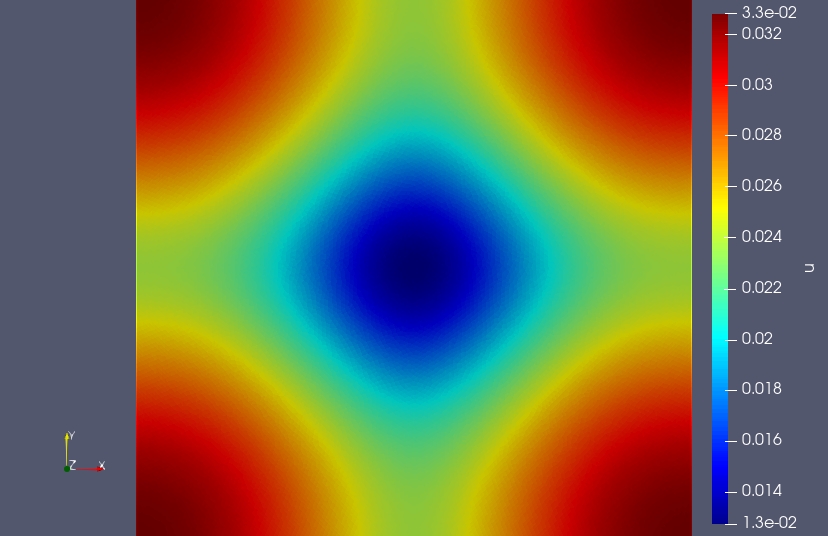} &
\includegraphics[scale=0.125]{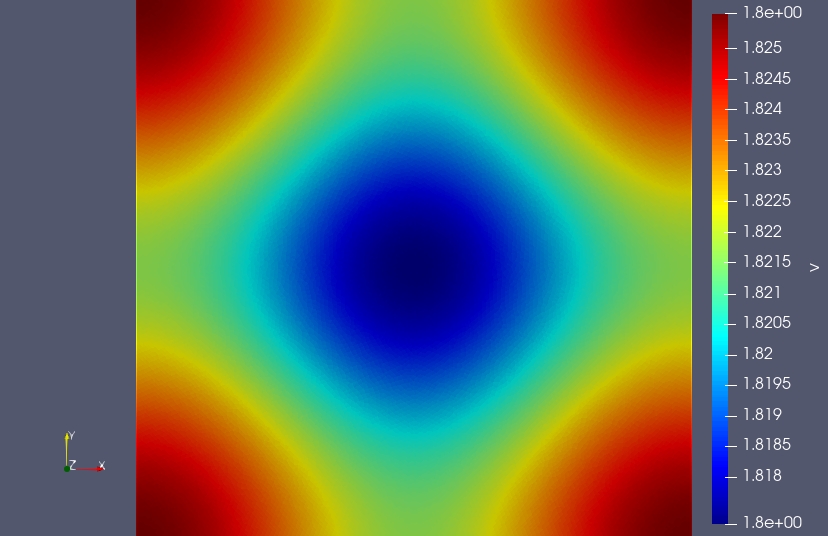} & 
\includegraphics[scale=0.125]{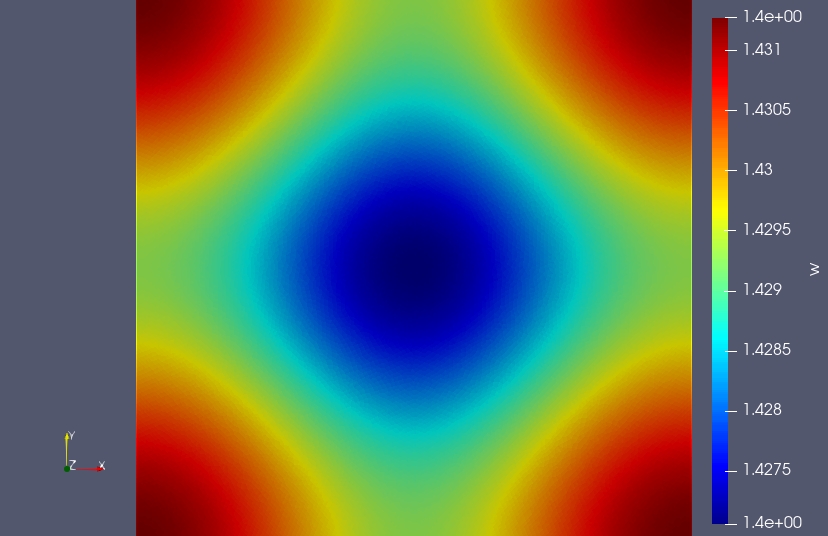} \\
($a_{5}$) $u,\quad t=20.0$ & $v,\quad t=20.0$ &  $w,\quad t=20.0$ 
\end{tabular}
\caption{Evolution of the spatial distribution of the three species. $e_1=1.0, e_2=0.5$}
 \label{Figu3b}
\end{figure}

%\newpage
\clearpage
Third, let  $ e_1=1.0,\, e_2=2.0$. Prey presents a strong defense capacity. Notice that predators tends to move towards the lower density areas of the prey population, (see Figures (\ref{ff2}) and (\ref{fig12/3})).

\begin{figure}[hbt]
\begin{tabular}{ccc}
\includegraphics[scale=0.125]{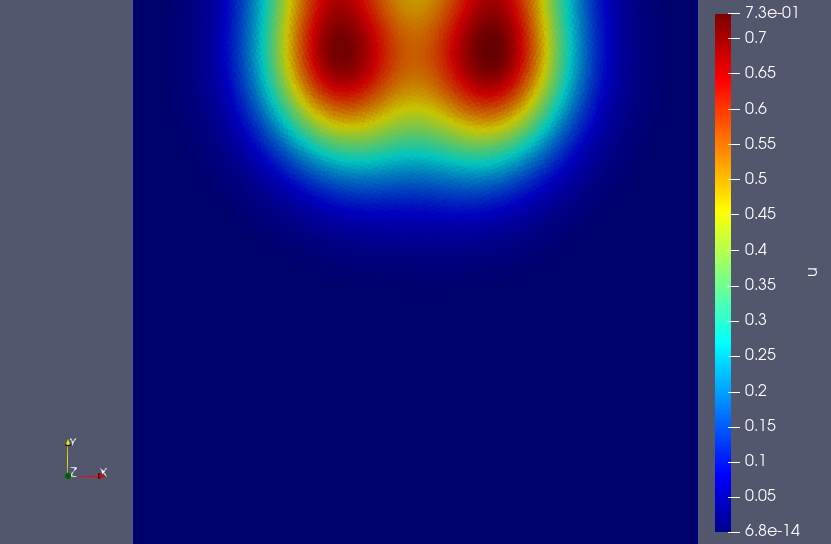} &
\includegraphics[scale=0.125]{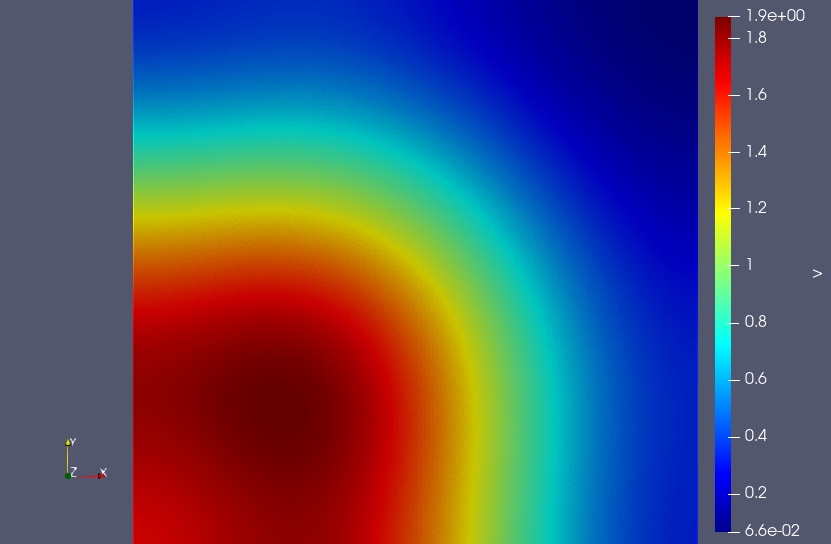} &
\includegraphics[scale=0.125]{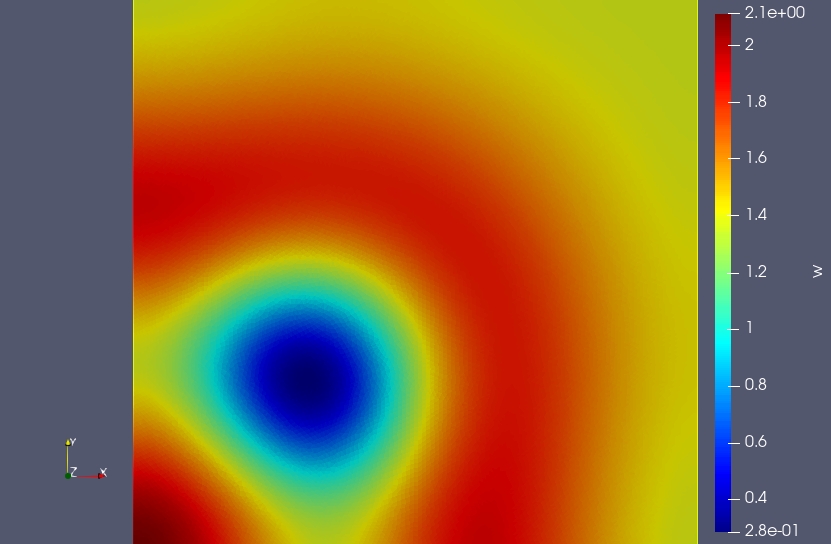} \\
($a_2$) $u,\quad t=0.1$ & ($b_2$) $v,\quad t=0.1$ & ($c_2$) $w,\quad t=0.1$  \\
%\end{tabular}
%\caption{{\em Contour plots} of time evolution of the resource $u$,  mesopredador $v$ and top predador $w$ at different times.} 
%\label{fig12/1}
%\end{figure}

%\begin{figure}[hbt]
%\begin{tabular}{ccc}
\includegraphics[scale=0.125]{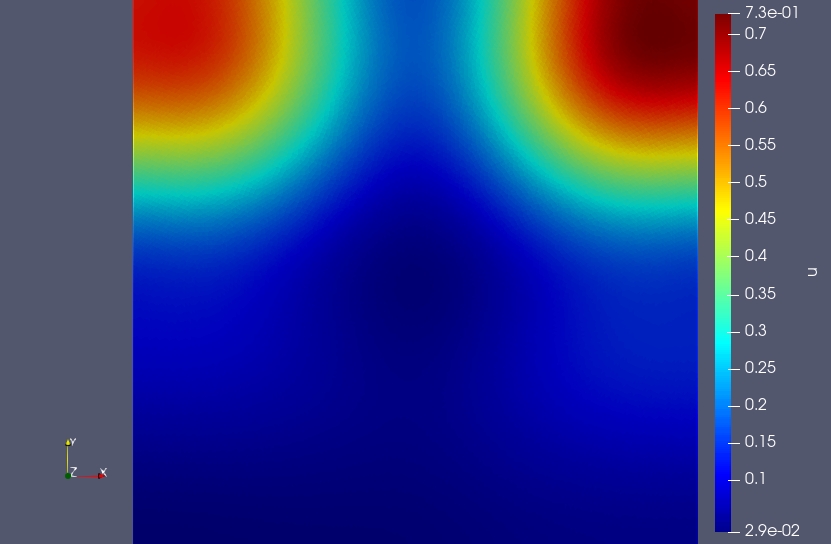} &
\includegraphics[scale=0.125]{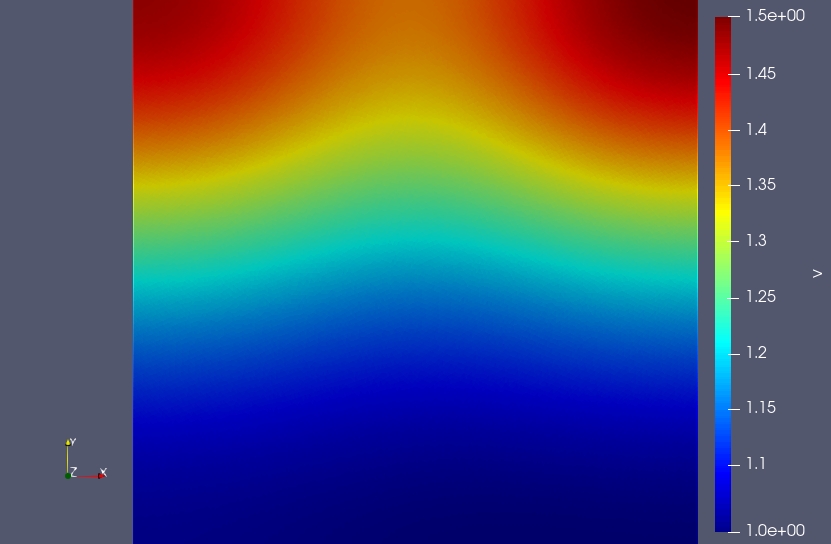} & 
\includegraphics[scale=0.125]{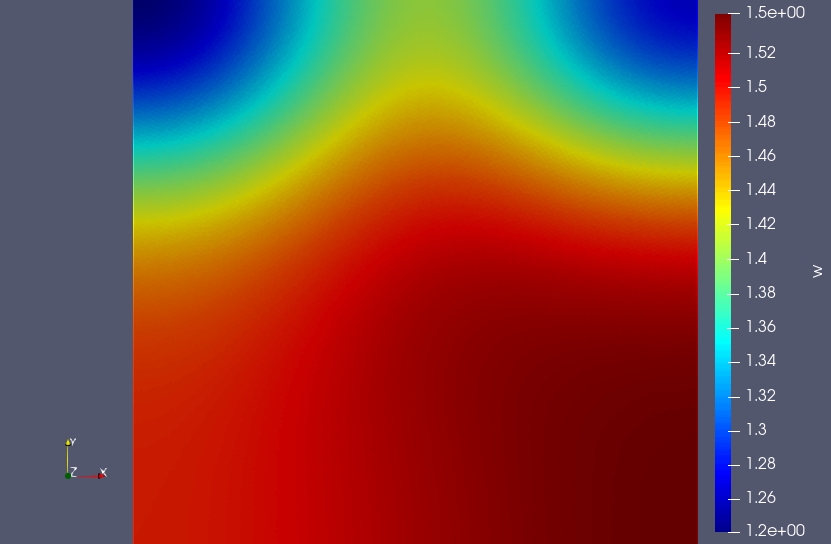} \\
($a_3$) $u,\quad t=1.5$ & ($b_3$) $v,\quad t=1.5$ & ($c_3$) $w,\quad t=1.5$  \\
%\end{tabular}
%\caption{{\em Contour plots} of time evolution of the resource $u$,  mesopredador $v$ and top predador $w$ at different times.} 
%\label{fig12/2}
%\end{figure}

%\begin{figure}[hbt]
%\begin{tabular}{ccc}
\includegraphics[scale=0.125]{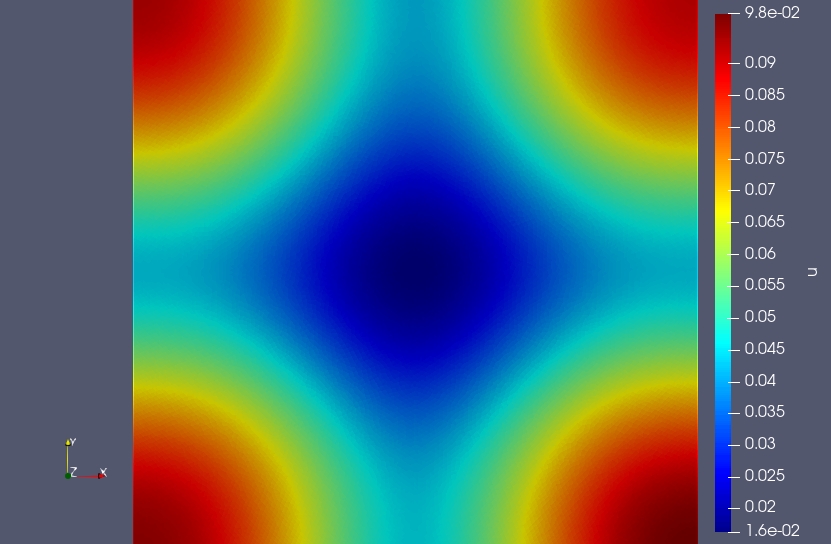} &
\includegraphics[scale=0.125]{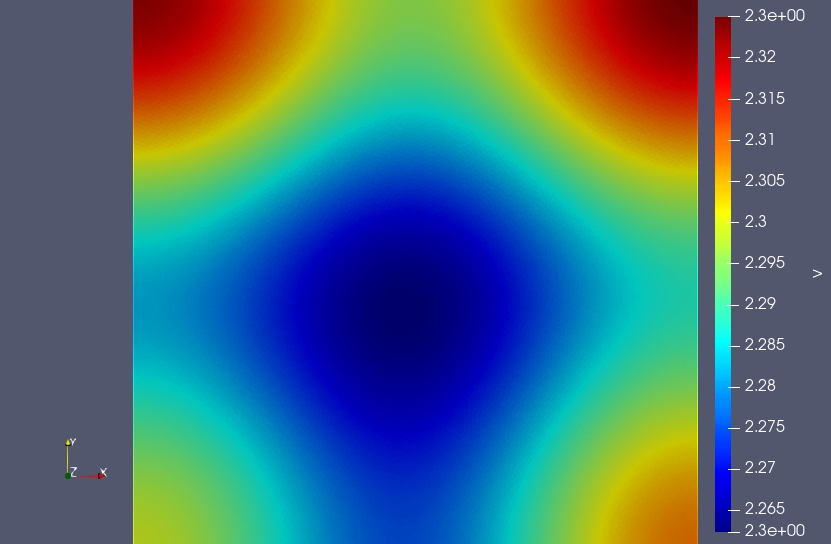} & 
\includegraphics[scale=0.125]{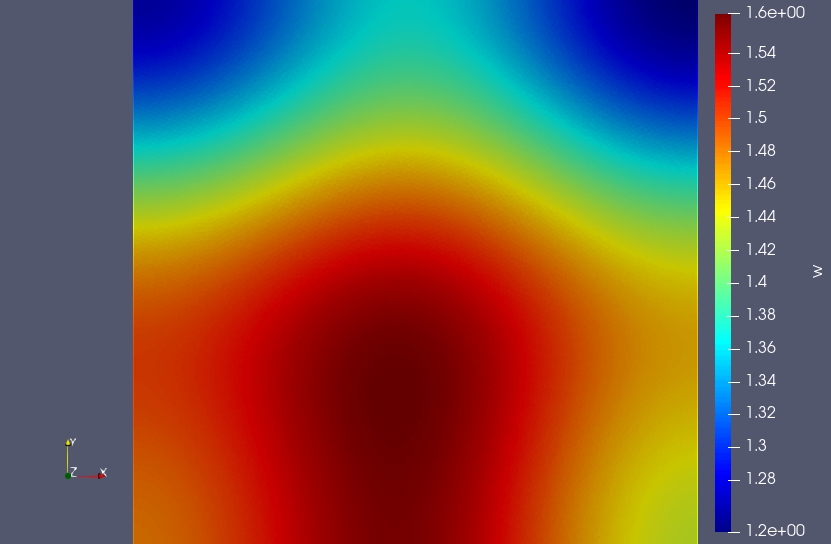} \\
($a_{4}$) $u,\quad t=4.0$ & ($b_{4}$) $v,\quad t=4.0$ & ($c_{4}$) $w,\quad t=4.0$ \\
\includegraphics[scale=0.125]{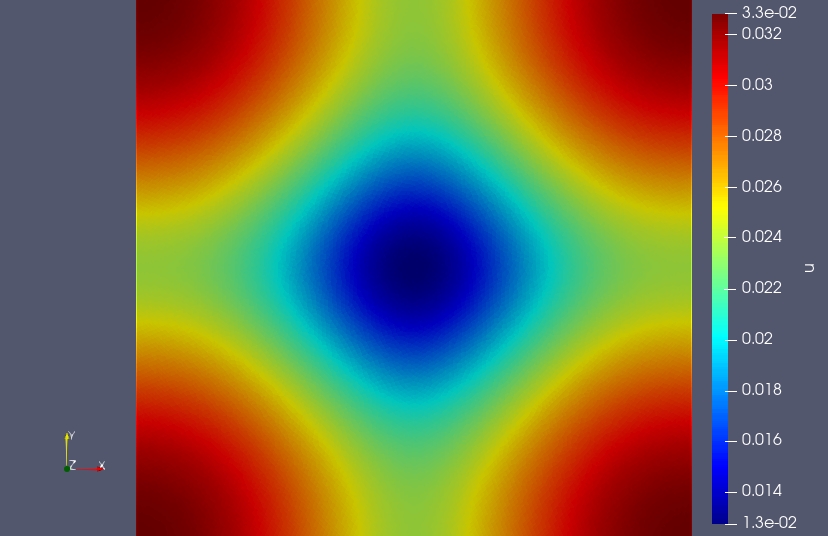} &
\includegraphics[scale=0.125]{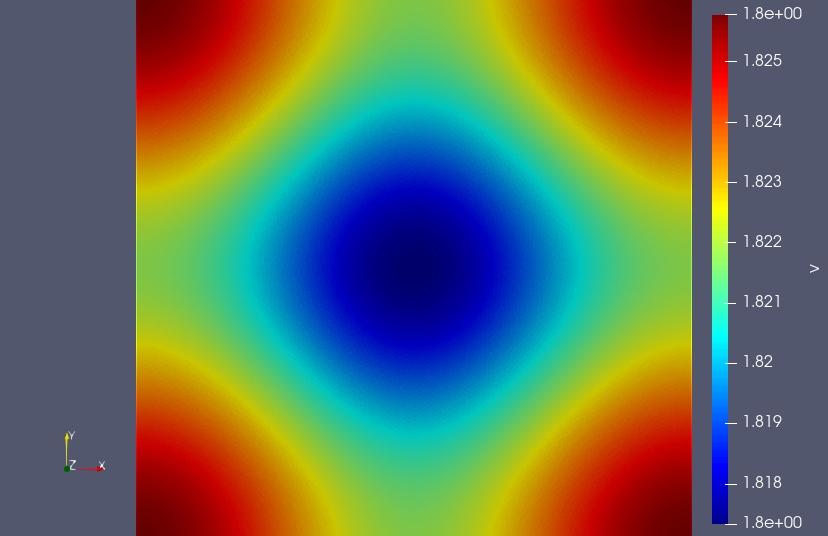} & 
\includegraphics[scale=0.125]{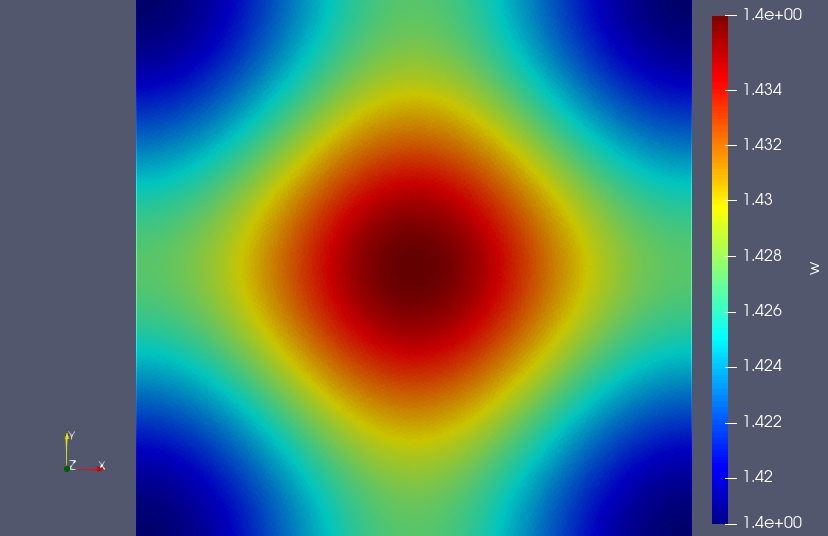} \\
($a_{5}$) $u,\quad t=20.0$ & ($b_{5}$) $v,\quad t=20.0$ & ($c_{5}$) $w,\quad t=20.0$ 
\end{tabular}
\caption{Evolution of the spatial distribution of the three species. $e_1=1.0, e_2=2.0$}
\label{fig12/3}
\end{figure}

%\newpage
\clearpage
Fourth case, let  $ e_1=1.0,\, e_2=10.0$. Prey presents still a defense capacity stronger than the previous case. 
\begin{figure}[hbt]
\begin{tabular}{ccc}
\includegraphics[scale=0.125]{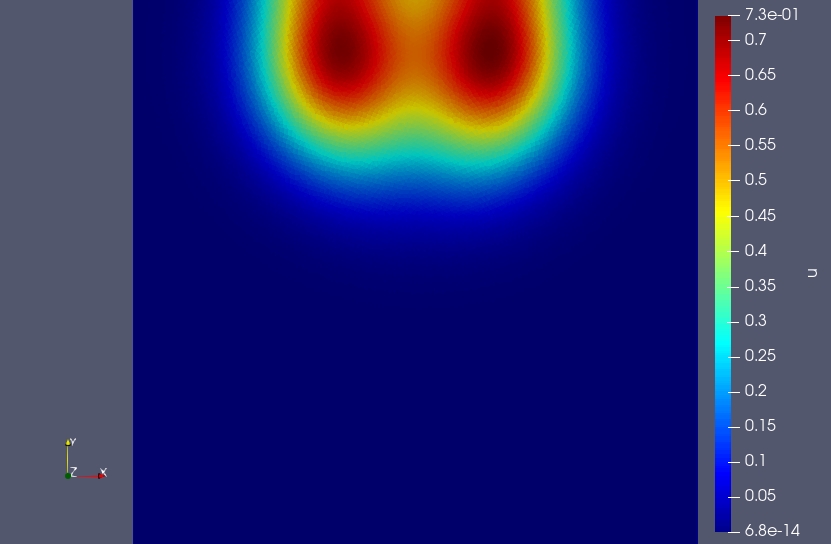} &
\includegraphics[scale=0.125]{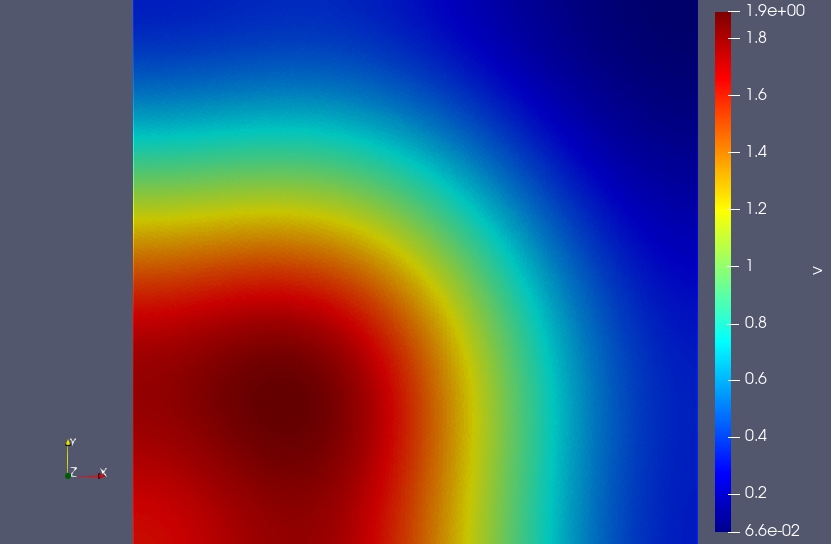} &
\includegraphics[scale=0.125]{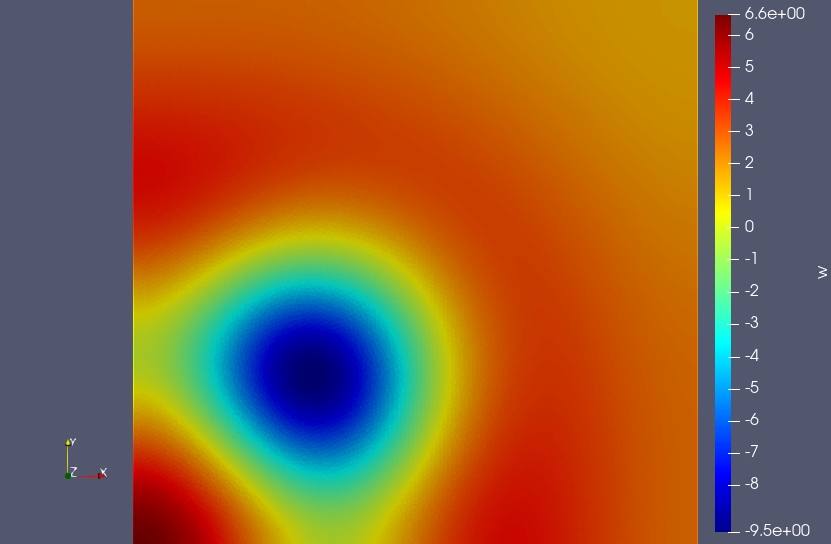} \\
($a_2$) $u,\quad t=0.1$ & ($b_2$) $v,\quad t=0.1$ & ($c_2$) $w,\quad t=0.1$  \\
\includegraphics[scale=0.125]{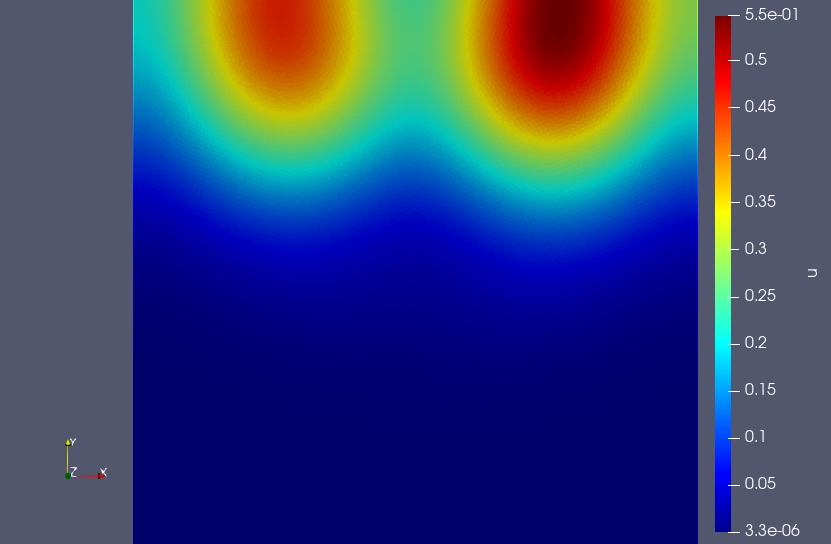} &
\includegraphics[scale=0.125]{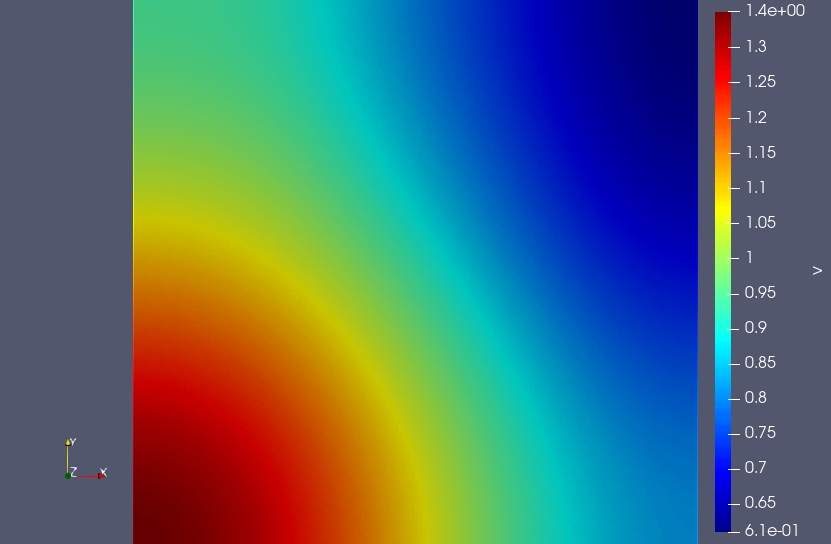} &
\includegraphics[scale=0.125]{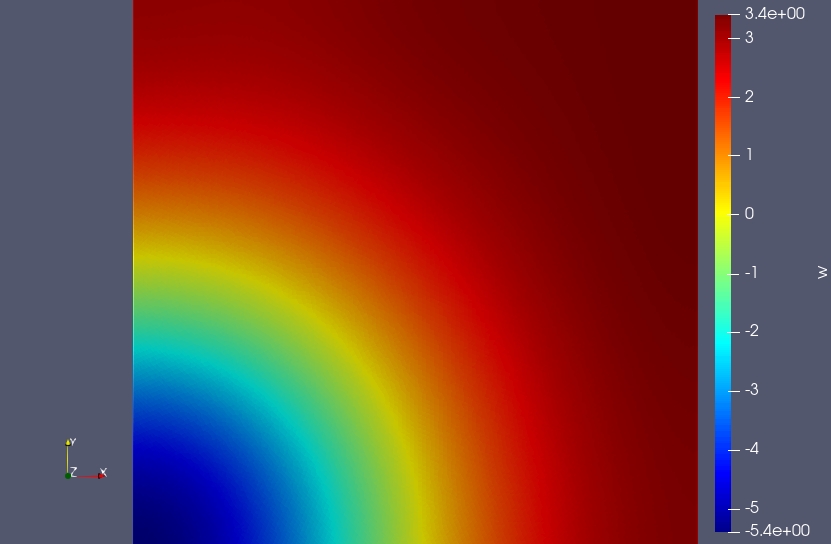} \\
($a_3$) $u,\quad t=0.5$ & ($b_3$) $v,\quad t=0.5$ & ($c_3$) $w,\quad t=0.5$  \\
%\end{tabular}
%\caption{{\em Contour plots} of time evolution of the resource $u$,  mesopredador $v$ and top predador $w$ at different times.} 
%\end{figure}
%\newpage

%\begin{figure}[hbt]
%\begin{tabular}{ccc}
\includegraphics[scale=0.125]{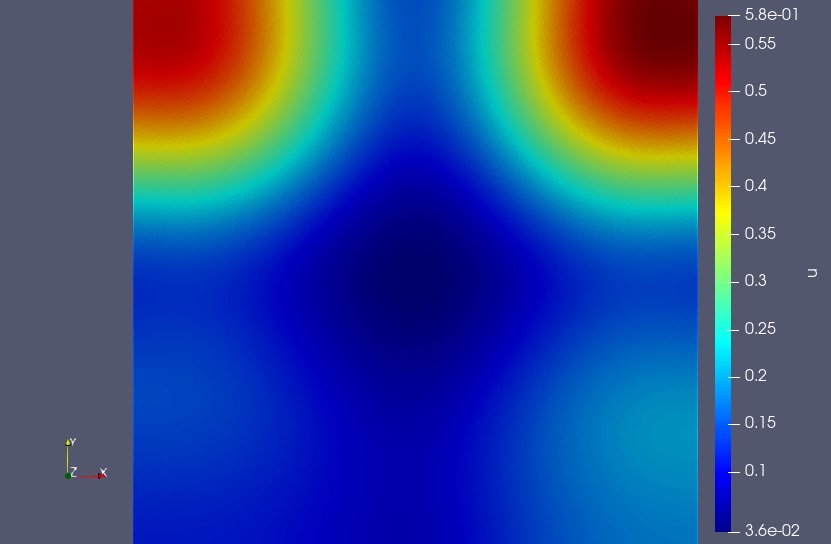} &
\includegraphics[scale=0.125]{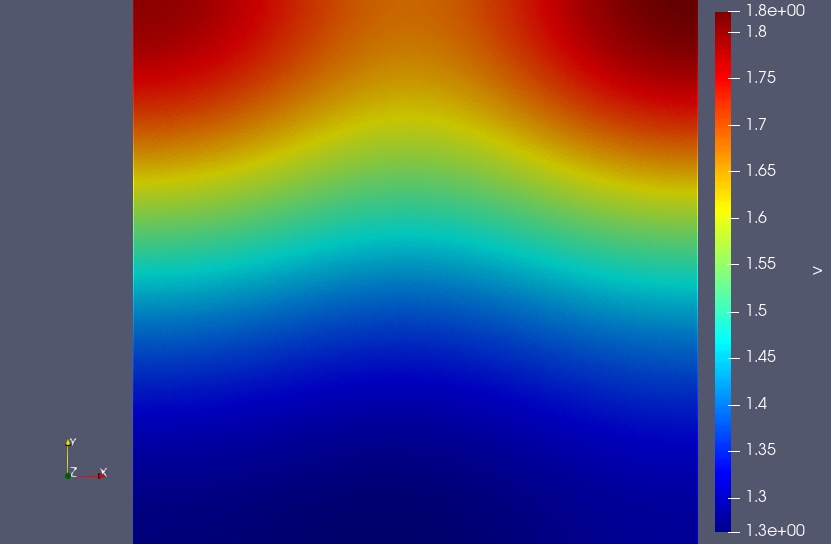} & 
\includegraphics[scale=0.125]{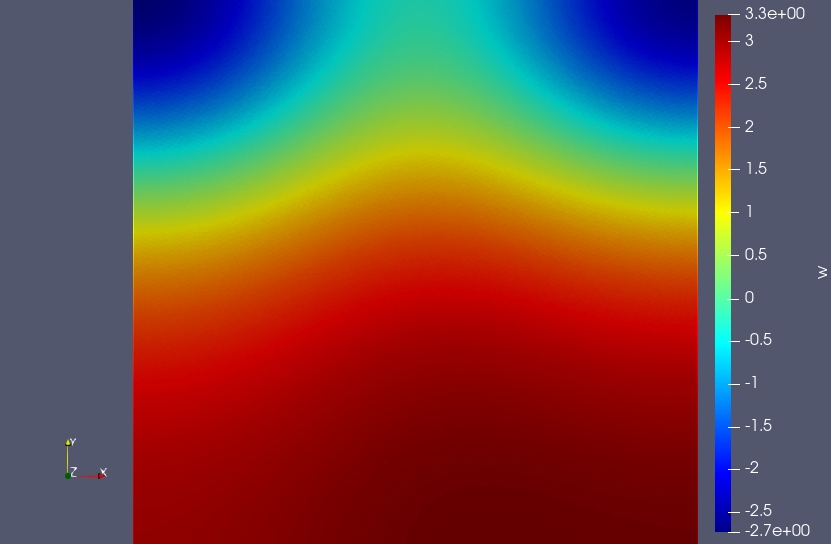} \\
($a_4$) $u,\quad t=2.0$ & ($b_4$) $v,\quad t=2.0$ & ($c_4$) $w,\quad t=2.0$ \\
%\end{tabular}
%\caption{{\em Contour plots} of time evolution of the resource $u$,  mesopredador $v$ and top predador $w$ at different times.} 
%\end{figure}

%\begin{figure}[hbt]
%\begin{tabular}{ccc}
\includegraphics[scale=0.125]{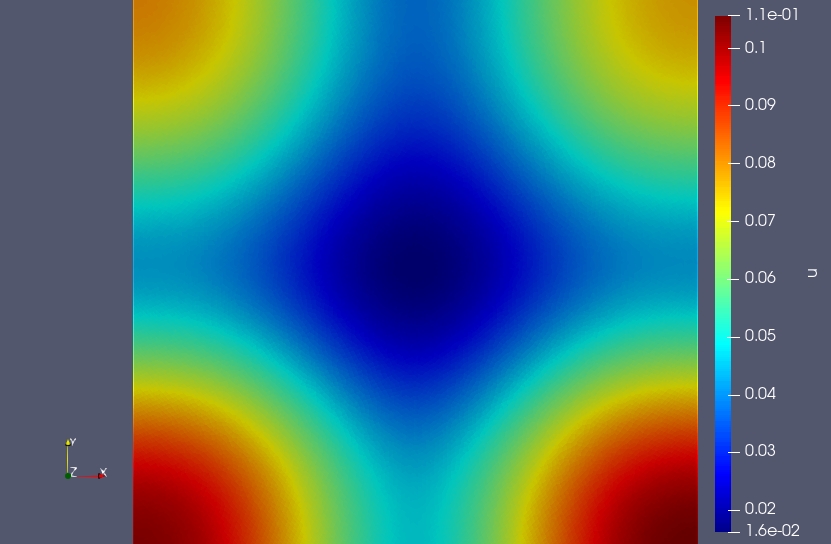} &
\includegraphics[scale=0.125]{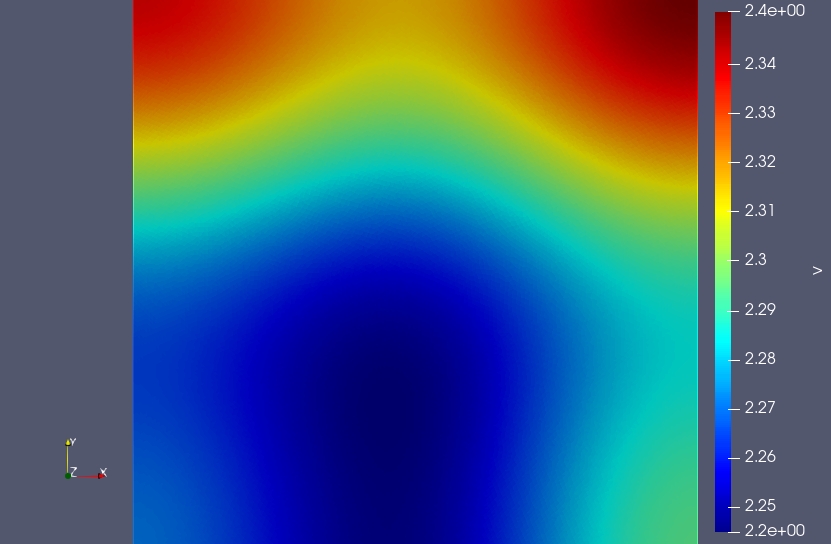} & 
\includegraphics[scale=0.125]{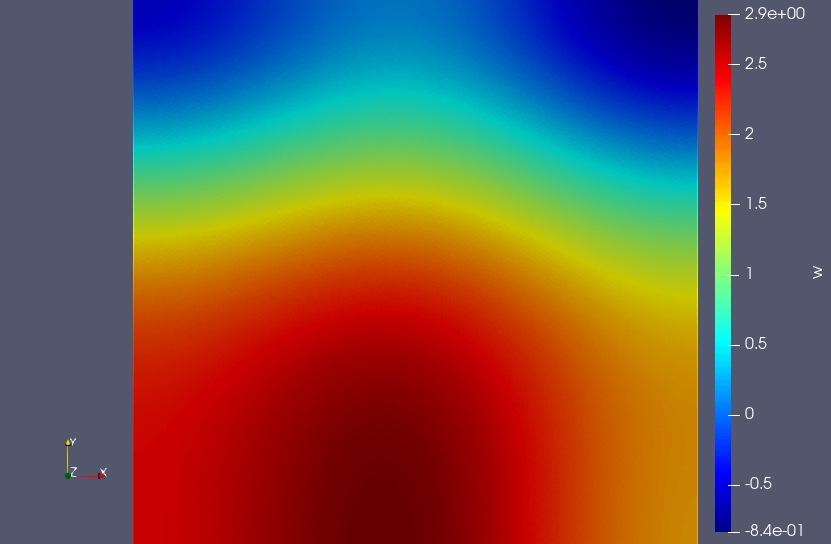} \\
($a_{5}$) $u,\quad t=4.0$ & ($b_{5}$) $v,\quad t=4.0$ & ($c_{5}$) $w,\quad t=4.0$ \\
\includegraphics[scale=0.125]{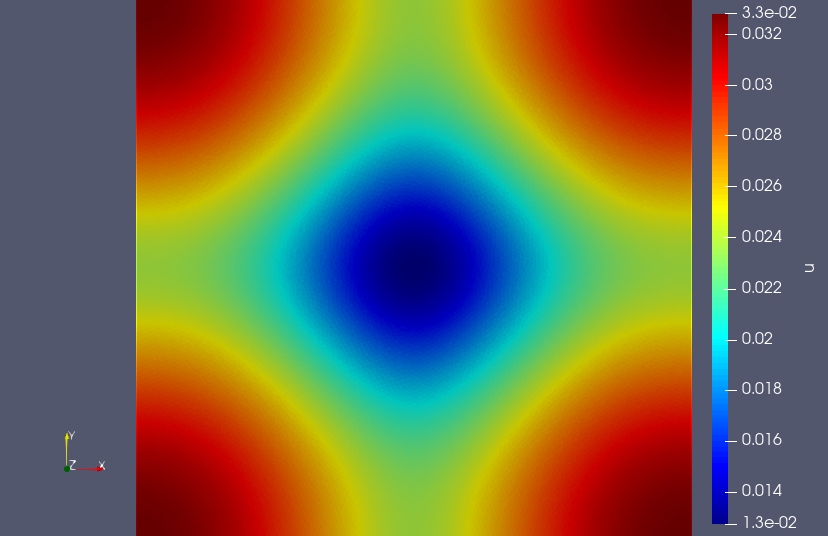} &
\includegraphics[scale=0.125]{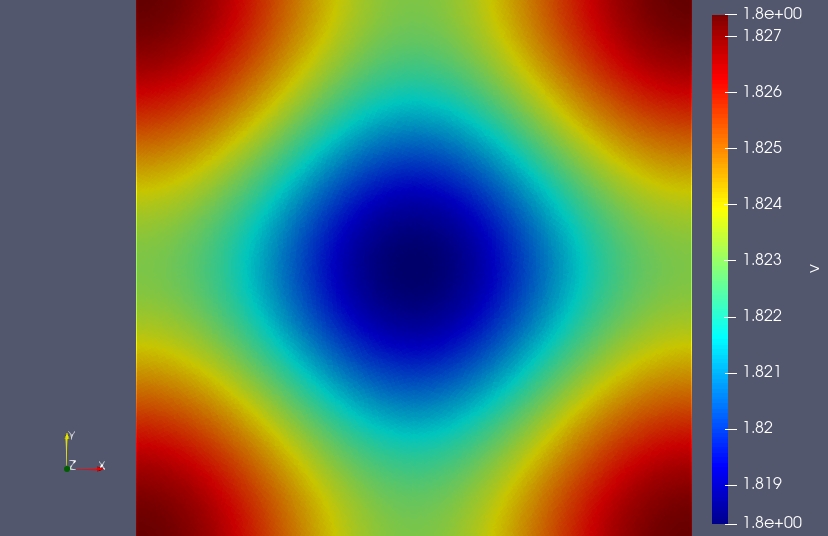} & 
\includegraphics[scale=0.125]{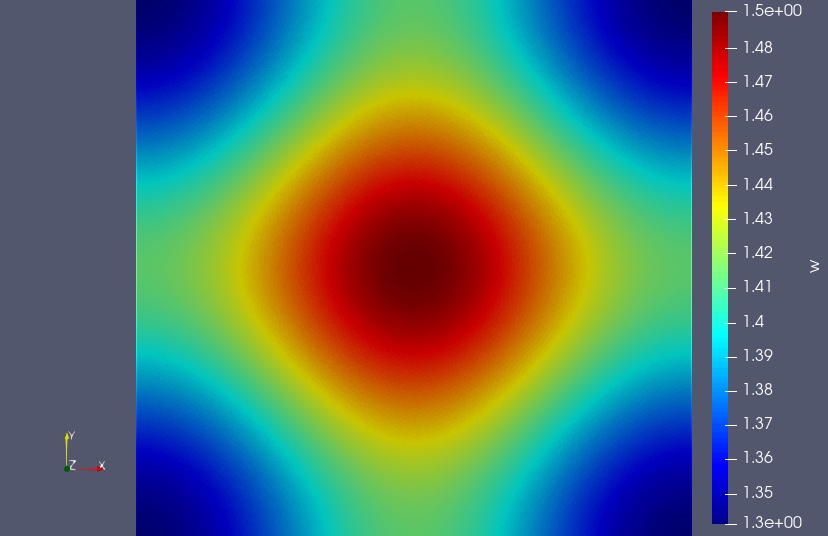} \\
($a_{6}$) $u,\quad t=20.0$ & ($b_{6}$) $v,\quad t=20.0$ & ($c_{6}$) $w,\quad t=20.0$ \\
\end{tabular}
\caption{Evolution of the spatial distribution of the three species. $e_1=1.0, e_2=10.0$}
\end{figure}

%\newpage
\clearpage
Fifth case, Let  $ e_1=10.0,\, e_2=1.0$ This is the smallest defensive capacity considered in this section. 
\begin{figure}[hbt]
\begin{tabular}{ccc}
\includegraphics[scale=0.125]{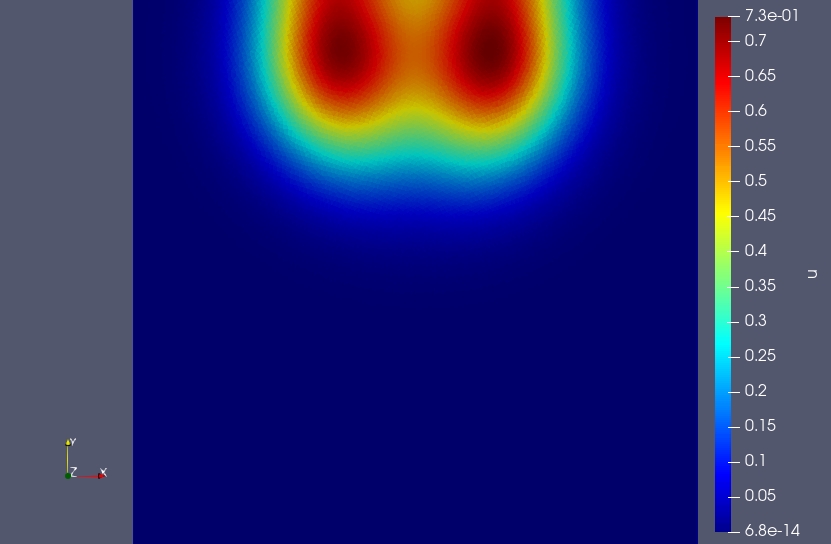} &
\includegraphics[scale=0.125]{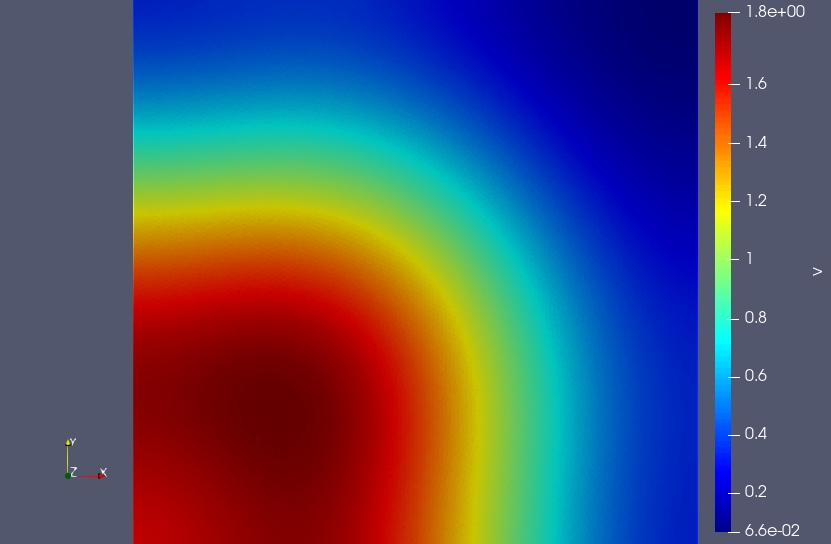} &
\includegraphics[scale=0.125]{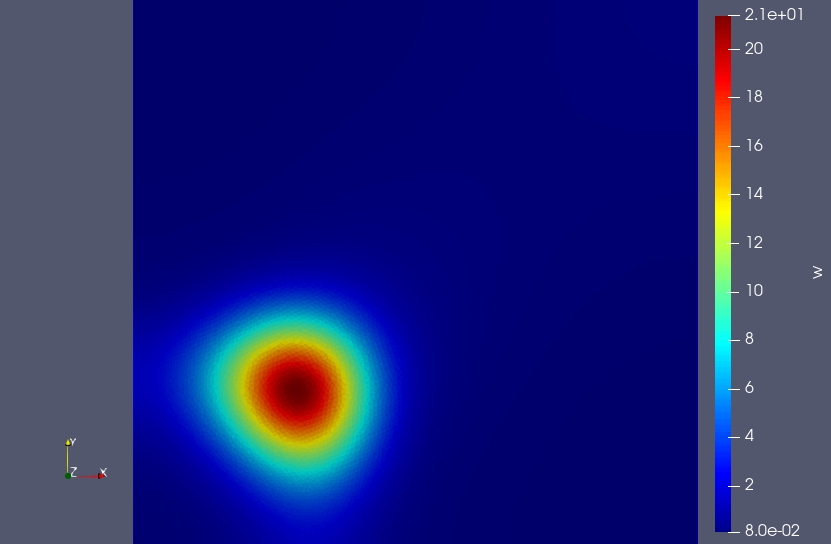} \\
($a_2$) $u,\quad t=0.1$ & ($b_2$) $v,\quad t=0.1$ & ($c_2$) $w,\quad t=0.1$  \\
\includegraphics[scale=0.125]{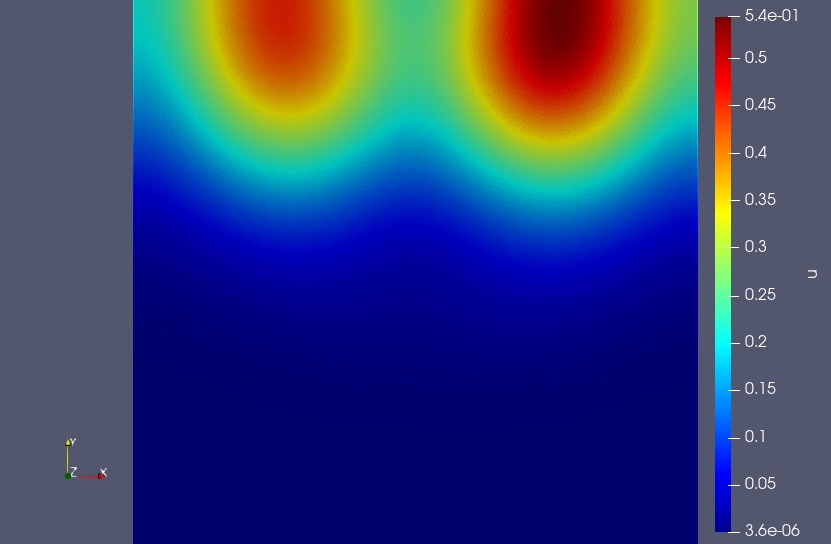} &
\includegraphics[scale=0.125]{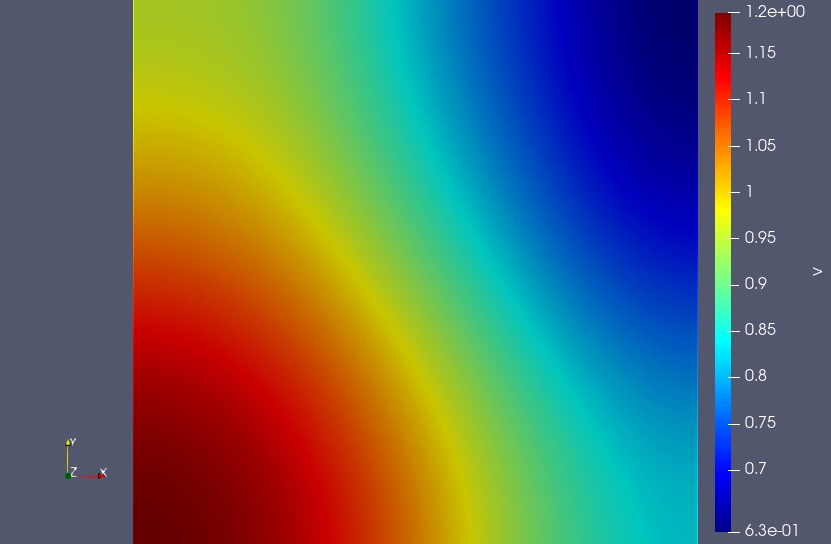} & 
\includegraphics[scale=0.125]{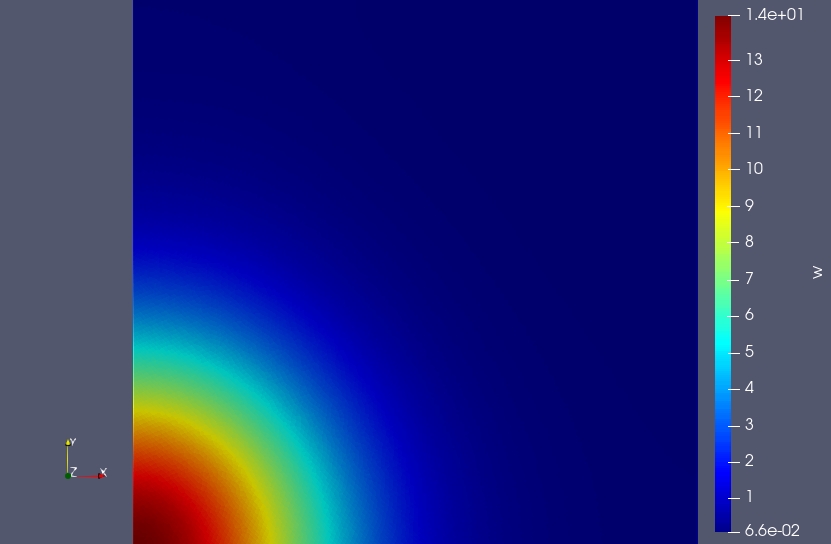} \\
($a_{3}$) $u,\quad t=0.5$ & ($b_{3}$) $v,\quad t=0.5$ & ($c_{3}$) $w,\quad t=0.5$ \\
%\end{tabular}
%\caption{{\em Contour plots} of time evolution of the resource $u$,  mesopredador $v$ and top predador $w$ at different times.} 
%\end{figure}

%\begin{figure}[hbt]
%\begin{tabular}{ccc}
\includegraphics[scale=0.125]{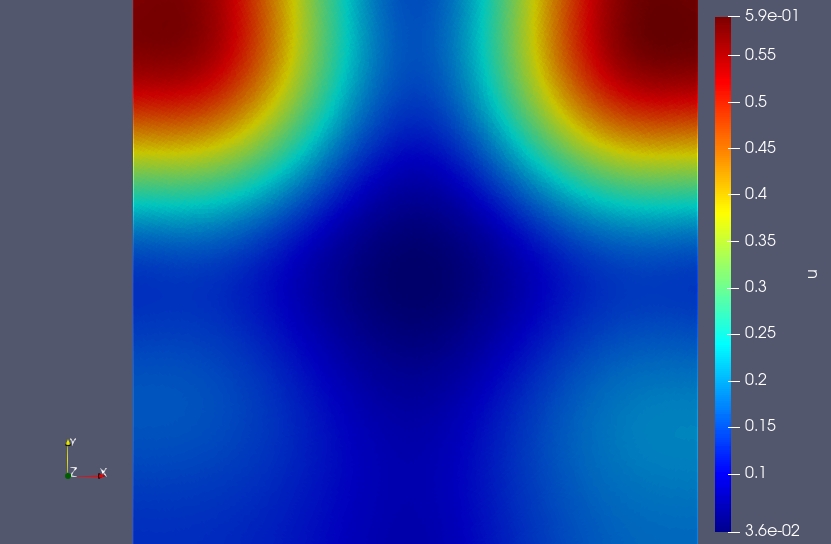} &
\includegraphics[scale=0.125]{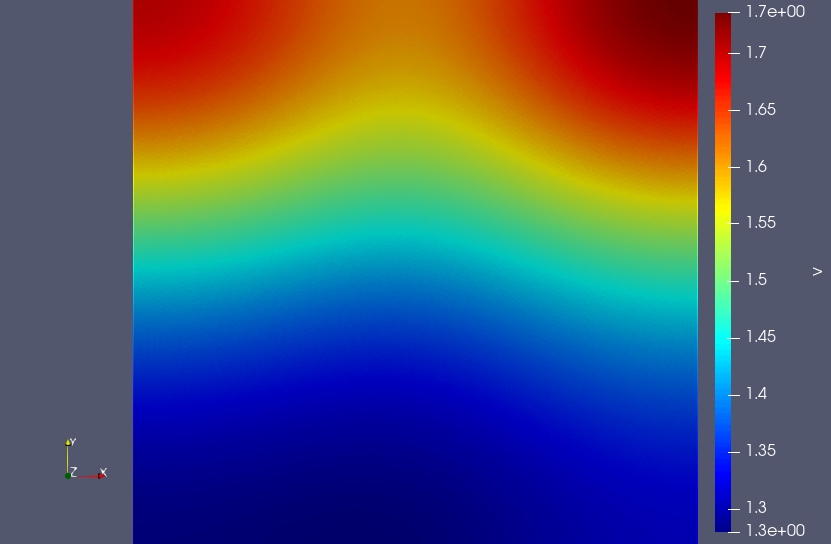} & 
\includegraphics[scale=0.125]{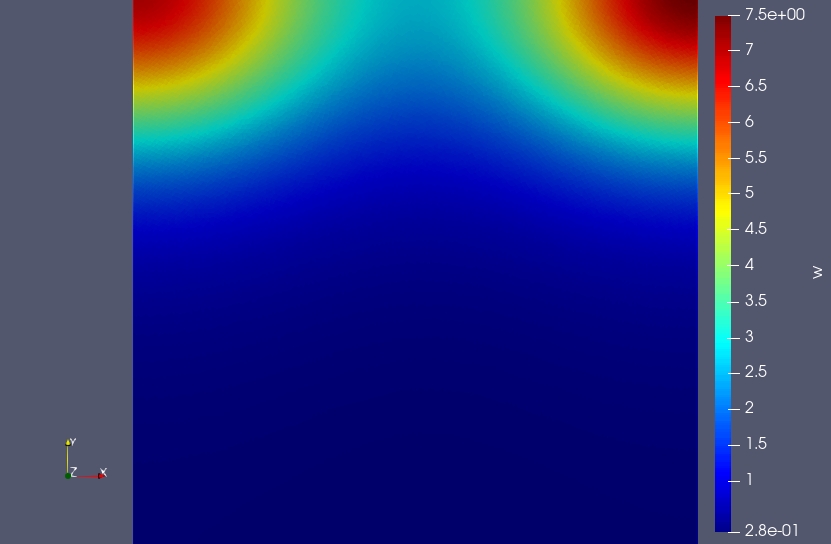} \\
($a_4$) $u,\quad t=2.0$ & ($b_4$) $v,\quad t=2.0$ & ($c_4$) $w,\quad t=2.0$ \\
%\end{tabular}
%\caption{{\em Contour plots} of time evolution of the resource $u$,  mesopredador $v$ and top predador $w$ at different times.} 
%\end{figure}

%\begin{figure}[hbt] 
%\begin{tabular}{ccc}
\includegraphics[scale=0.125]{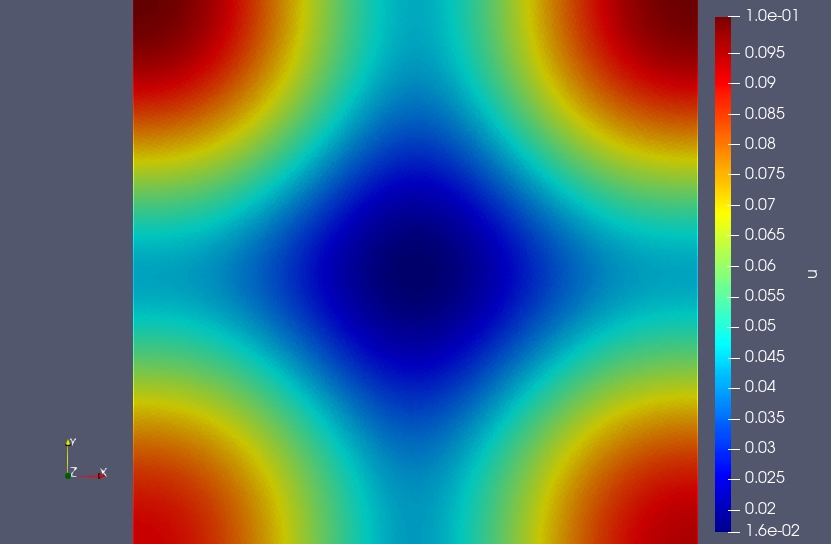} &
\includegraphics[scale=0.125]{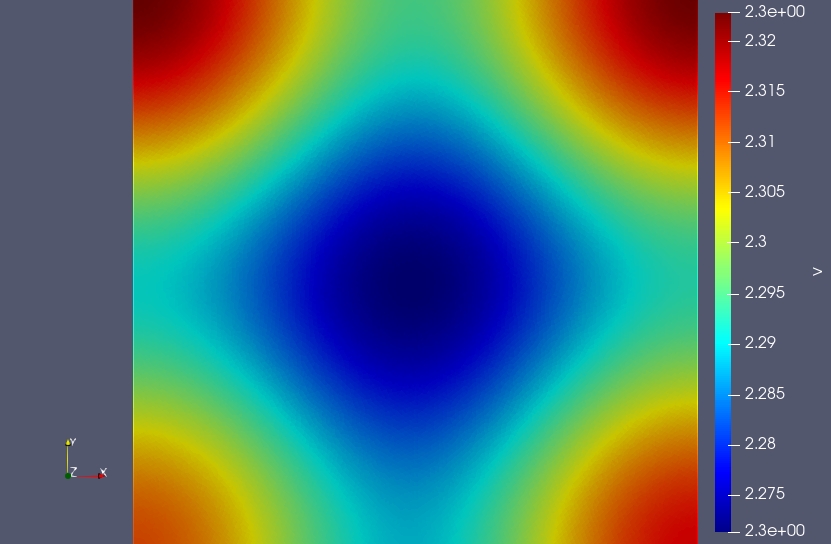} & 
\includegraphics[scale=0.125]{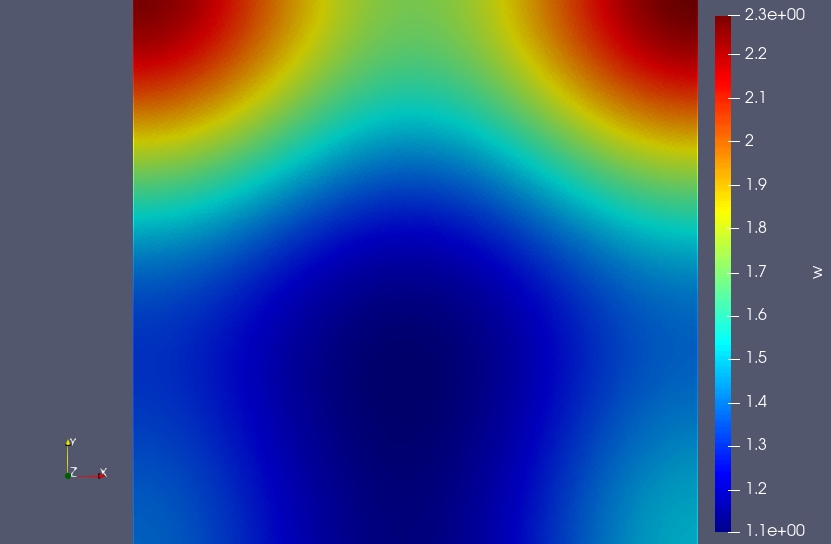} \\
($a_{5}$) $u,\quad t=4.0$ & ($b_{5}$) $v,\quad t=4.0$ & ($c_{5}$) $w,\quad t=4.0$ \\
\includegraphics[scale=0.125]{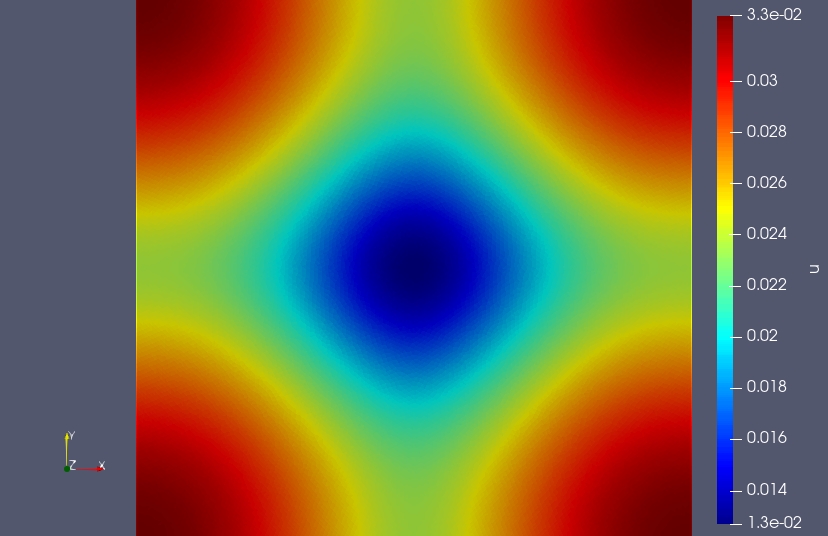} &
\includegraphics[scale=0.125]{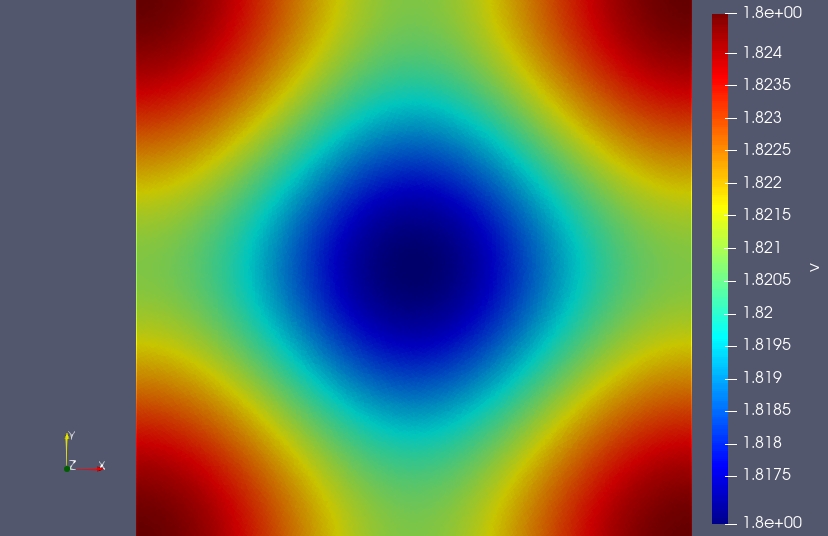} & 
\includegraphics[scale=0.125]{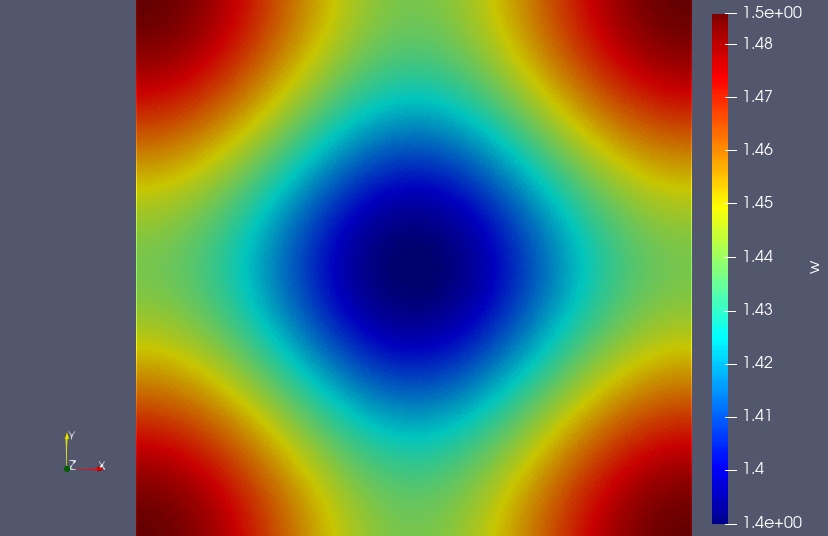} \\
($a_{6}$) $u,\quad t=20.0$ & ($b_{6}$) $v,\quad t=20.0$ & ($c_{6}$) $w,\quad t=20.0$
\end{tabular}
\caption{Evolution of the spatial distribution of the three species. $e_1=10.0, e_2=1.0$}
\label{ff4}
\end{figure}

%\newpage
\clearpage
From the comparison of Figures \ref{Figu3b}-\ref{ff4}, we conclude that defensive capacity has a negligible effect on the prey population, if the predation rate is not large enough. Indeed, the main impact is over the spatial distribution of both the meso predator and  the top predator.
%\end{tabular}
%\caption{{\em Contour plots} of time evolution of the resource $u$,  mesopredador $v$ and top predador $w$ at different times.} 
%\end{figure}
%\newpage

%\begin{figure}[hbt] 
%\begin{tabular}{ccc}
%
%\begin{tabular}{ccc}
\subsubsection{Habitat suitability and species distribution}
To understand how the ecological landscape impact species distribution, we consider two different characterization of the carrying capacity. In either case, the values of parameters of $\chi_1$ are $ e_1=1.0,\, e_2=10.0,$ and the initial condition of $u$ is
$$
u_0(x,y) =  2\exp(-10(x^2+ y^2)) (1-x^2)^2(1 - y^2)^2
$$
The initial conditions $v_0(x,y)$ and $w_0(x,y)$ are the same as above.
\\First, we consider a carrying capacity given by 

\begin{eqnarray*}
K(x,y) &=&  2\exp(-5((x+.75)^2+(y-.75)^2))+2\exp(-5((x-.75)^2+(y+.75)^2)),   \\
          &  &  +2 \exp(-5((x+.75)^2+(y+.75)^2))+2\exp(-5((x-.75)^2+(y-.75)^2)) . 
\end{eqnarray*}
The highest  suitability is reached at four symmetrical points respect to the origin. 

In Figure (\ref{Figu1}) are  shown plots of the numerical solutions of $u$, $v$ and $w$ at different times.  Note that as time passes, the resource tends to occupy the most suitable sites. The mesopredator moves towards the sites with the higher resource density and its defensive capacity ($e_2/e_1$) is large enough to keep the top predator away.

%\newpage

%\subsection*{Numerical results}

\begin{figure}[hbt] 
\begin{tabular}{ccc}
\includegraphics[scale=0.125]{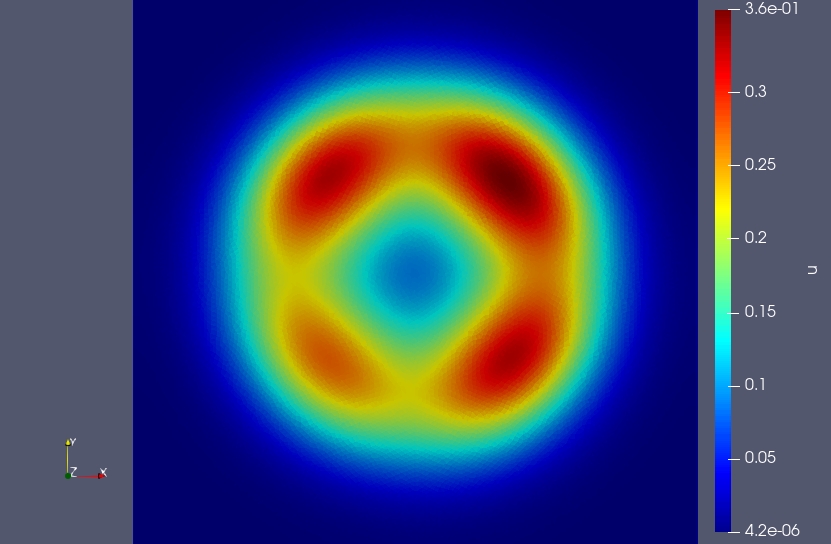} &
\includegraphics[scale=0.125]{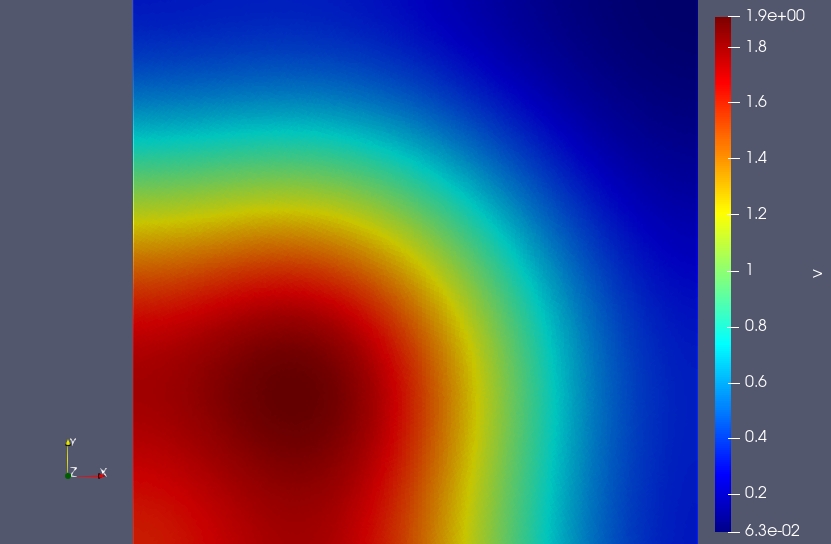} &
\includegraphics[scale=0.125]{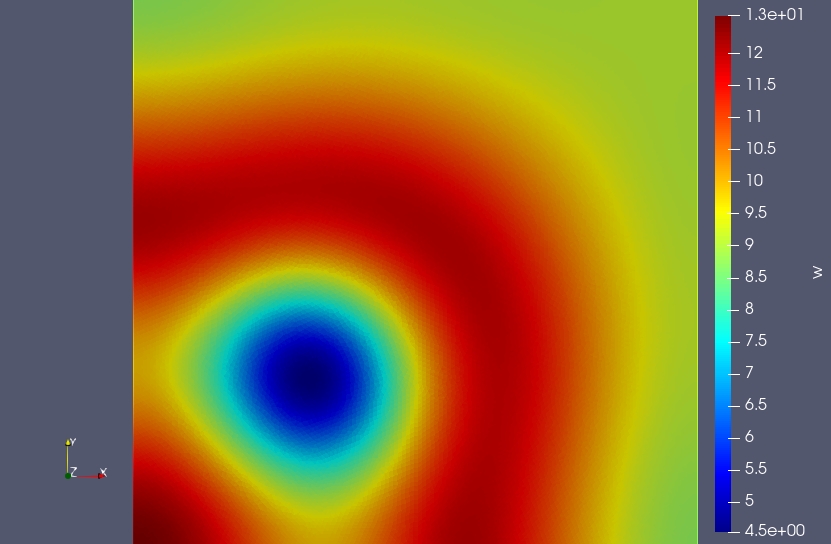} \\
($a_1$) $u,\quad t=0.1$ & ($b_1$) $v,\quad t=0.1$ & ($c_1$) $w,\quad t=0.1$  \\
\includegraphics[scale=0.15]{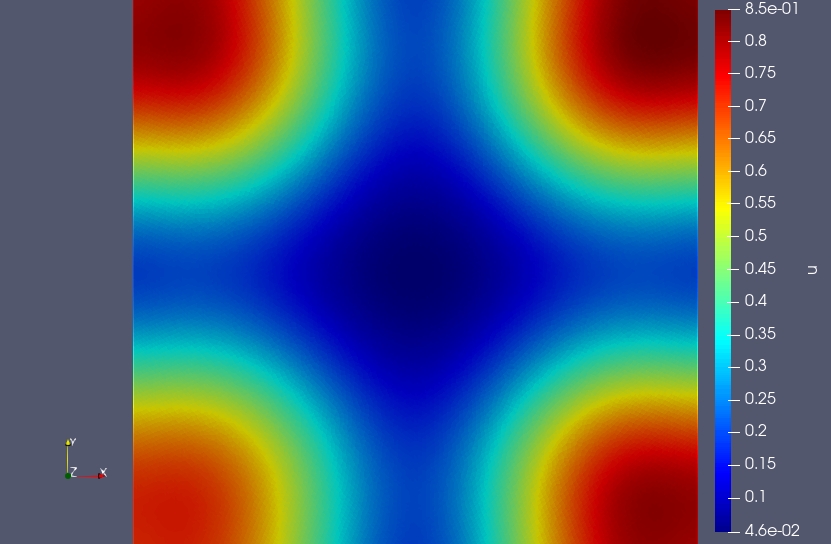} &
\includegraphics[scale=0.125]{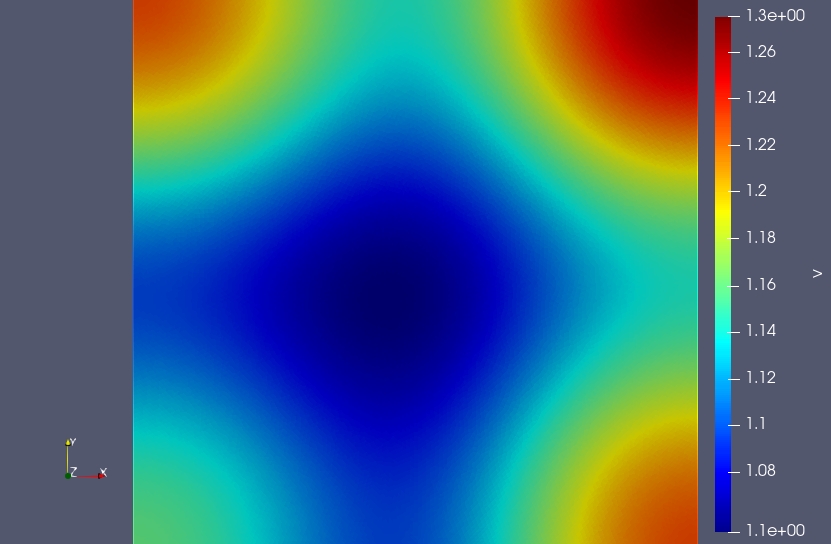} & 
\includegraphics[scale=0.125]{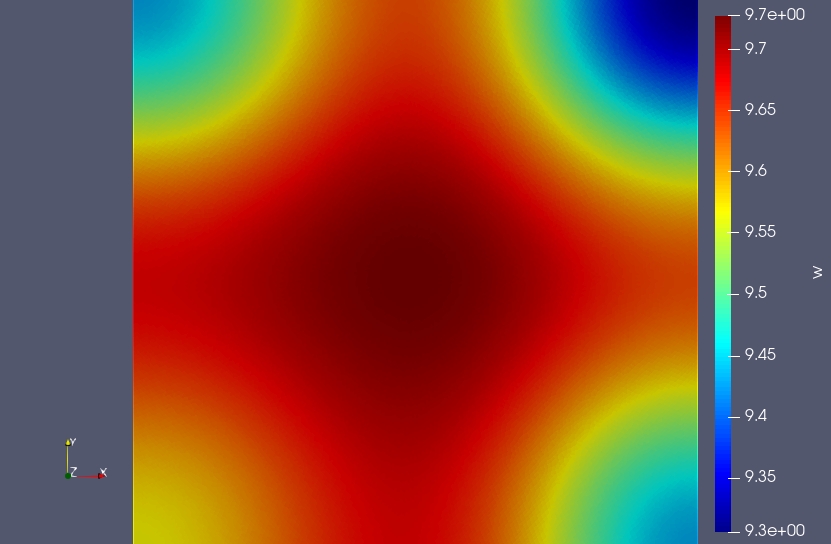} \\
($a_2$) $u,\quad t=2.0$ & ($b_2$) $v,\quad t=2.0$ & ($c_2$) $w,\quad t=2.0$\\
\includegraphics[scale=0.125]{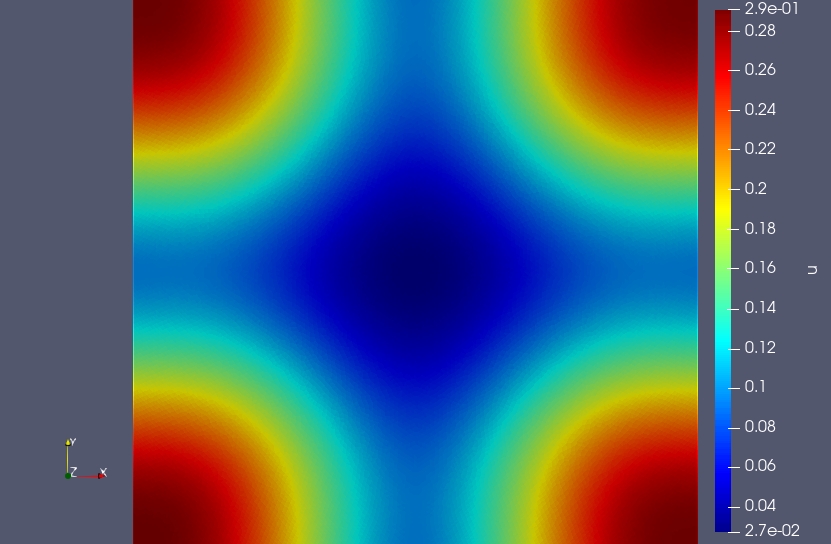} &
\includegraphics[scale=0.125]{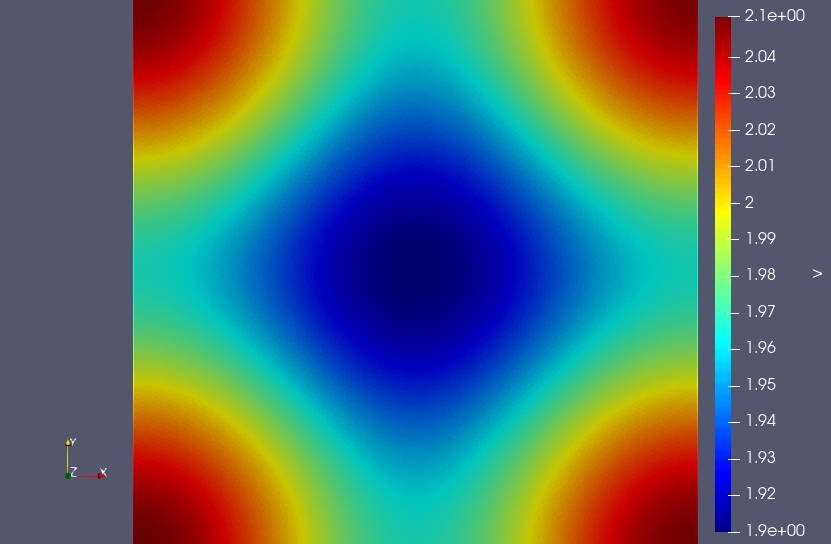} & 
\includegraphics[scale=0.125]{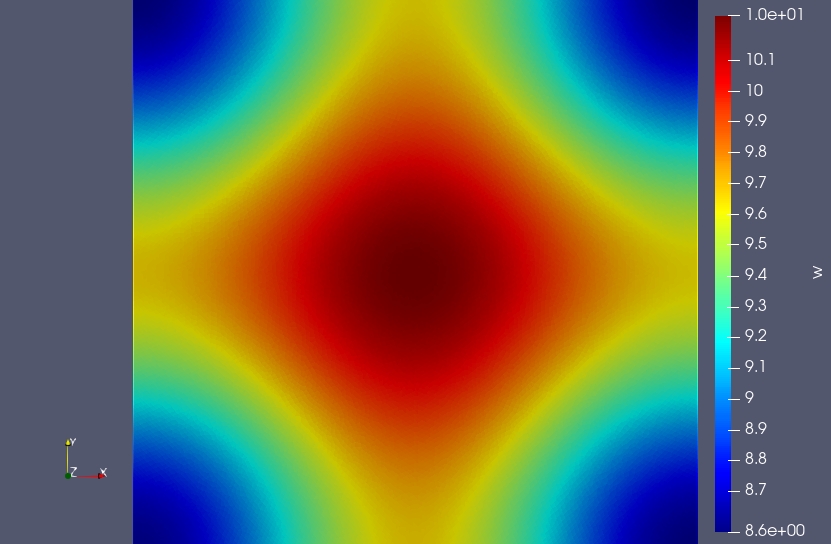} \\
($a_{3}$) $u,\quad t=4.0$ & ($b_{3}$) $v,\quad t=4.0$ & ($c_{3}$) $w,\quad t=4.0$ \\
 \includegraphics[scale=0.125]{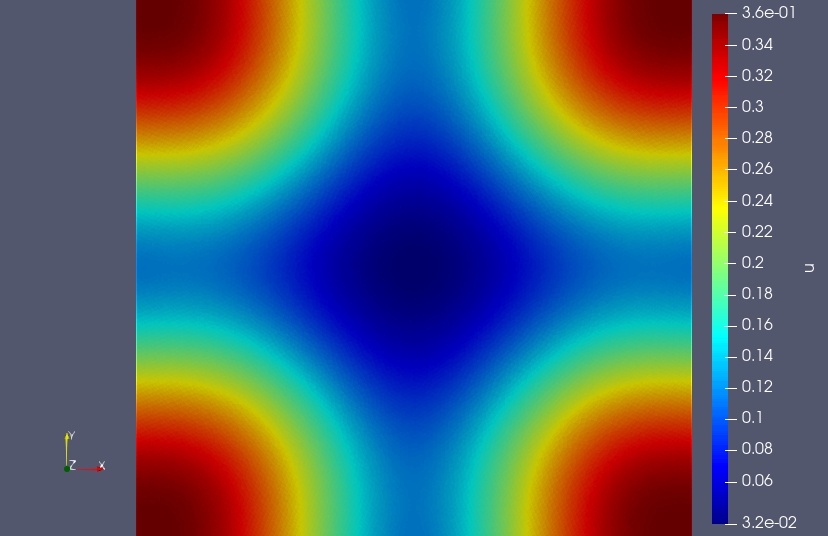} &
\includegraphics[scale=0.125]{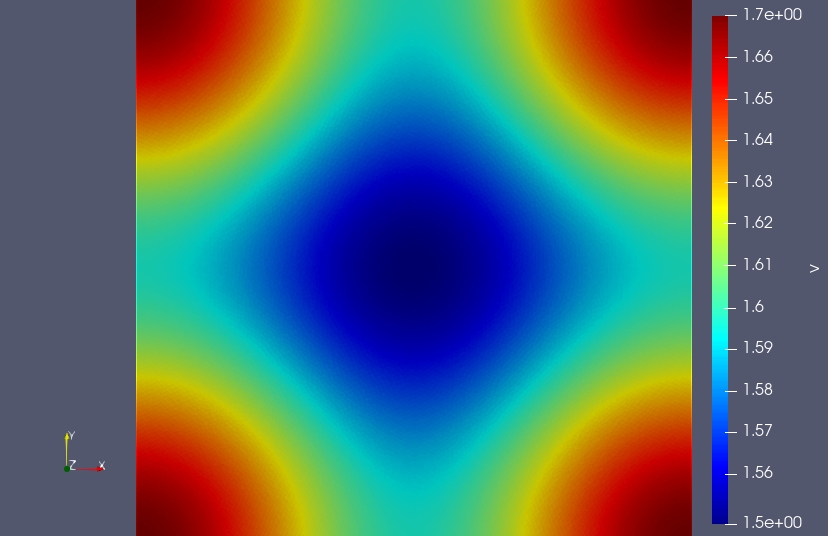} & 
\includegraphics[scale=0.125]{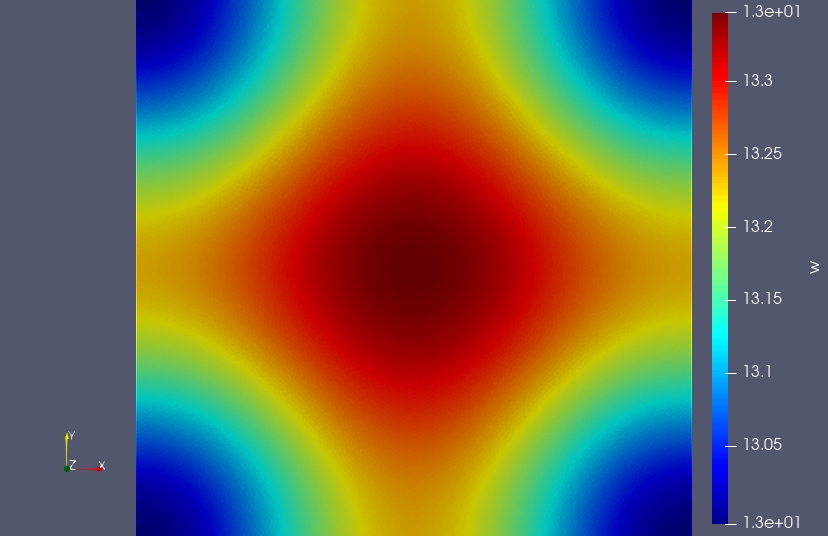} \\
($a_{4}$) $u,\quad t=20.0$ & ($b_{4}$) $v,\quad t=20.0$ & ($c_{4}$) $w,\quad t=20.0$
\end{tabular}
\caption{Evolution of the spatial distribution of the three species.}  \label{Figu1}
\end{figure}

\newpage
\clearpage

In this second case, the habitat of the resource is richer since its suitability is given by
\begin{eqnarray} \label{Kuno}
K(x,y) &=&  2\exp(-5((x+.75)^2+(y-.75)^2))+2\exp(-5((x-.75)^2+(y+.75)^2)),  \nonumber \\
          &  &  +2 \exp(-5((x+.75)^2+(y+.75)^2))+2\exp(-5((x-.75)^2+(y-.75)^2)) \\
          & & + 2\exp(-5(x^2+y^2)) \nonumber
\end{eqnarray}
The highest  suitability is reached at four symmetrical points respect to the origin and at the origin. 
The spatial distribution of the three species is shown in Figure \ref{Figu4}.  \\
As in the first case, the mesopredators move towards the sites of higher density of the resource and the top predator is located far enough away from its prey because $e_2/e_1$ is relatively high. It seems that the richness of the habitat does not induce any change in the distribution patterns. Top predator tends to occupy the areas less densely populated by mesopredators, if $e_2/e_1$ is high enough.

%\clearpage
%\newpage
%\subsection*{Numerical results}
%                                                             SEPA de donde lo agarraron
\begin{figure}[hbt] %OOOOJO figuras
\begin{tabular}{ccc}
\includegraphics[scale=0.125]{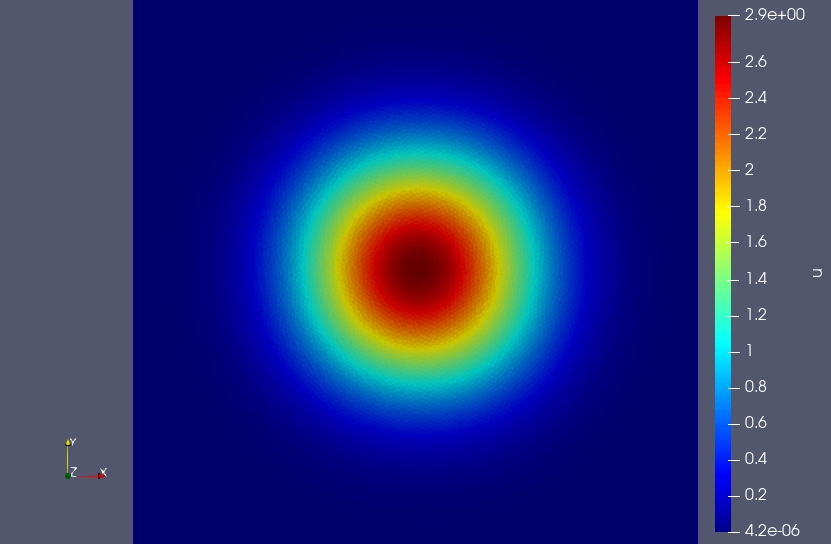} &
\includegraphics[scale=0.125]{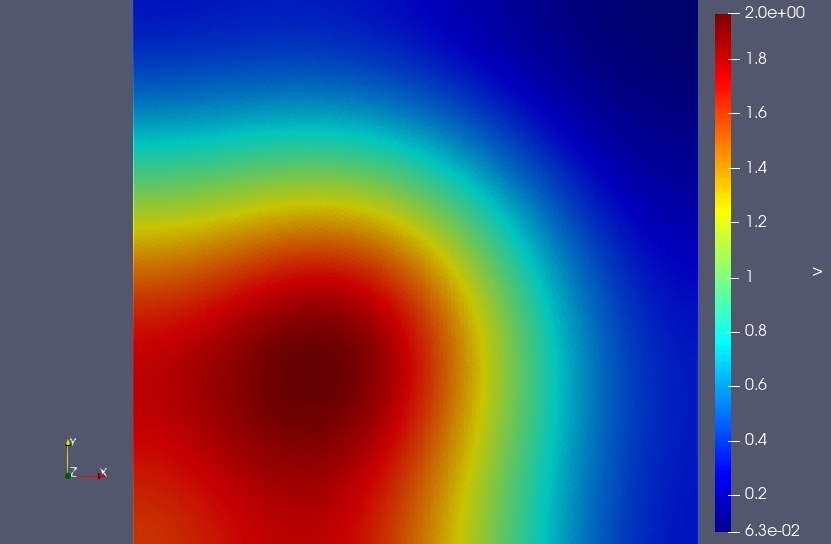} &
\includegraphics[scale=0.125]{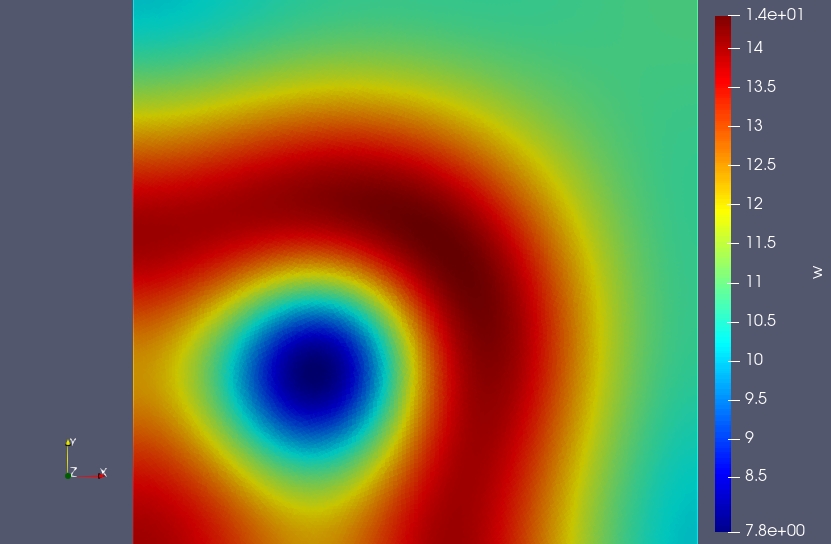} \\
($a_{1}$) $u,\quad t=0.1$ & ($b_{1}$) $v,\quad t=0.1$ & ($c_{1}$) $w,\quad t=0.1$   \\
\includegraphics[scale=0.125]{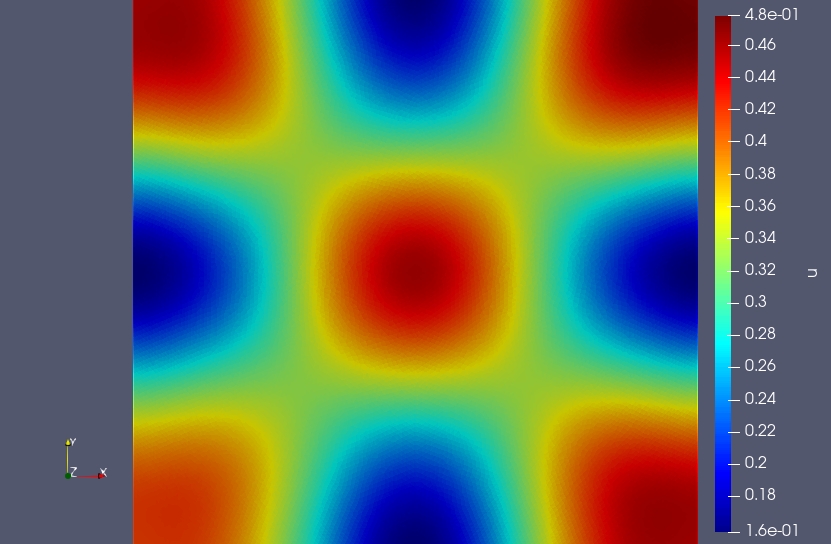} &
\includegraphics[scale=0.125]{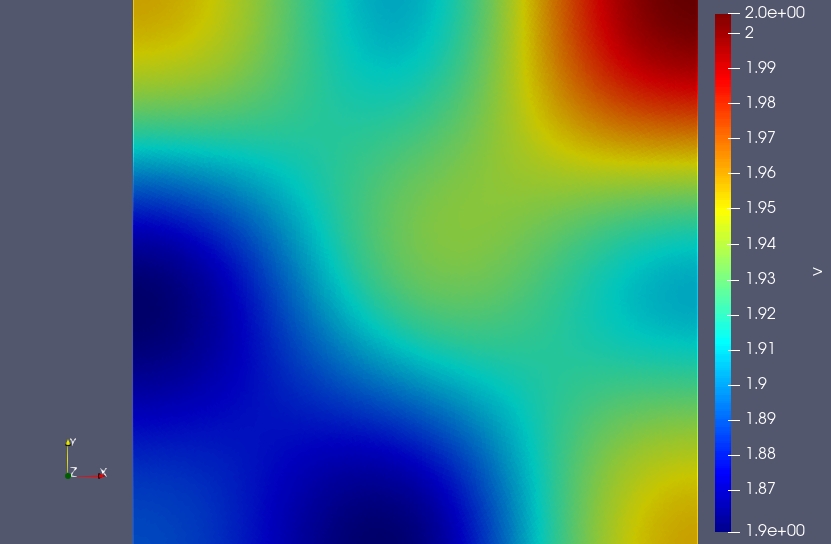} & 
\includegraphics[scale=0.125]{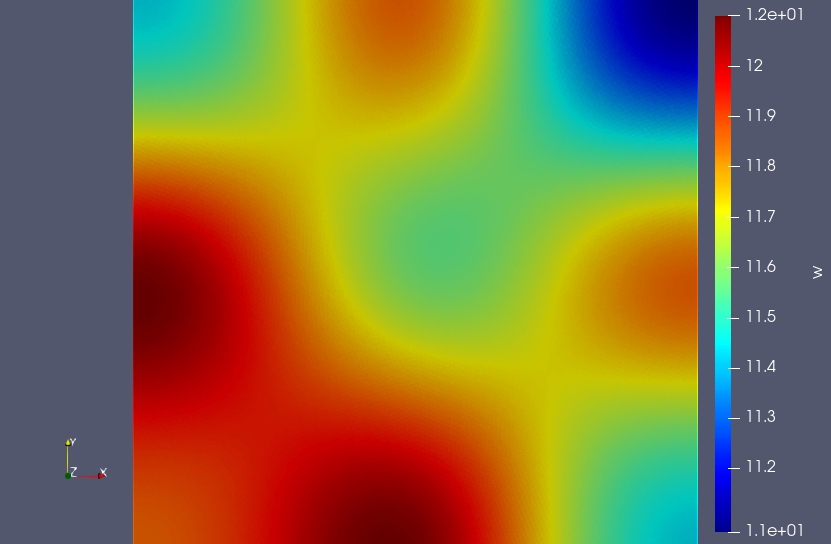} \\
($a_{2}$) $u,\quad t=2.0$ & ($b_{2}$) $v,\quad t=2.0$ & ($c_{2}$) $w,\quad t=2.0$  \\
\includegraphics[scale=0.125]{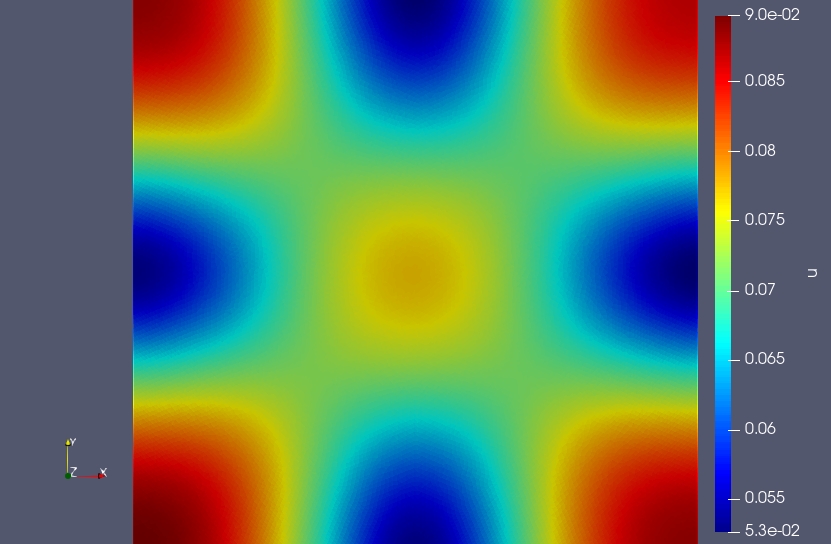} &
\includegraphics[scale=0.125]{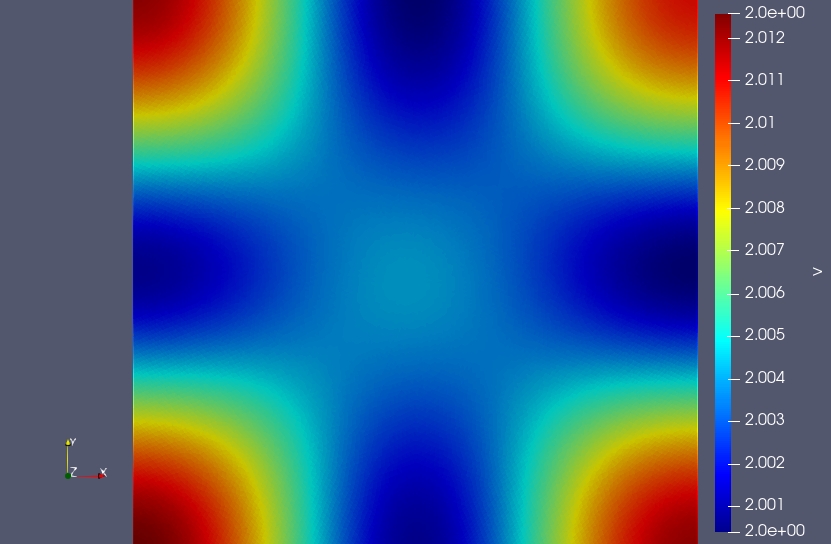} & 
\includegraphics[scale=0.125]{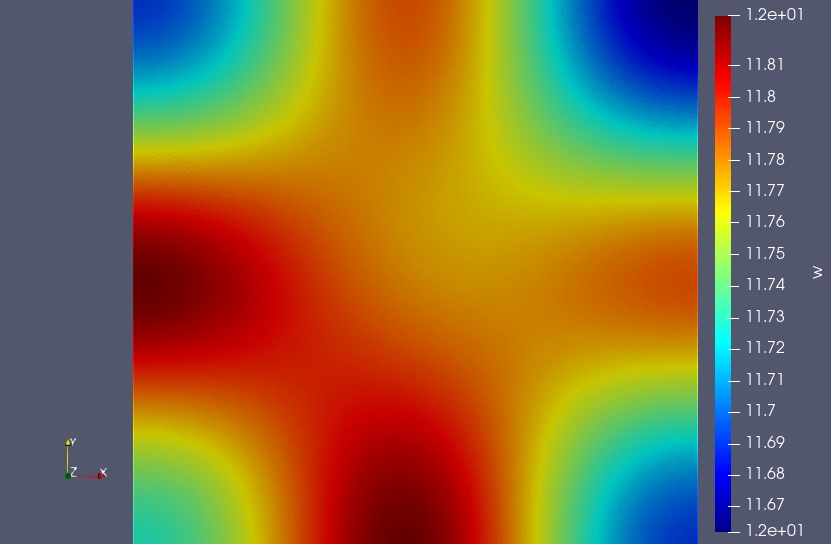} \\
($a_{3}$) $u,\quad t=4.0$ & ($b_{3}$) $v,\quad t=4.0$ & ($c_{3}$) $w,\quad t=4.0$  \\
\includegraphics[scale=0.125]{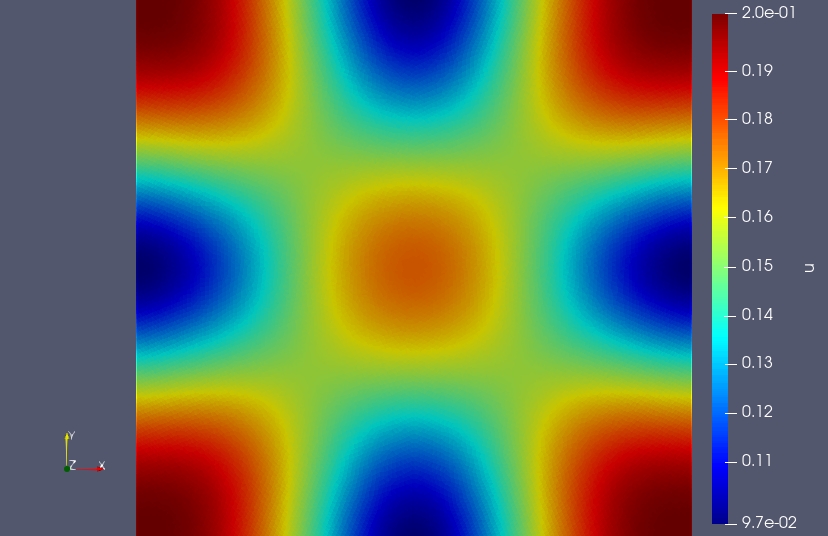} &
\includegraphics[scale=0.125]{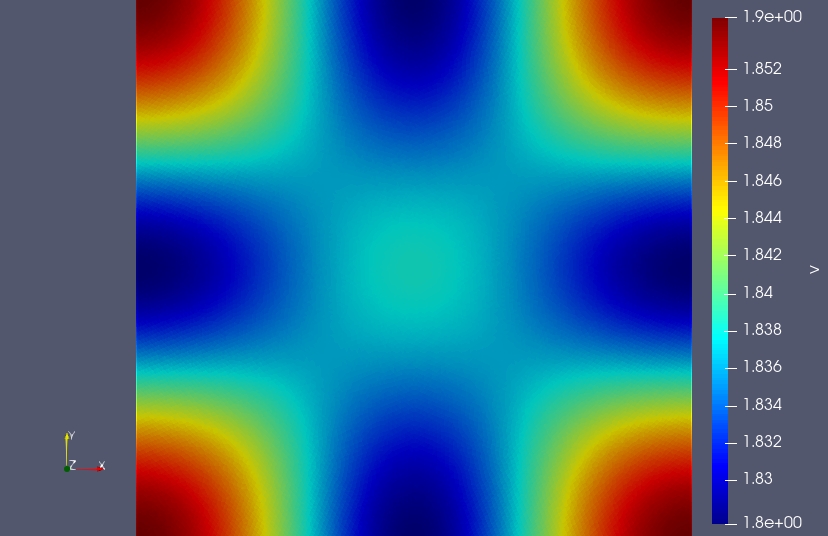} & 
\includegraphics[scale=0.125]{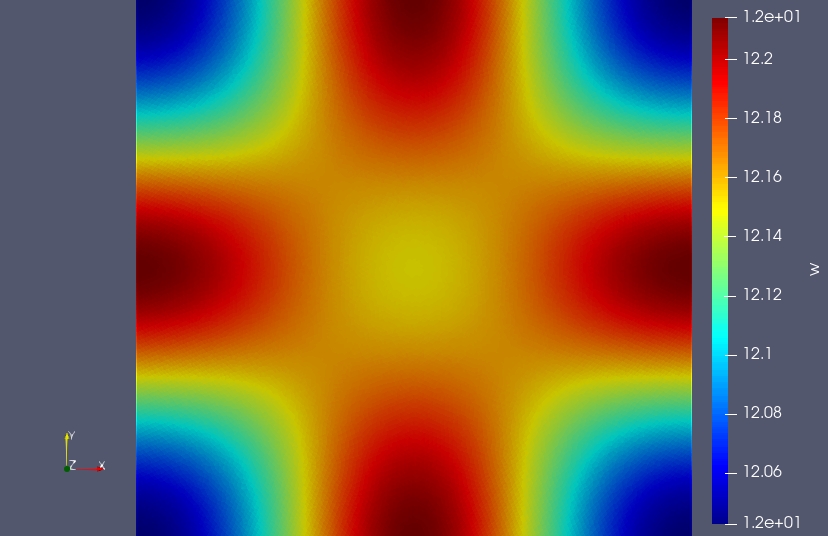} \\
($a_{4}$) $u,\quad t=20.0$ & ($b_{4}$) $v,\quad t=20.0$ & ($c_{4}$) $w,\quad t=20.0$  
\end{tabular}
\caption{Evolution of the spatial distribution of the three species. The suitability of resource habitat is given by (\ref{Kuno})}. \label{Figu4}
\end{figure}
%\clearpage
%\newpage
\subsection{Resource defense and species distribution}
Some species defend themselves by attracting predators from their natural enemies. This is very frequent for instance in plant species, see \cite{pri} and the bibliography cited there. In \cite{alj} it has been described 24 species of predators  which are attracted by volatiles generated by plants damaged by herbivores. In this paper, the authors  raise the question about the effectiveness of predator species in controlling specific insect pests.
In the following we analyze numerically the impact on the mesopredator distribution of an increasing predation rate of the top predator when this is attracted by the resource species. To analyze the relationship between the distribution of the mesopredator and the predation rate of a top predator that is attracted to the resource, we use Model (\ref{mod2f}) which is shown below.
\begin{eqnarray} 
\frac{\partial u}{\partial t} & =&  d_0\Delta u +  \alpha u  ( 1 - \frac{u }{K(x,y) })- \frac{b u v }{u + a} ,\nonumber \\
\frac{\partial v}{\partial t} & = &  d_1\Delta v + \gamma \frac{b u v }{u + a} - \frac{c v w}{v+d} -\mu v,  \\
\frac{\partial w}{\partial t} &= &  d_2\Delta w +\beta \frac{c v w }{v + d}  -\nu w -  \div ( \chi_2(u,w) \grad u) .\nonumber
\end{eqnarray} 
The sensitivity function is $\chi_2(u,w)=q u w$. Thus, the movement of top predators towards the gradient of $u$ is faster the higher its own density or that of the resource. \\
Initial conditions for the spatial distribution of the resource, the meso-predator and top predator are considered as 
\begin{eqnarray*}
u_0(x,y) &= & 2\exp(-(x^2+ (y-.9)^2)(1-x^2)^2(1 - y^2)^2; \\
v_0(x,y) &= & 2\exp(-(x + .9)^2 - (y + .9)^2)(1 - x^2)^2(1 - y^2)^2;\\
w_0(x,y) &= & 1.5
\end{eqnarray*}
for all $x,y\in \Omega$. 
The suitability of the habitat of the resource is given by
\begin{eqnarray*}
K(x,y) &=&  2\exp(-5((x+.75)^2+(y-.75)^2))+2\exp(-5((x-.75)^2+(y+.75)^2)),   \\
          &  &  + 2\exp(-5((x+.75)^2+(y+.75)^2))+2\exp(-5((x-.75)^2+(y-.75)^2)) . 
\end{eqnarray*}
Let the parameter values be given by $\alpha=5, \, a= 2.0,\, b=5.0,\, d=2.0,\, \beta=1.0,\, \gamma=1.0,\, \mu=0.05,\, \nu=0.05,$  $d_0=0.1$, $d_1=1, d_2=1$. \\
The sensitivity function is $\chi_2(u,w)=q uw$. The below simulations are executed for different values of $q$ and $c.$

\begin{figure}[hbt]
\begin{tabular}{ccc}
\includegraphics[scale=0.1125]{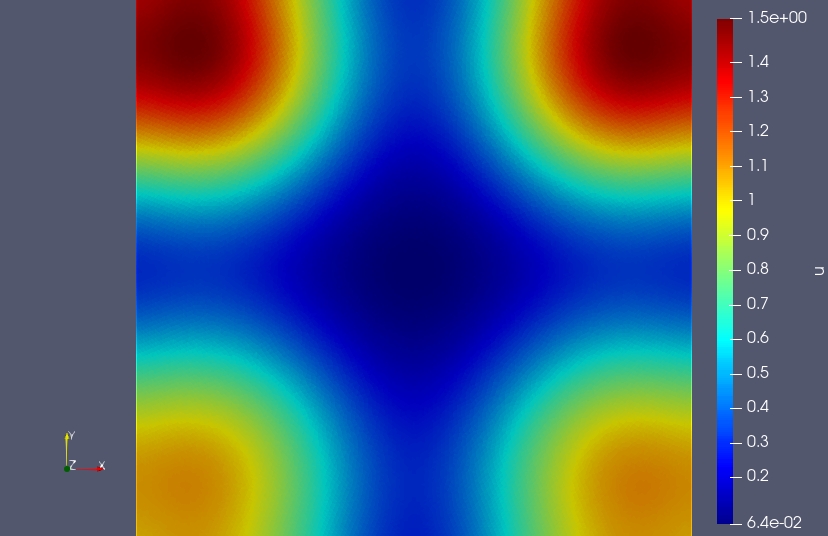} &
\includegraphics[scale=0.1125]{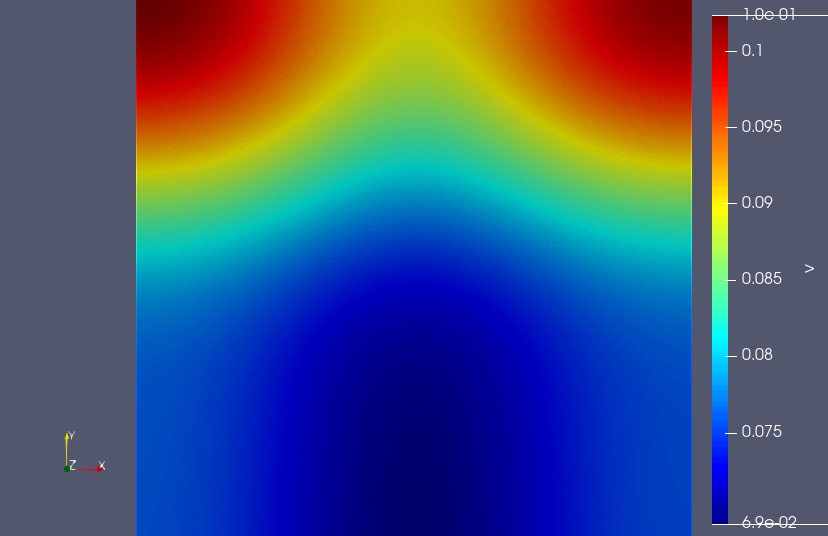} &
\includegraphics[scale=0.1125]{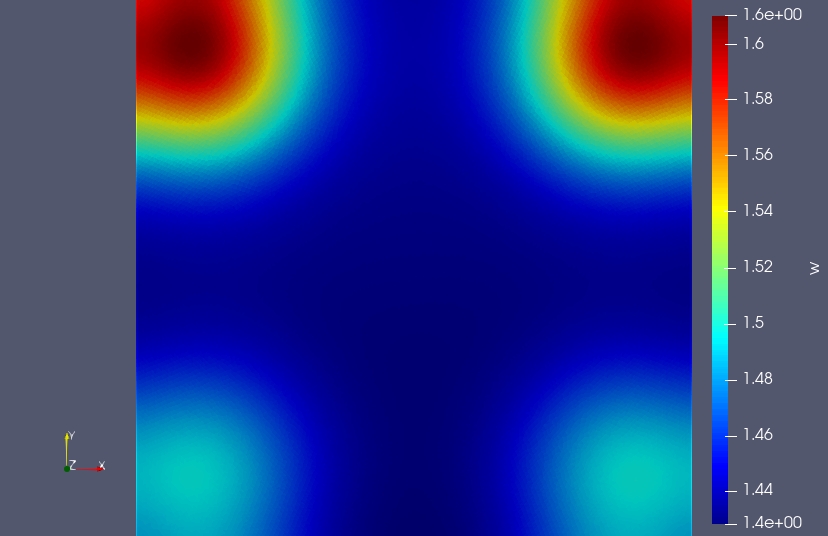} \\
($a_1$) $u,\quad t=2$ & ($b_1$) $v,\quad t=2$ & ($c_1$) $w,\quad t=2$  \\
\includegraphics[scale=0.1125]{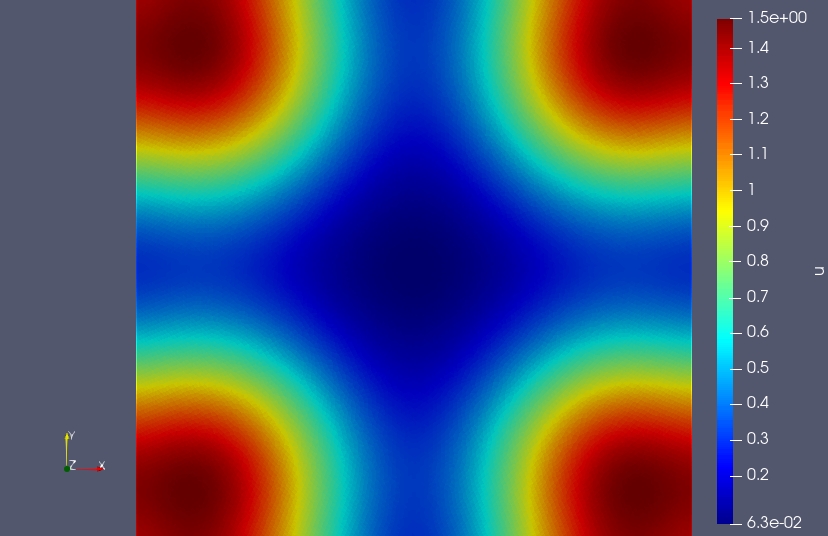} &
\includegraphics[scale=0.1125]{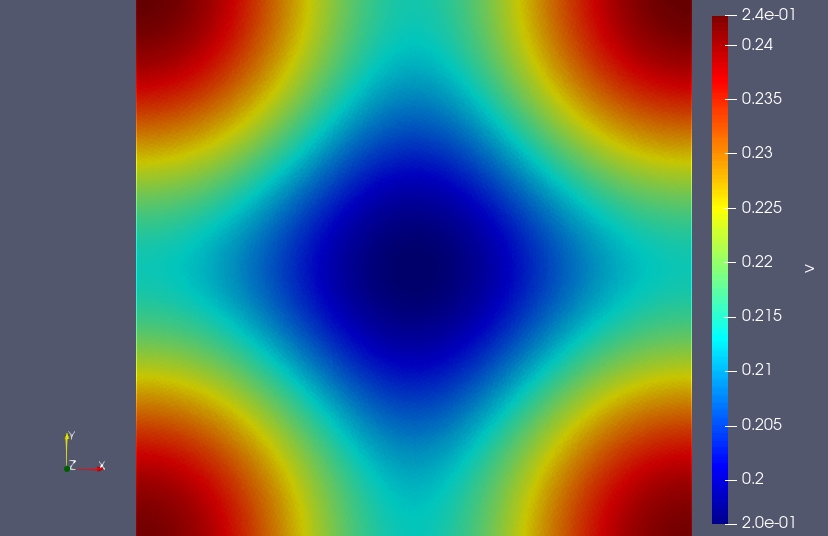} &
\includegraphics[scale=0.1125]{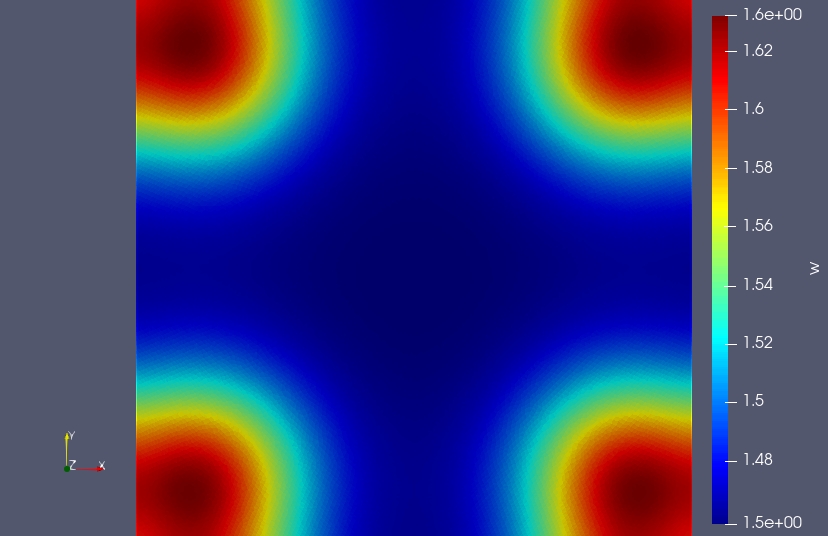} \\
($a_2$) $u,\quad t=4$ & ($b_2$) $v,\quad t=4$ & ($c_2$) $w,\quad t=4$ \\
\includegraphics[scale=0.1125]{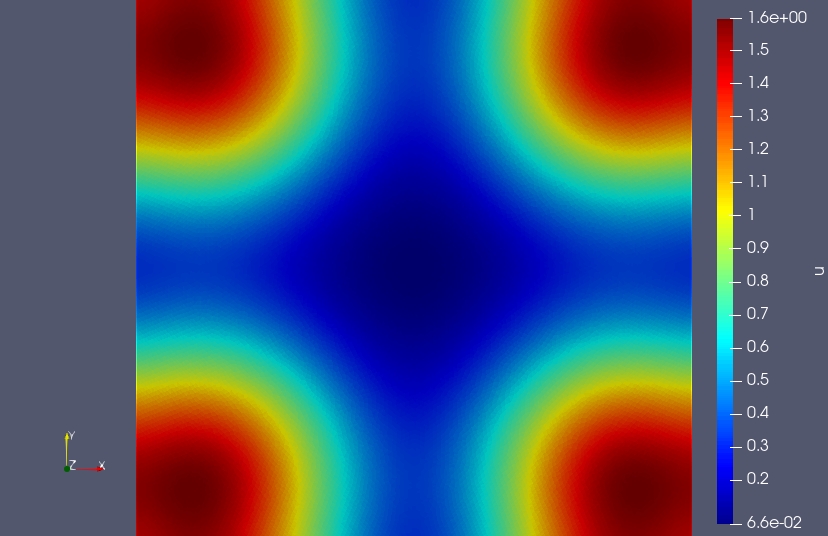} &
\includegraphics[scale=0.1125]{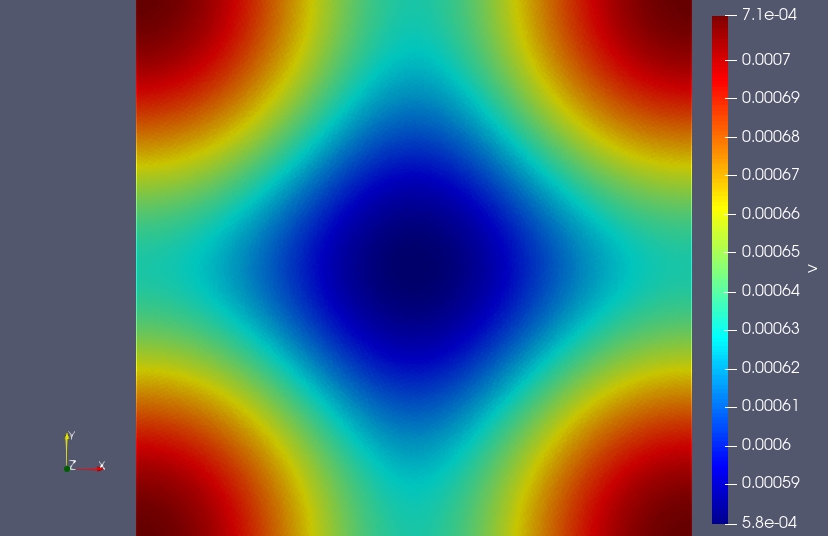} & 
\includegraphics[scale=0.1125]{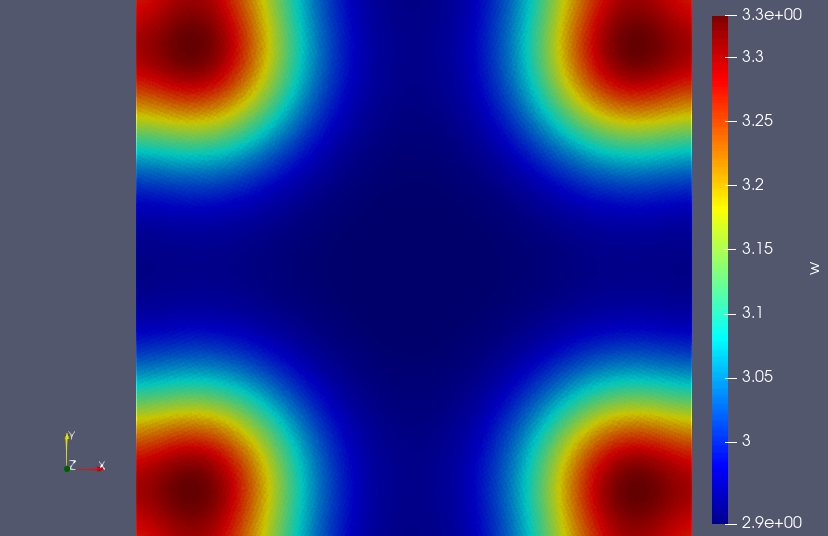} \\
($a_{3}$) $u,\quad t=20$ & ($b_{3}$) $v,\quad t=20$ & ($c_{3}$) $w,\quad t=20$ 
\end{tabular}
\caption{{\em Contour plots} of time evolution of the resource $u$,  mesopredador $v$ and top predador $w$ at different times. $ q=0.1$, $c=1.0$} \label{Figu9}
\end{figure}

%%%The sensitivity function is $\chi_1(u,w)=q uw$, $ q=0.1$, $c=1.5$
\begin{figure}[hbt]
\begin{tabular}{ccc}
\includegraphics[scale=0.1125]{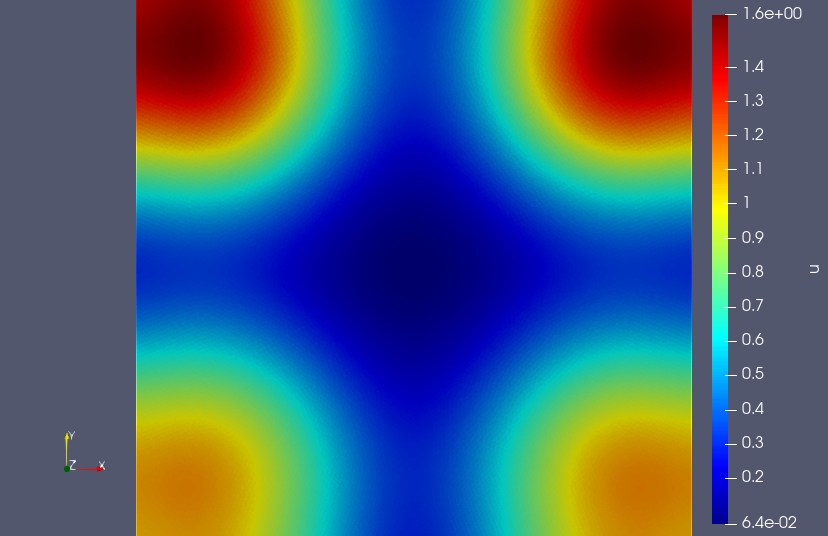} &
\includegraphics[scale=0.1125]{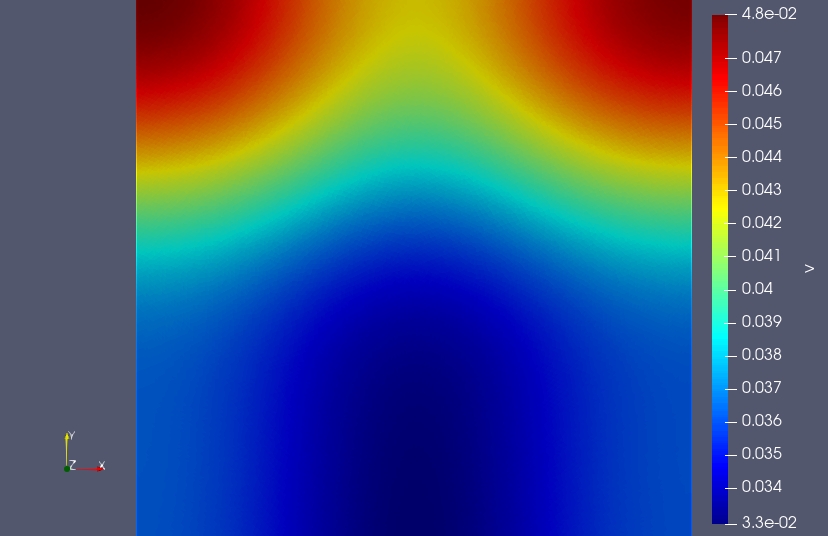} &
\includegraphics[scale=0.1125]{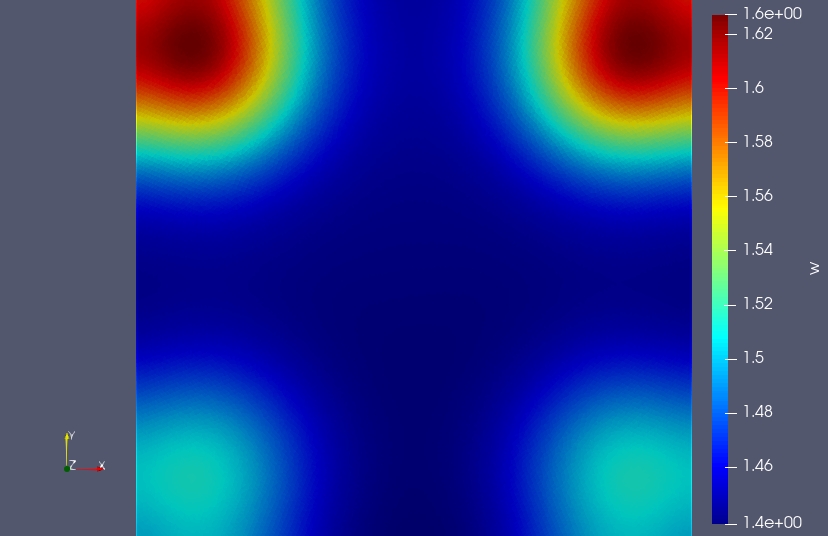} \\
($a_1$) $u,\quad t=2$ & ($b_1$) $v,\quad t=2$ & ($c_1$) $w,\quad t=2$  \\
\includegraphics[scale=0.1125]{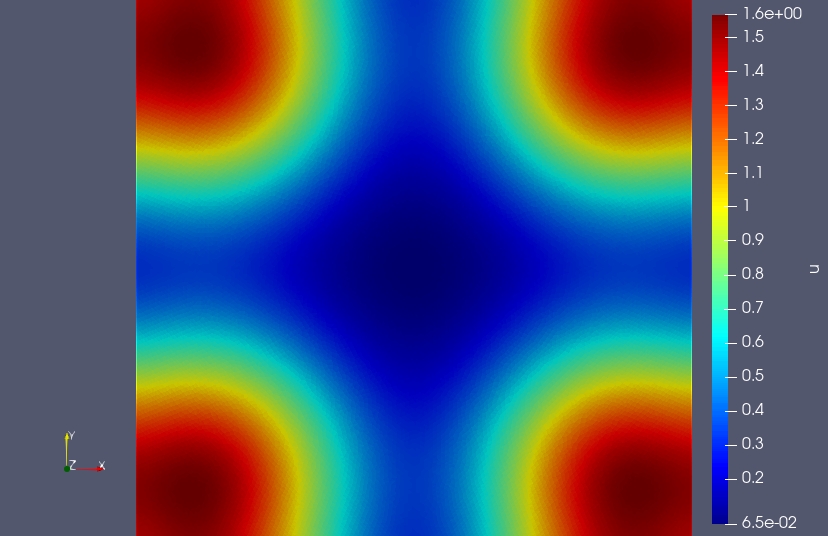} &
\includegraphics[scale=0.1125]{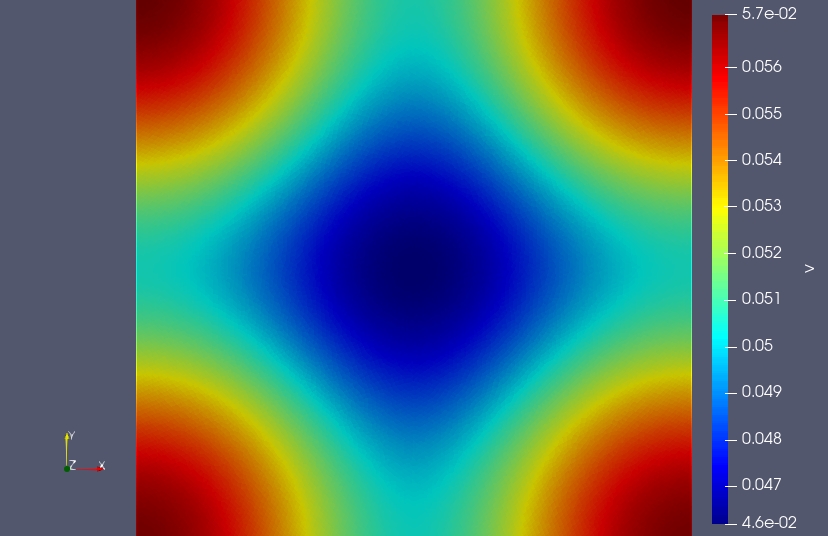} &
\includegraphics[scale=0.1125]{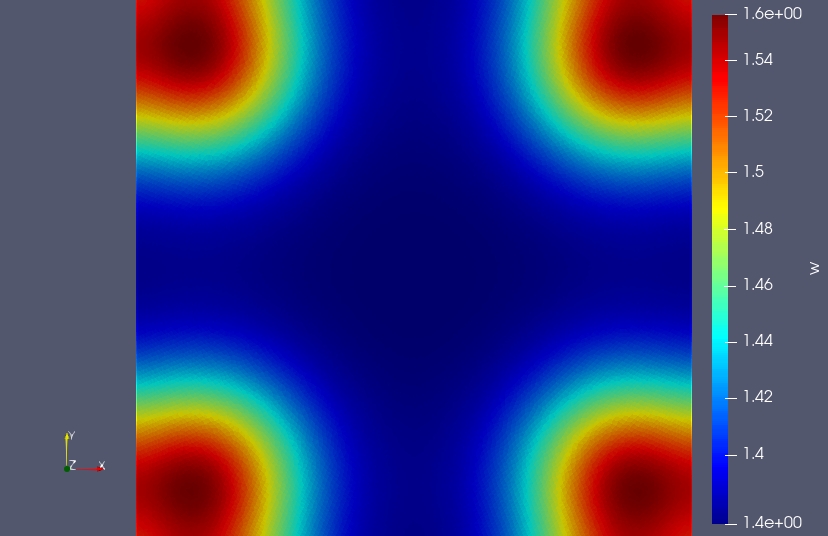} \\
($a_2$) $u,\quad t=4$ & ($b_2$) $v,\quad t=4$ & ($c_2$) $w,\quad t=4$ \\
\includegraphics[scale=0.1125]{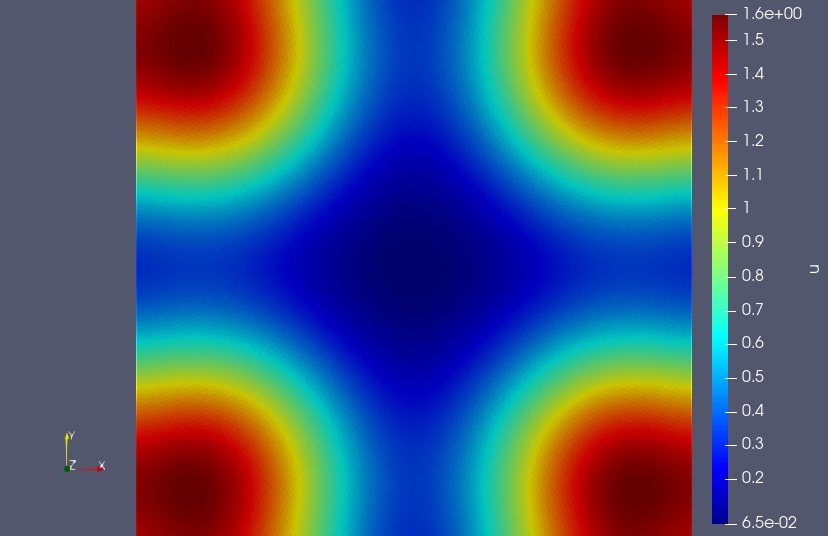} &
\includegraphics[scale=0.1125]{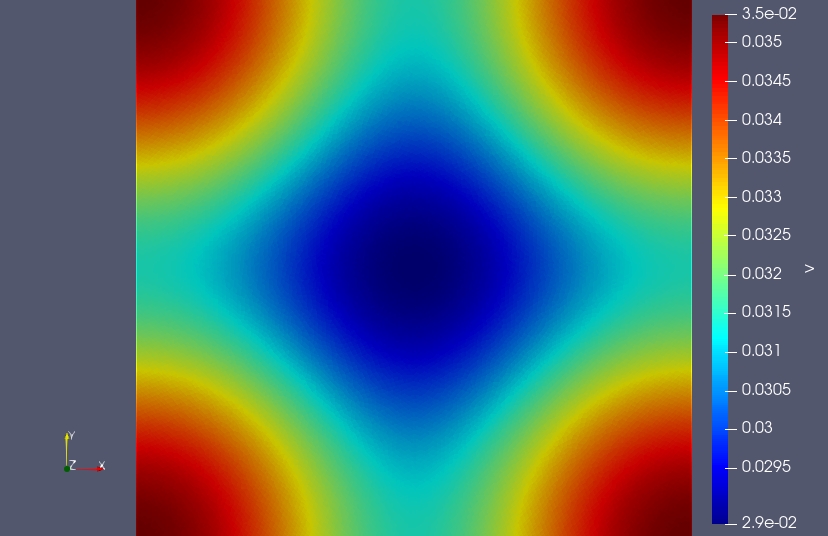} & 
\includegraphics[scale=0.1125]{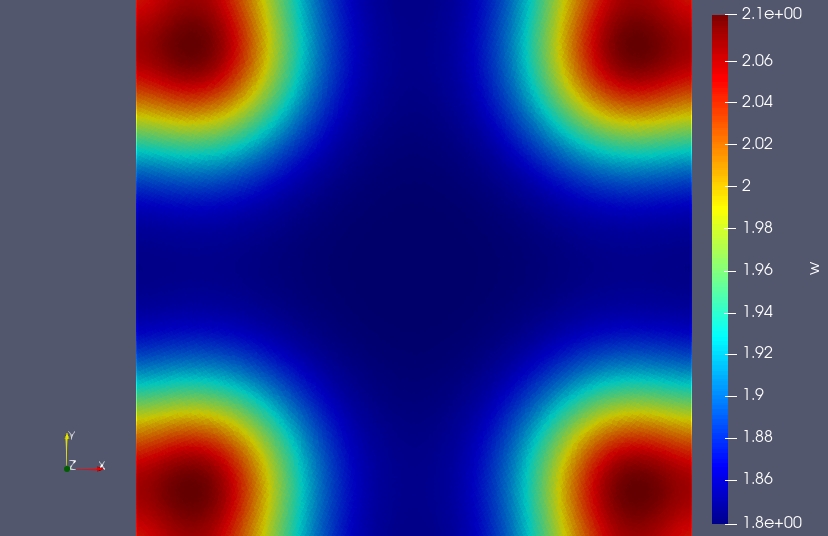} \\
($a_3$) $u,\quad t=20$ & ($b_3$) $v,\quad t=20$ & ($c_3$) $w,\quad t=20$  
\end{tabular}
\caption{{\em Contour plots} of time evolution of the resource $u$,  mesopredador $v$ and top predador $w$ at different times. $ q=0.1$, $c=1.5$}  \label{Figu10}
\end{figure}

%%%%The sensitivity function is $\chi_1(u,w)=q uw$, $ q=1.0$, $c=1.5$
\begin{figure}[hbt]
\begin{tabular}{ccc}
\includegraphics[scale=0.1125]{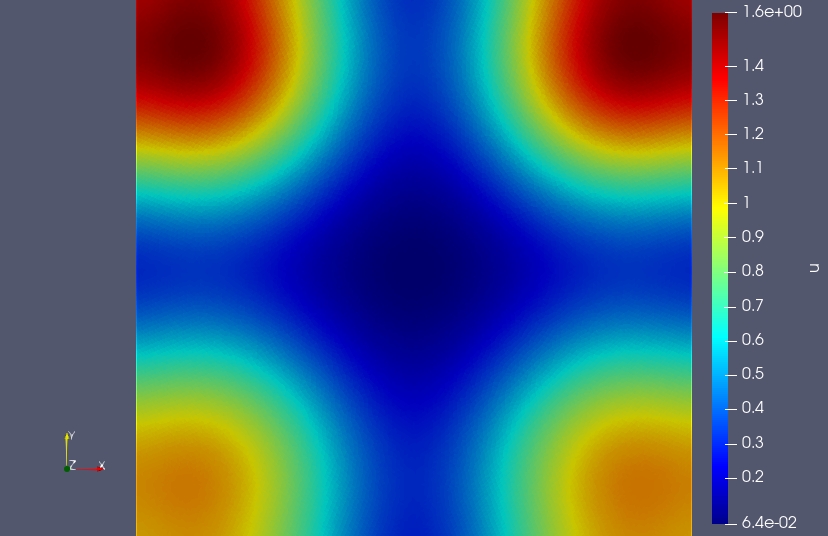} &
\includegraphics[scale=0.1125]{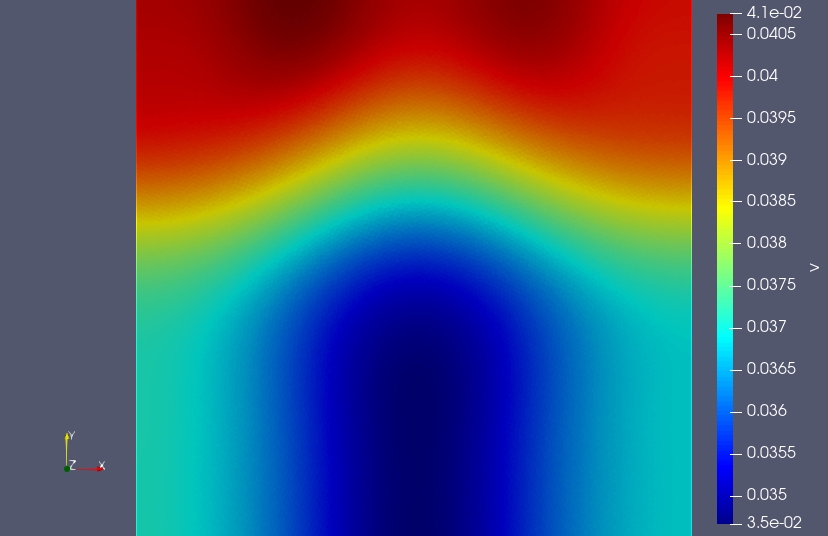} & 
\includegraphics[scale=0.1125]{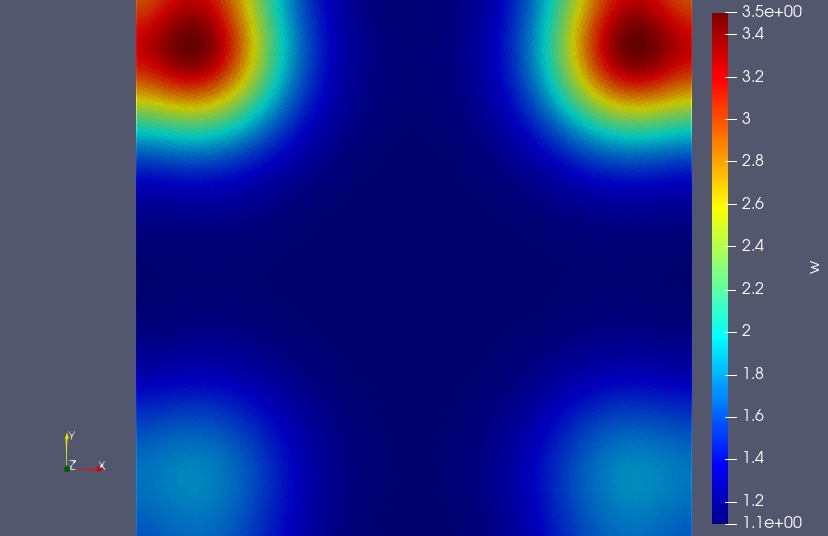} \\
($a_1$) $u,\quad t=2.0$ & ($b_1$) $v,\quad t=2.0$ & ($c_1$) $w,\quad t=2.0$ \\
\includegraphics[scale=0.1125]{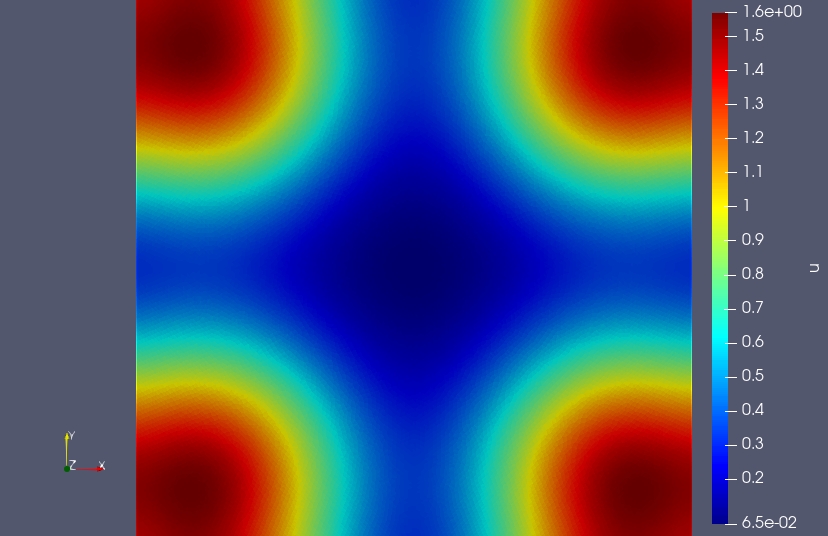} &
\includegraphics[scale=0.1125]{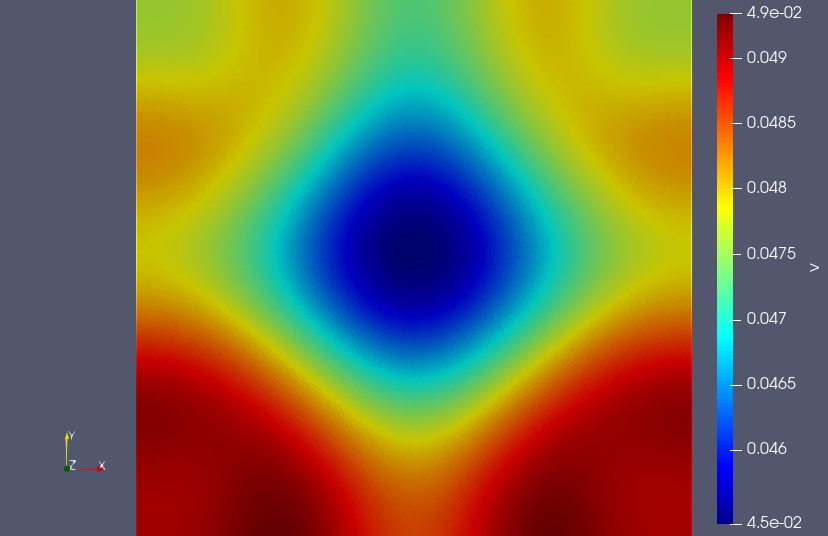} & 
\includegraphics[scale=0.1125]{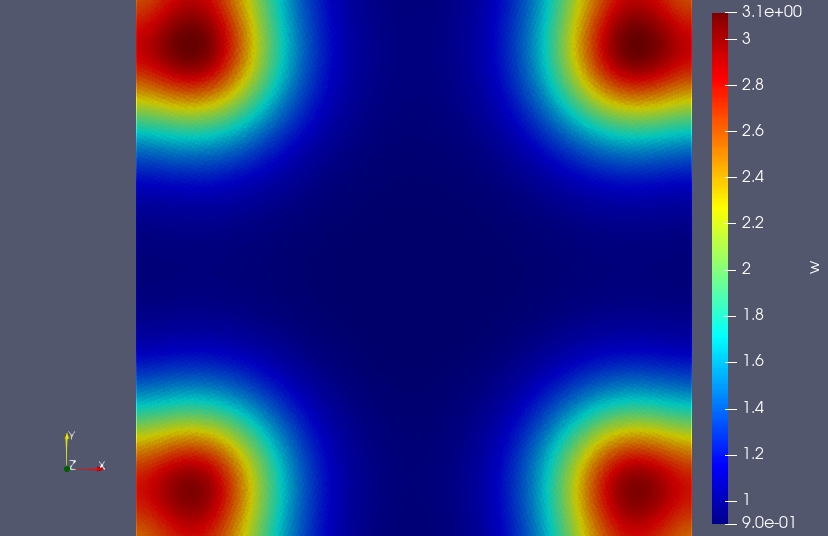} \\
($a_2$) $u,\quad t=4.0$ & ($b_2$) $v,\quad t=4.0$ & ($c_2$) $w,\quad t=4.0$ \\
\includegraphics[scale=0.1125]{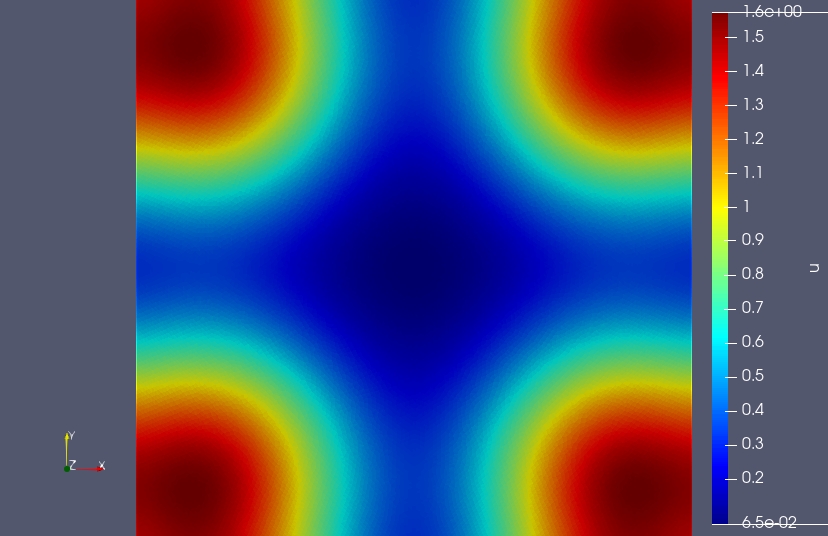} &
\includegraphics[scale=0.1125]{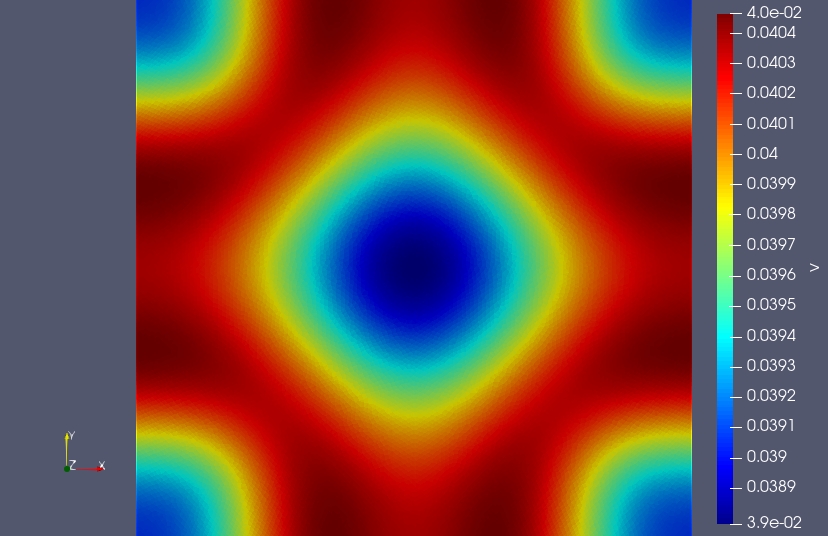} & 
\includegraphics[scale=0.1125]{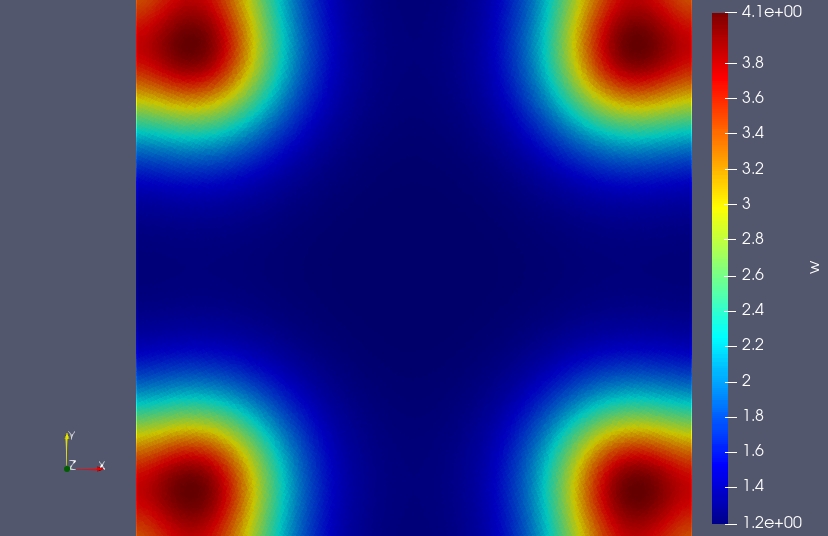} \\
($a_3$) $u,\quad t=20.0$ & ($b_3$) $v,\quad t=20.0$ & ($c_3$) $w,\quad t=20.0$ 
\end{tabular}
\caption{{\em Contour plots} of time evolution of the resource $u$,  mesopredador $v$ and top predador $w$ at different times. $ q=1.0$, $c=1.5$}  \label{Figu11}
\end{figure}

%%%\newpage
%%%%The sensitivity function is $\chi_1(u,w)=q uw$, $ q=1.0$, $c=2.5$

\begin{figure}[hbt]
\begin{tabular}{ccc}
\includegraphics[scale=0.1125]{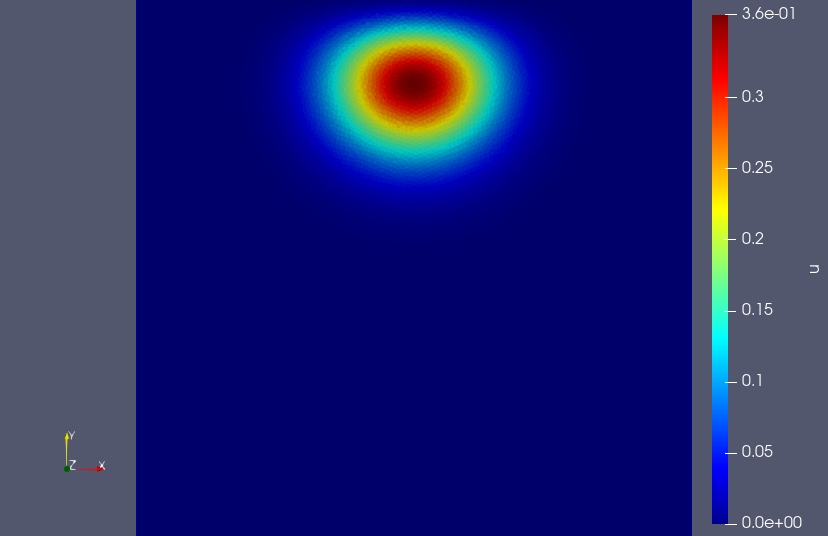} &
\includegraphics[scale=0.1125]{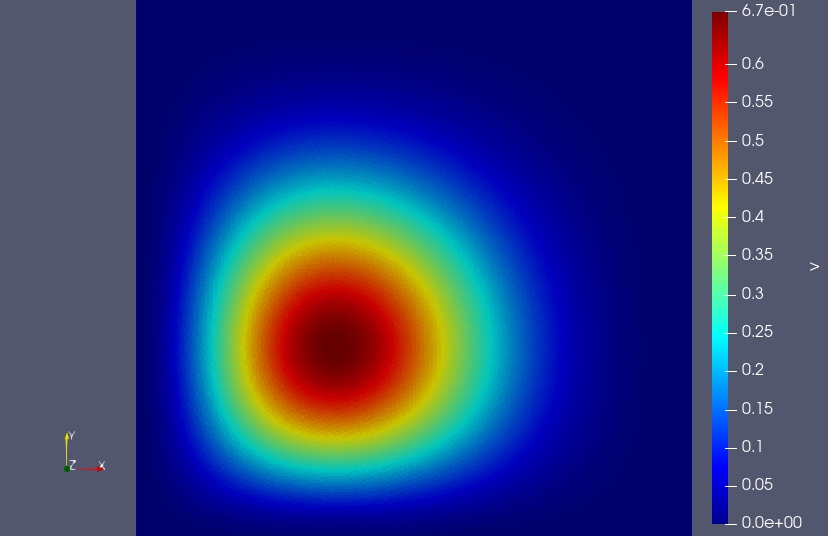} &
\includegraphics[scale=0.1125]{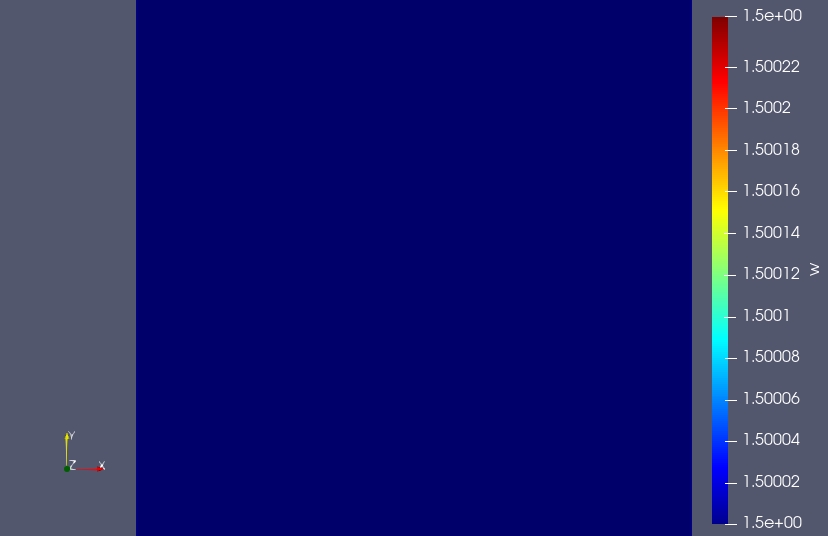} \\
($a_1$) $u,\quad t=0$ & ($b_1$) $v,\quad t=0$ & ($c_1$) $w,\quad t=0$  \\
\includegraphics[scale=0.1125]{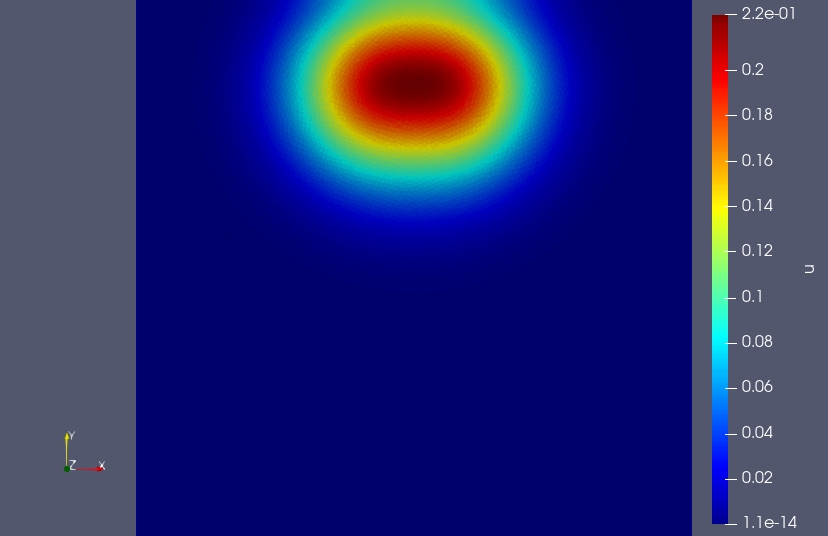} &
\includegraphics[scale=0.1125]{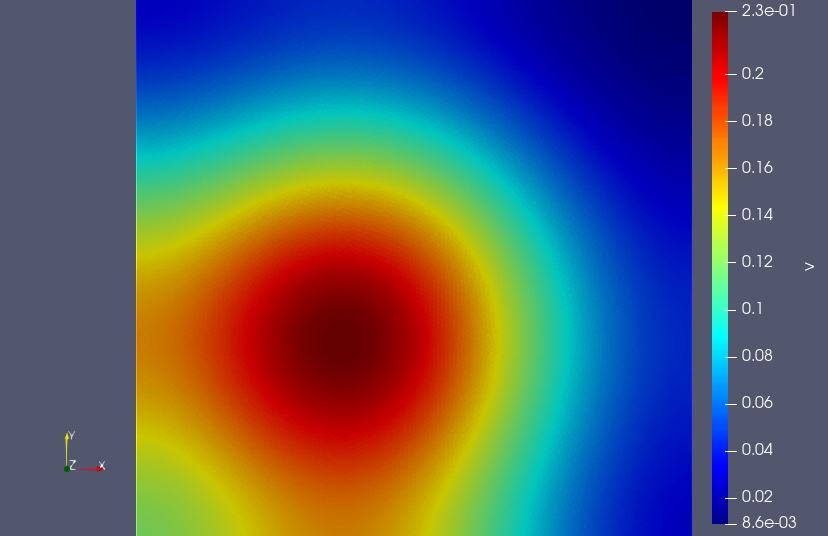} &
\includegraphics[scale=0.1125]{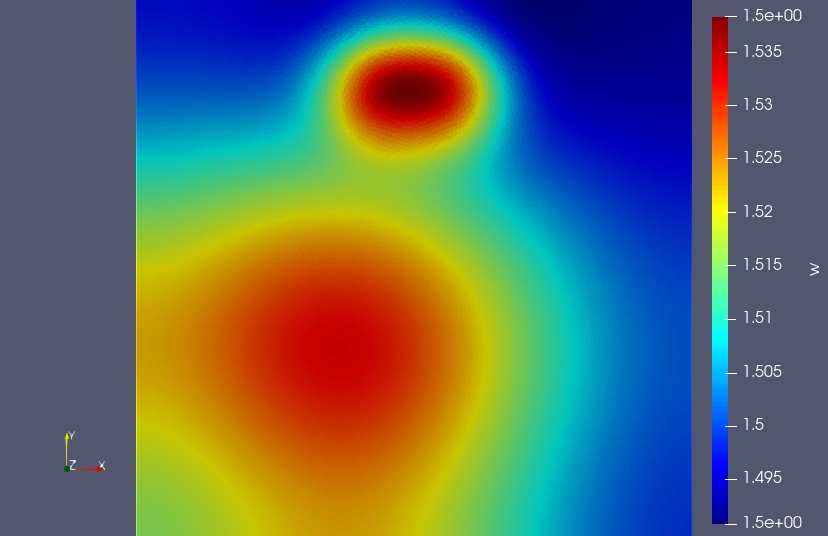} \\
($a_2$) $u,\quad t=0.1$ & ($b_2$) $v,\quad t=0.1$ & ($c_2$) $w,\quad t=0.1$ \\
\includegraphics[scale=0.1125]{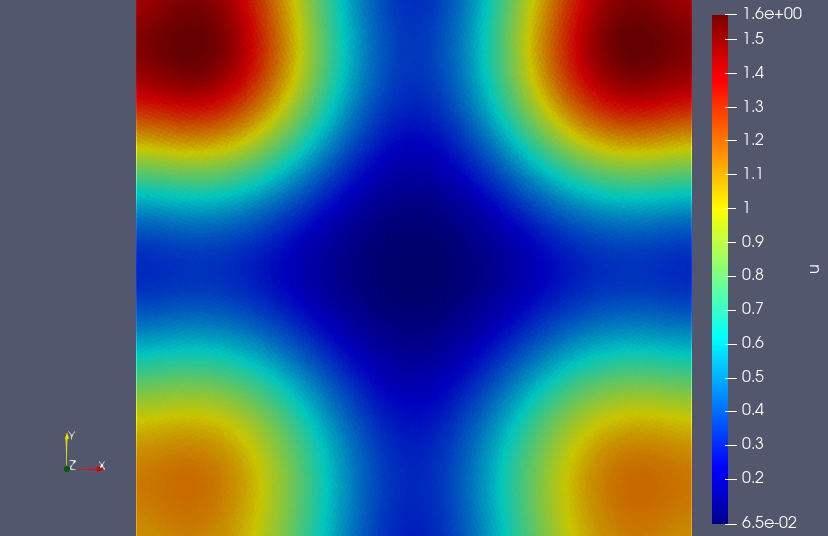} &
\includegraphics[scale=0.1125]{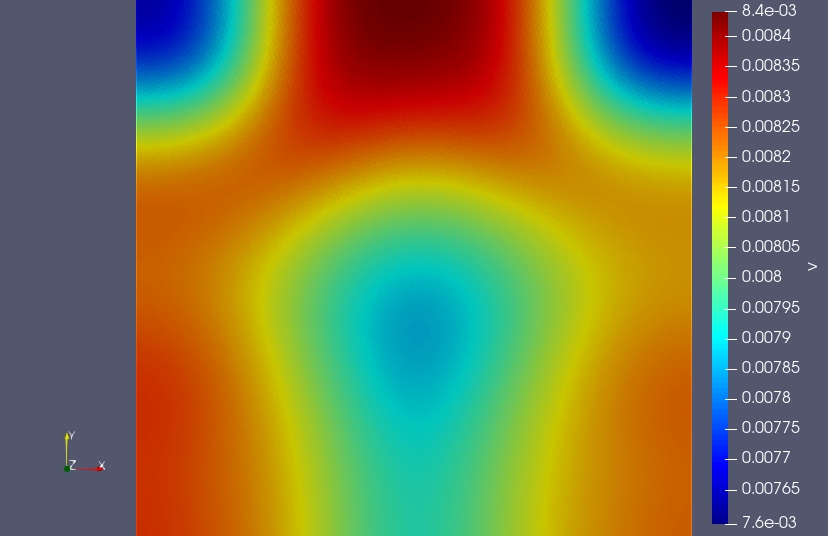} & 
\includegraphics[scale=0.1125]{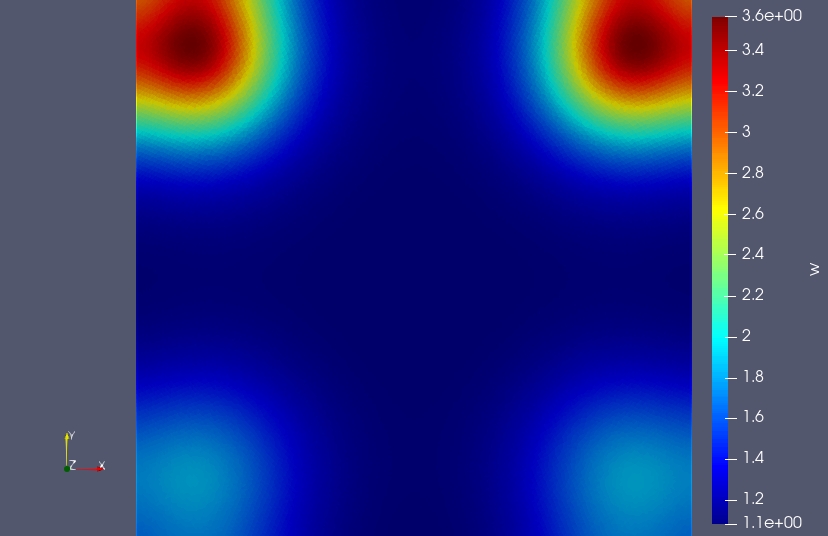} \\
($a_3$) $u,\quad t=2.0$ & ($b_3$) $v,\quad t=2.0$ & ($c_3$) $w,\quad t=2.0$ \\
\includegraphics[scale=0.1125]{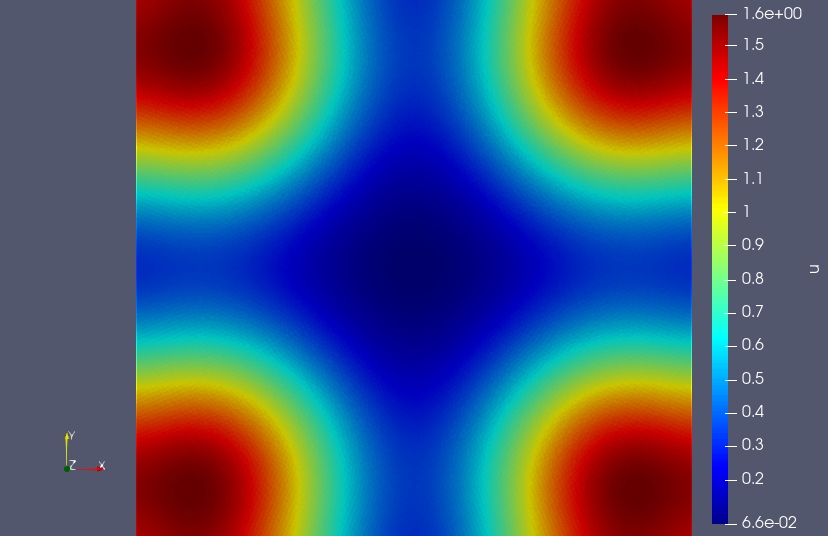} &
\includegraphics[scale=0.1125]{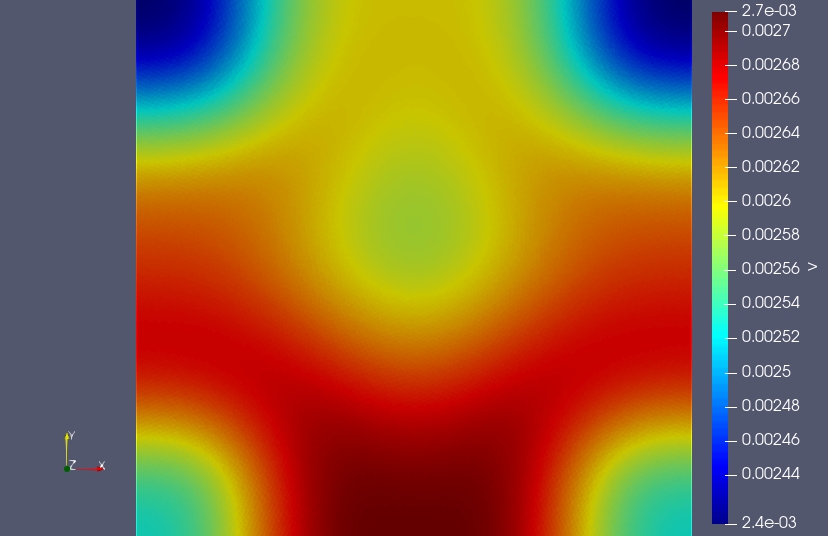} & 
\includegraphics[scale=0.1125]{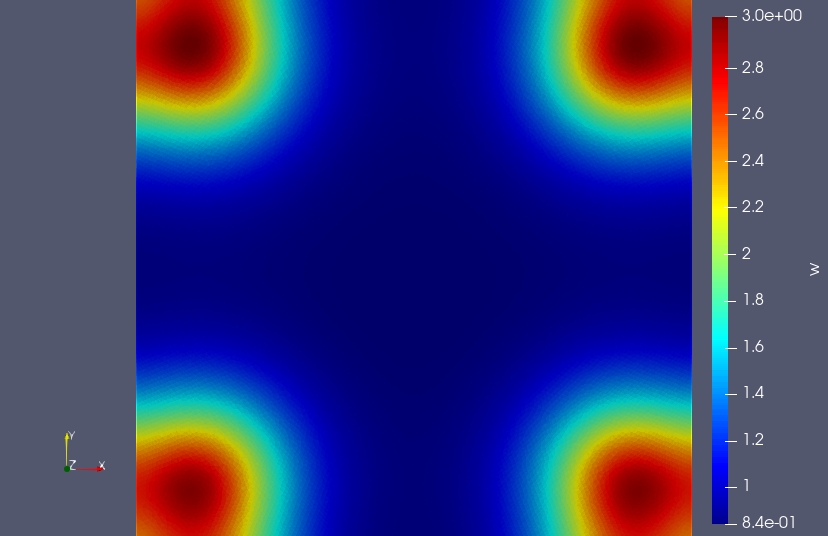} \\
($a_4$) $u,\quad t=4.0$ & ($b_4$) $v,\quad t=4.0$ & ($c_4$) $w,\quad t=4.0$ \\
\includegraphics[scale=0.1125]{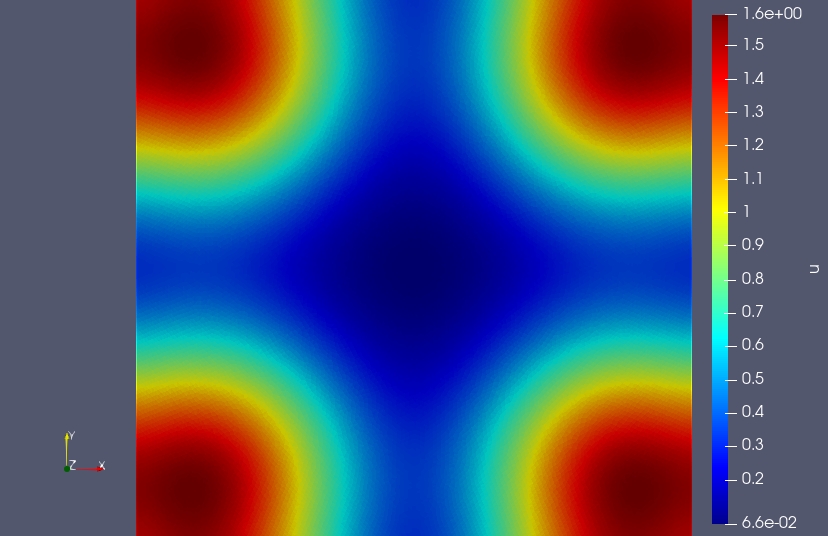} &
\includegraphics[scale=0.1125]{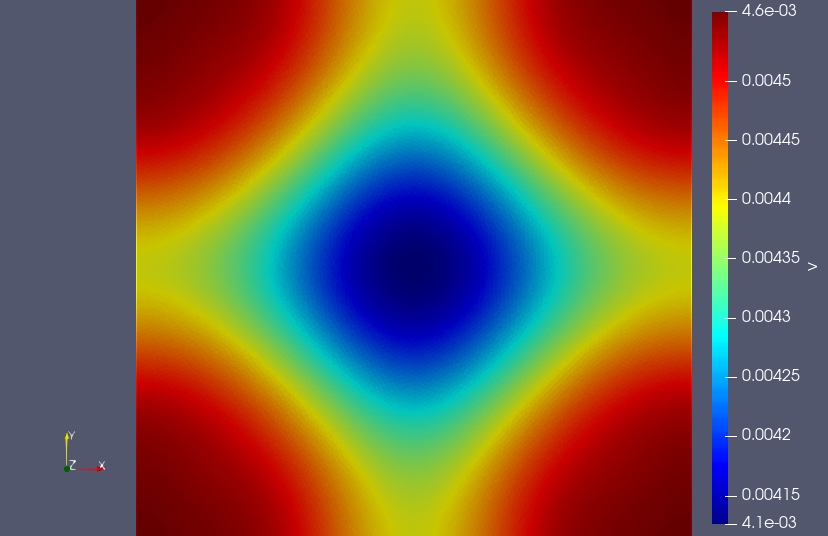} & 
\includegraphics[scale=0.1125]{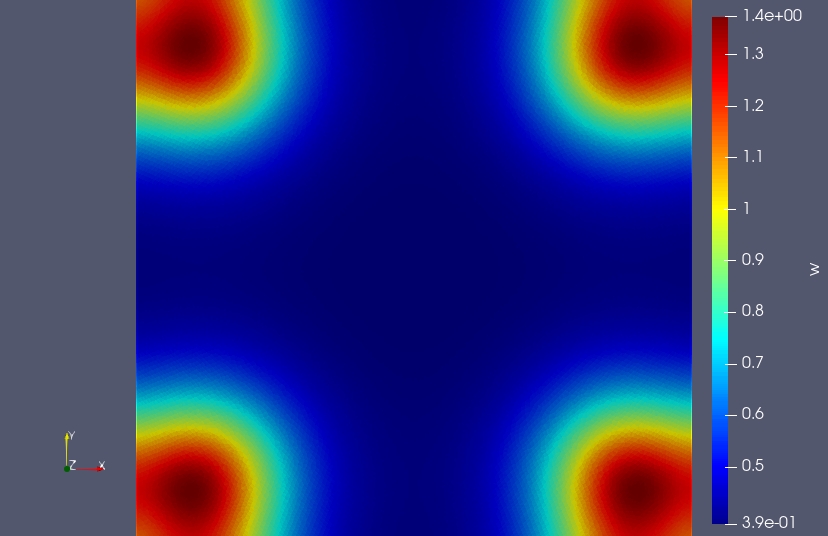} \\
($a_5$) $u,\quad t=20.0$ & ($b_5$) $v,\quad t=20.0$ & ($c_5$) $w,\quad t=20.0$ 
\end{tabular}
\caption{{\em Contour plots} of time evolution of the resource $u$,  mesopredador $v$ and top predador $w$ at different times. $ q=1.0$, $c=2.5$}  \label{Figu12}
\end{figure}

%
%\newpage
%The sensitivity function is $\chi_1(u,w)=q uw$, $ q=10.0$, $c=.1$

\begin{figure}[hbt]
\begin{tabular}{ccc}
\includegraphics[scale=0.1125]{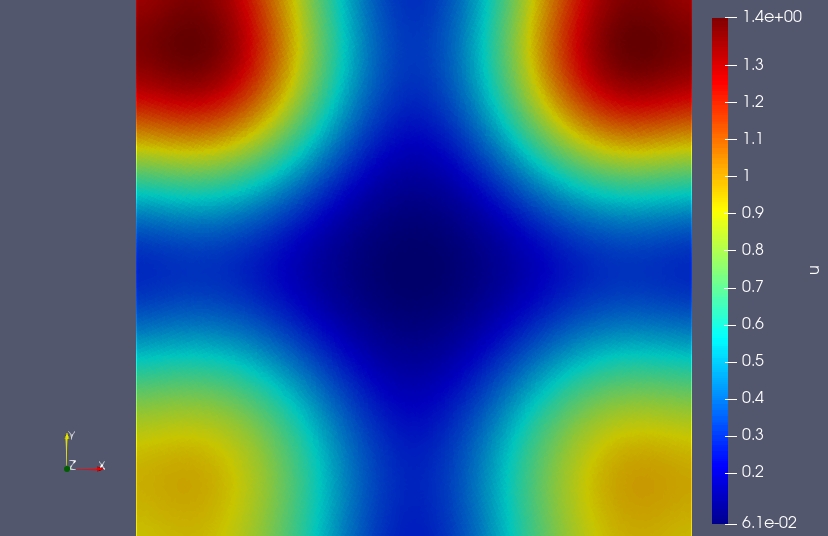} &
\includegraphics[scale=0.1125]{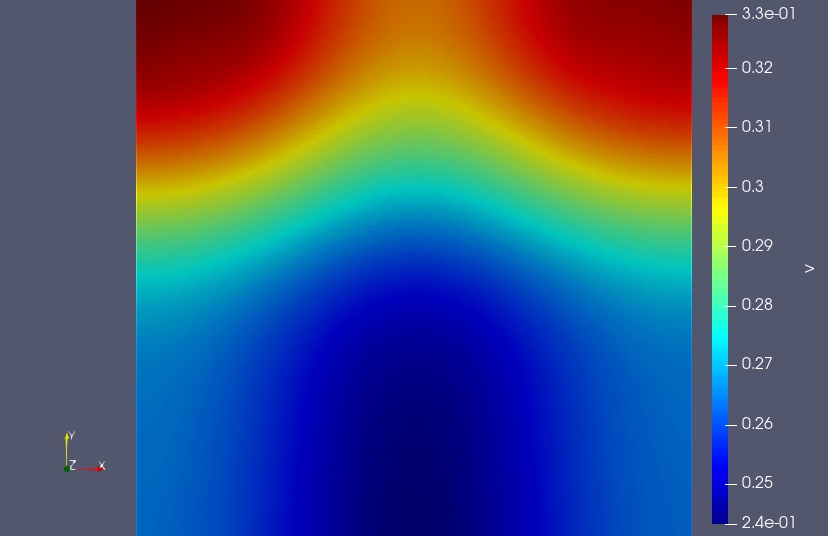} & 
\includegraphics[scale=0.1125]{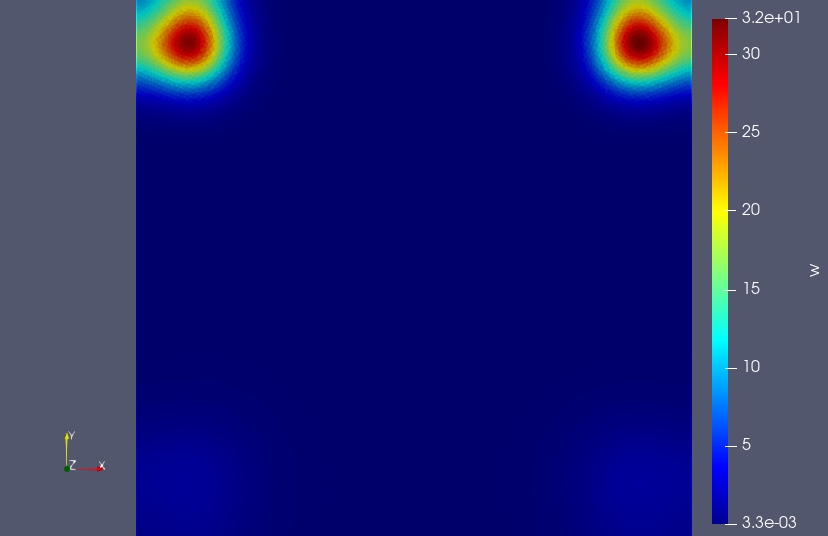} \\
($a_1$) $u,\quad t=2.0$ & ($b_1$) $v,\quad t=2.0$ & ($c_1$) $w,\quad t=2.0$ \\
\includegraphics[scale=0.1125]{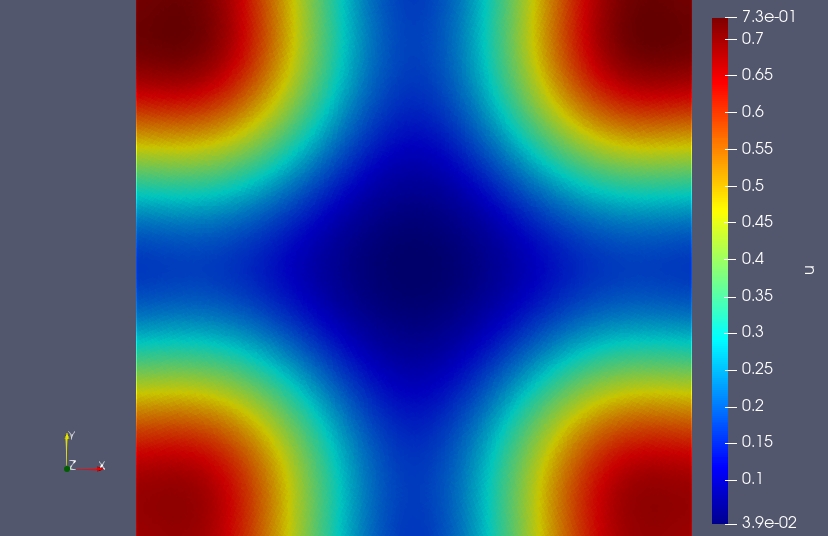} &
\includegraphics[scale=0.1125]{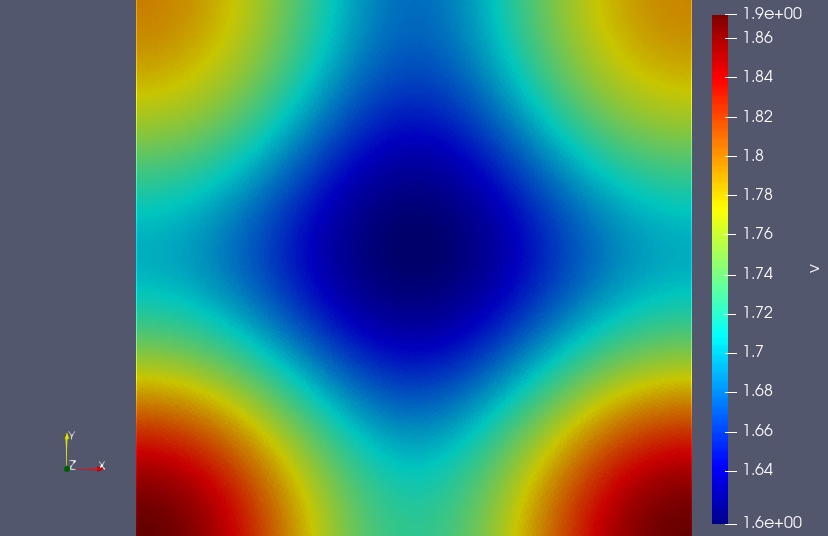} & 
\includegraphics[scale=0.1125]{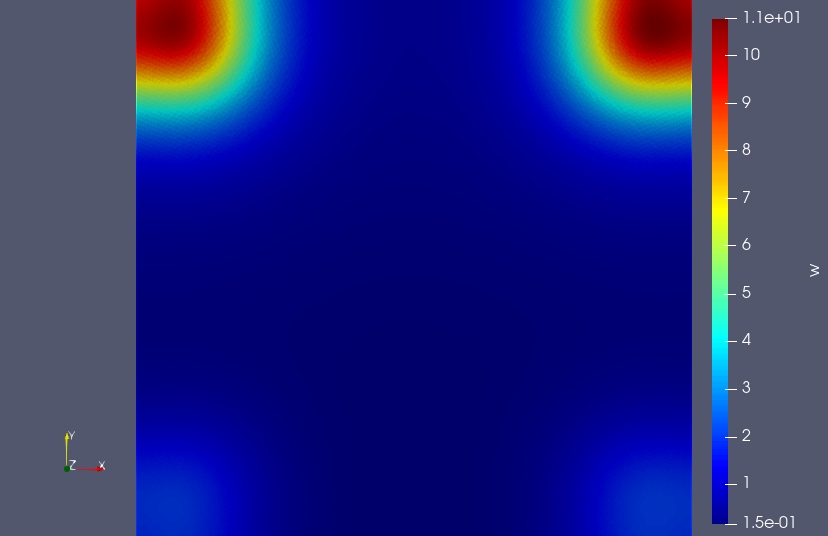} \\
($a_2$) $u,\quad t=4.0$ & ($b_2$) $v,\quad t=4.0$ & ($c_2$) $w,\quad t=4.0$ \\
\includegraphics[scale=0.1125]{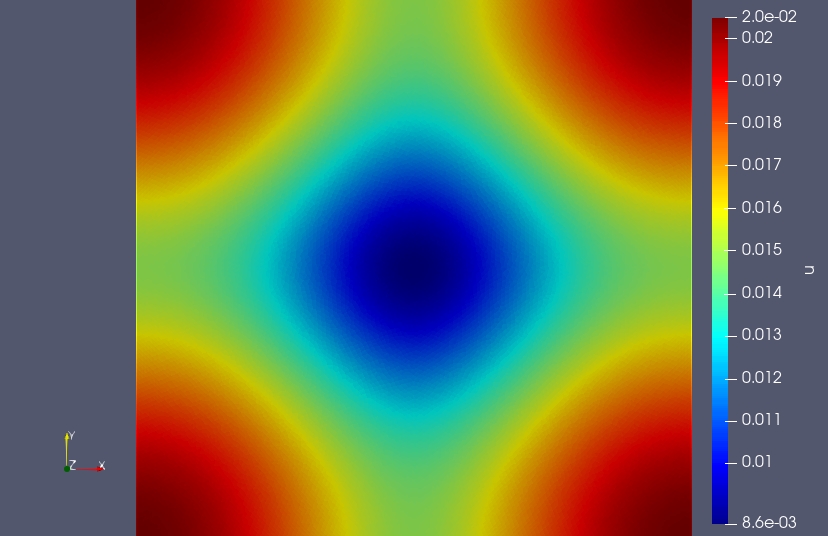} &
\includegraphics[scale=0.1125]{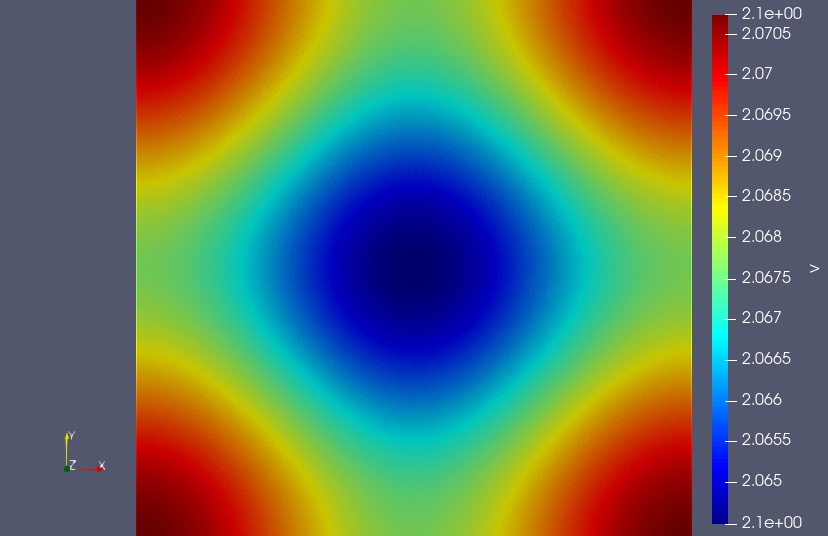} & 
\includegraphics[scale=0.1125]{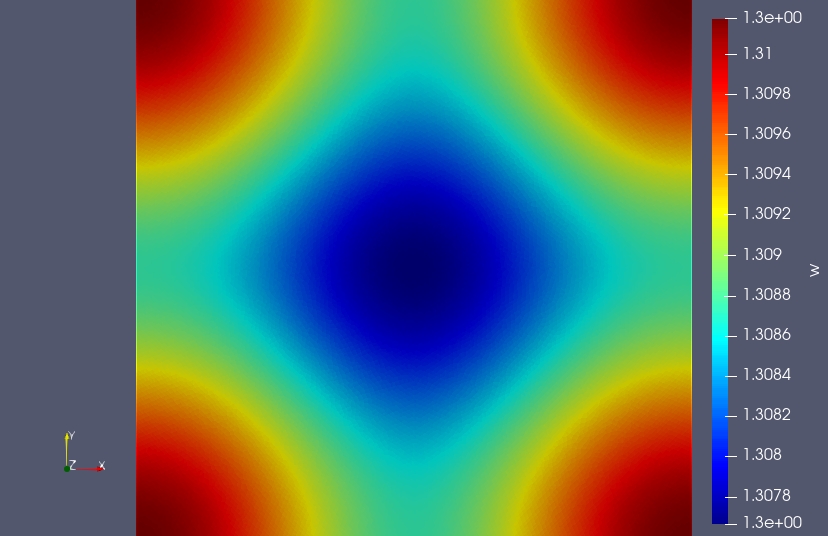} \\
($a_3$) $u,\quad t=20.0$ & ($b_3$) $v,\quad t=20.0$ & ($c_3$) $w,\quad t=20.0$ 
\end{tabular}
\caption{{\em Contour plots} of time evolution of the resource $u$,  mesopredador $v$ and top predador $w$ at different times. $ q=10.0$, $c=.1$}  \label{Figu13}
\end{figure}

%\newpage
%The sensitivity function is $\chi_1(u,w)=q uw$, $ q=10.0$, $c=1.0$
\begin{figure}[hbt]
\begin{tabular}{ccc}
\includegraphics[scale=0.1125]{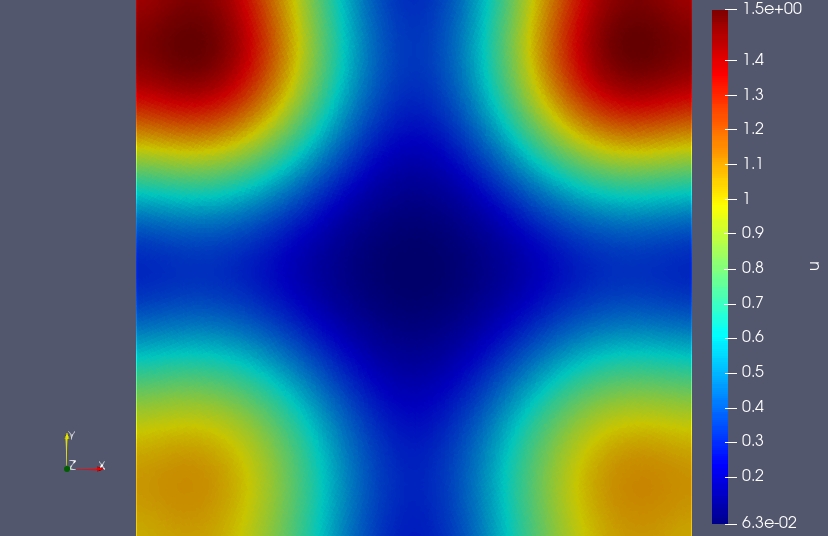} &
\includegraphics[scale=0.1125]{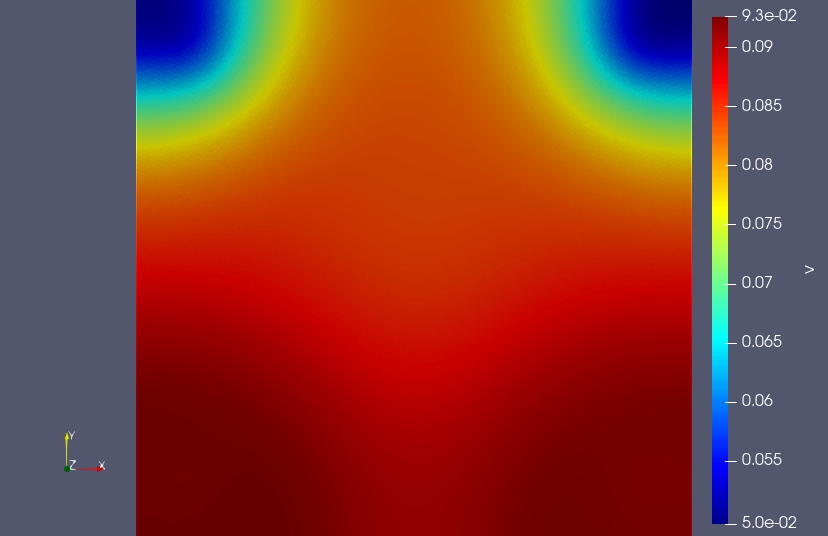} & 
\includegraphics[scale=0.1125]{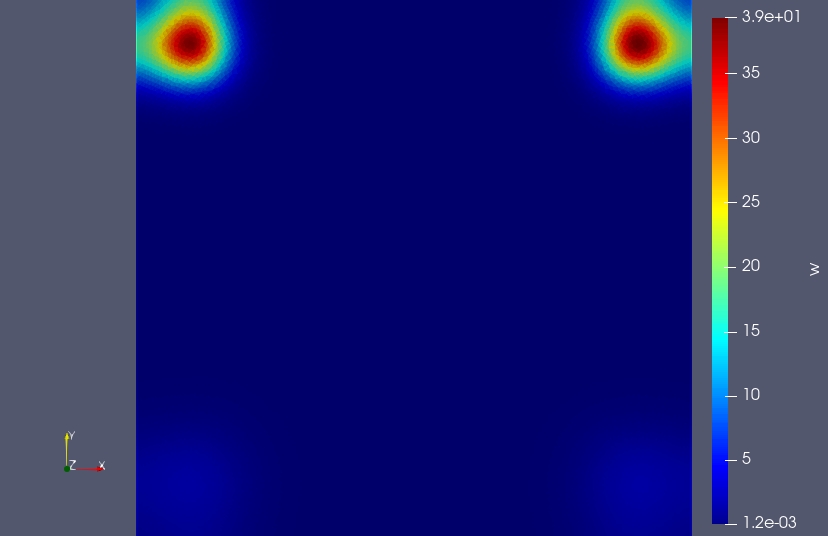} \\
($a_1$) $u,\quad t=2.0$ & ($b_1$) $v,\quad t=2.0$ & ($c_1$) $w,\quad t=2.0$ \\
\includegraphics[scale=0.1125]{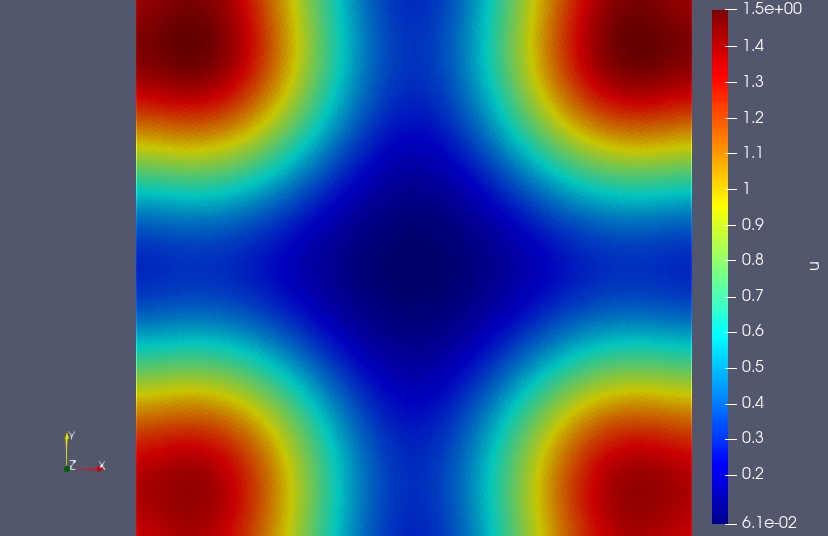} &
\includegraphics[scale=0.1125]{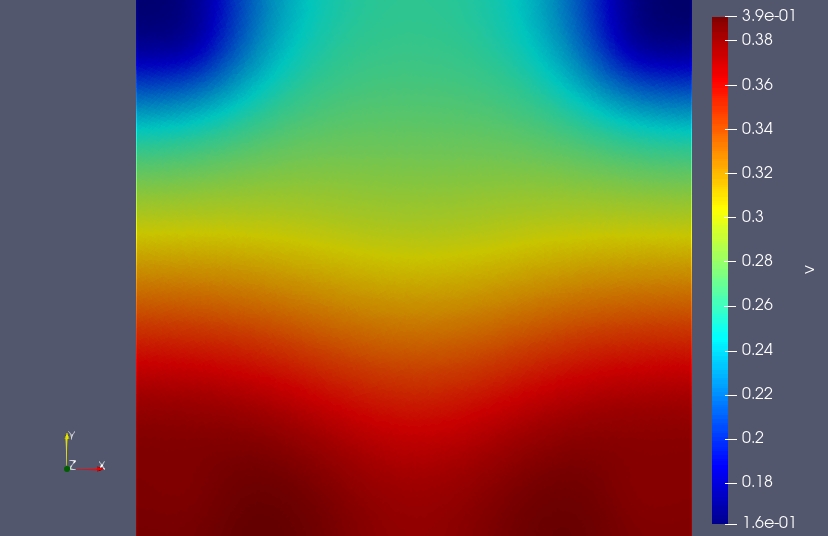} & 
\includegraphics[scale=0.1125]{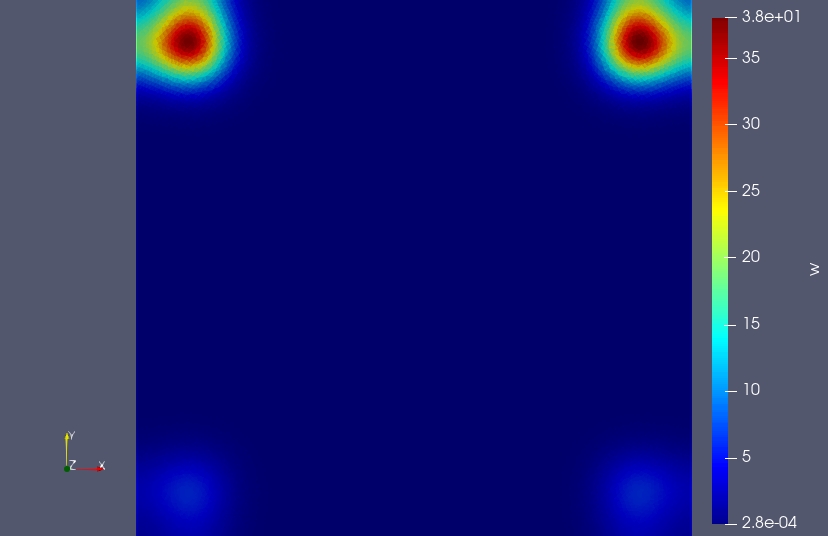} \\
($a_2$) $u,\quad t=4.0$ & ($b_2$) $v,\quad t=4.0$ & ($c_2$) $w,\quad t=4.0$ \\
\includegraphics[scale=0.1125]{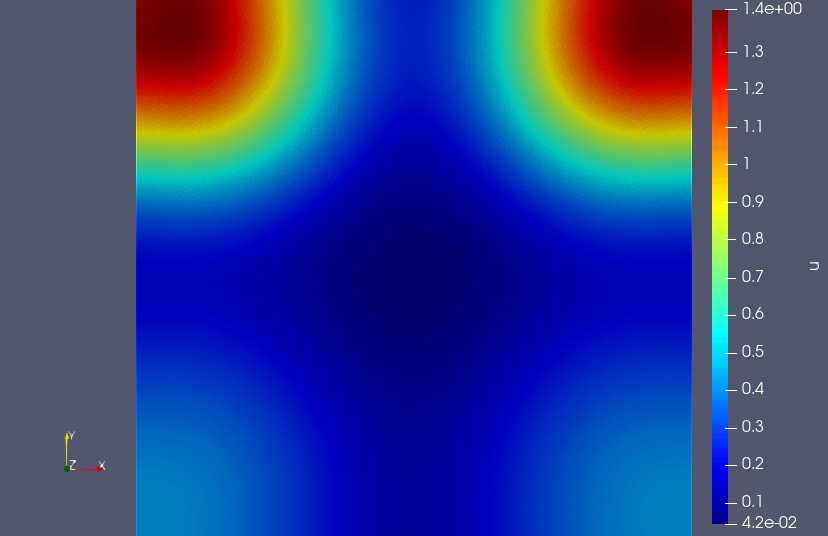} &
\includegraphics[scale=0.1125]{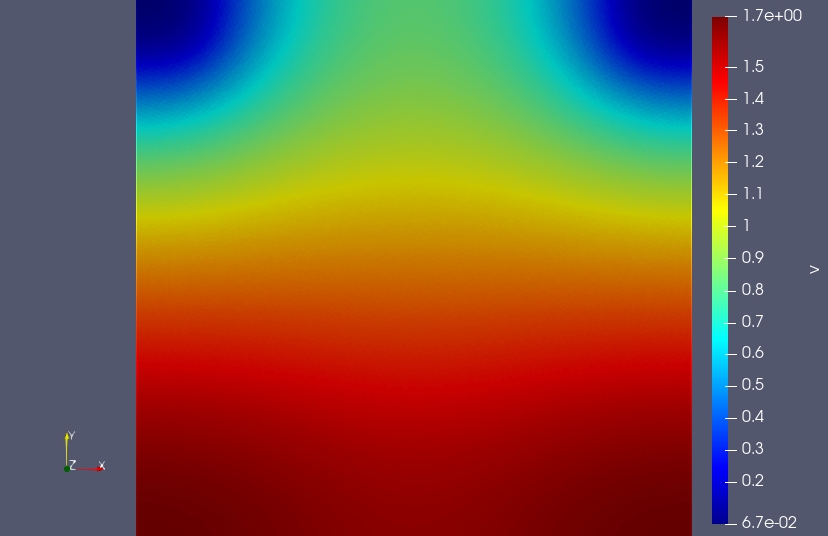} & 
\includegraphics[scale=0.1125]{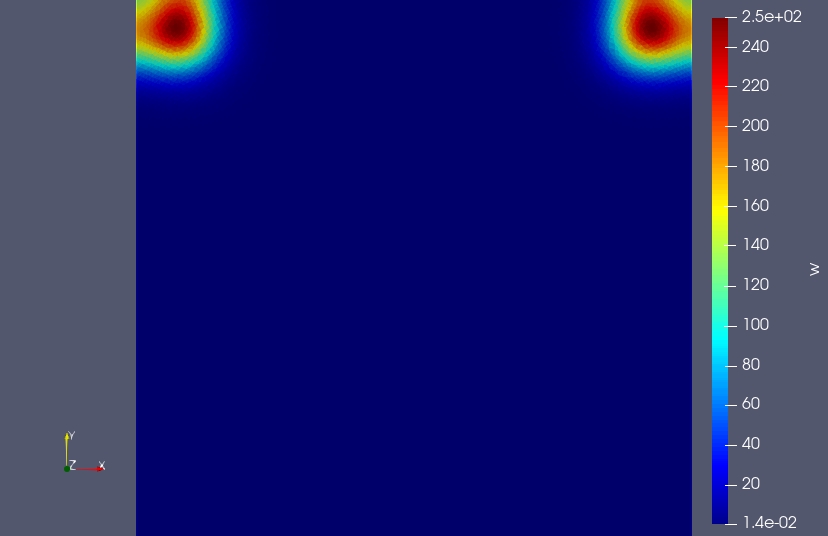} \\
($a_3$) $u,\quad t=20.0$ & ($b_3$) $v,\quad t=20.0$ & ($c_3$) $w,\quad t=20.0$ 
\end{tabular}
\caption{{\em Contour plots} of time evolution of the resource $u$,  mesopredador $v$ and top predador $w$ at different times. $ q=10.0$, $c=1.0$} \label{Figu14}
\end{figure}

%\newpage
%The sensitivity function is $\chi_1(u,w)=q uw$, $ q=10.0$, $c=1.5$
\begin{figure}[hbt]
\begin{tabular}{ccc}\
\includegraphics[scale=0.1125]{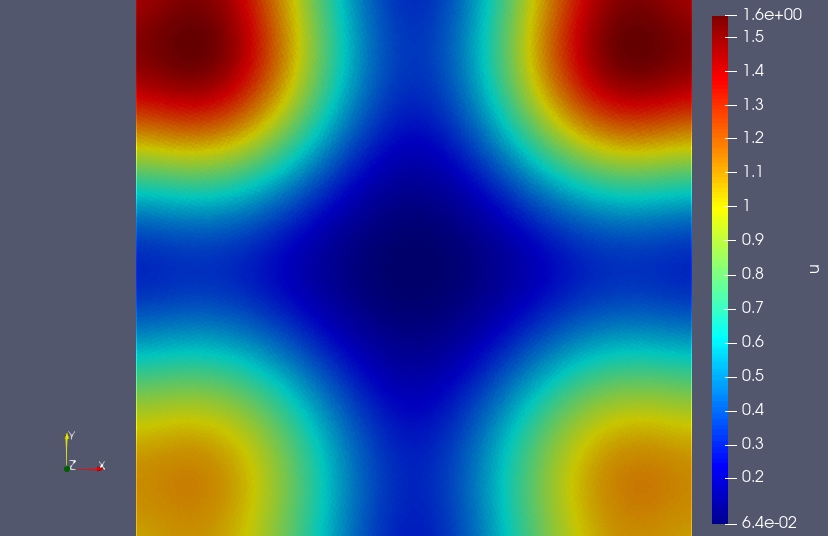} &
\includegraphics[scale=0.1125]{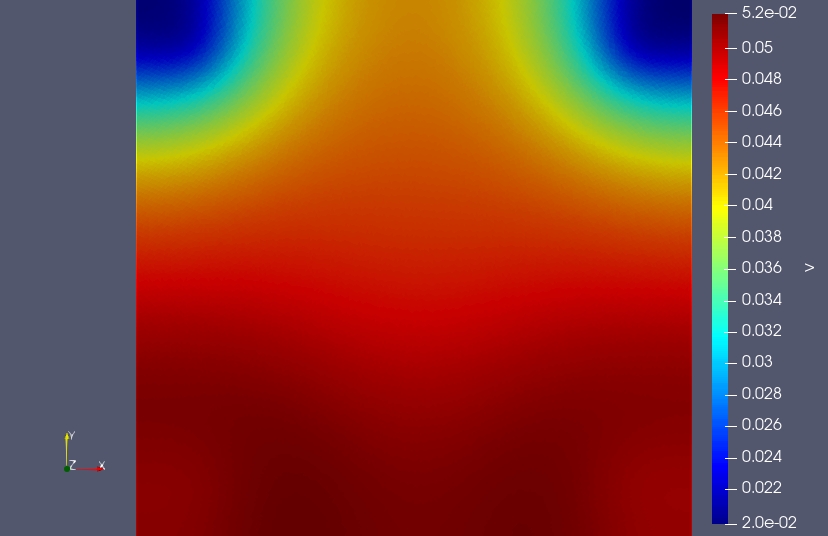} & 
\includegraphics[scale=0.1125]{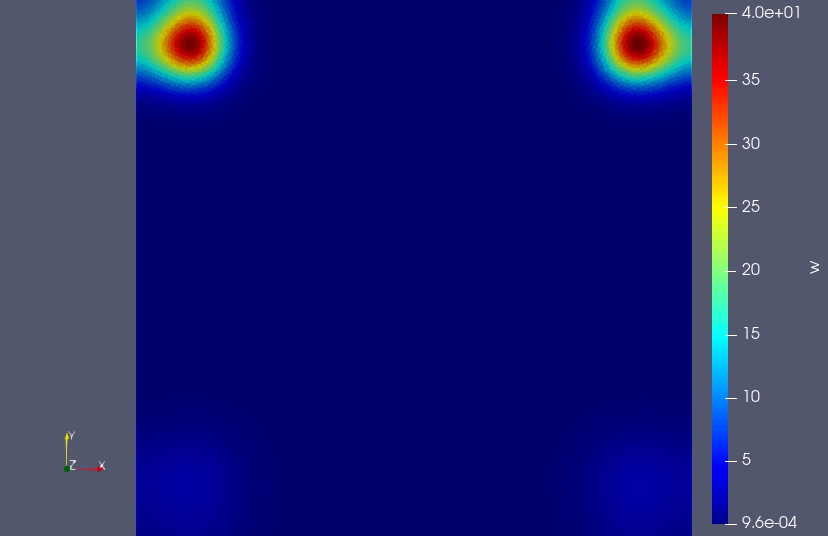} \\
($a_1$) $u,\quad t=2.0$ & ($b_1$) $v,\quad t=2.0$ & ($c_1$) $w,\quad t=2.0$ \\
\includegraphics[scale=0.1125]{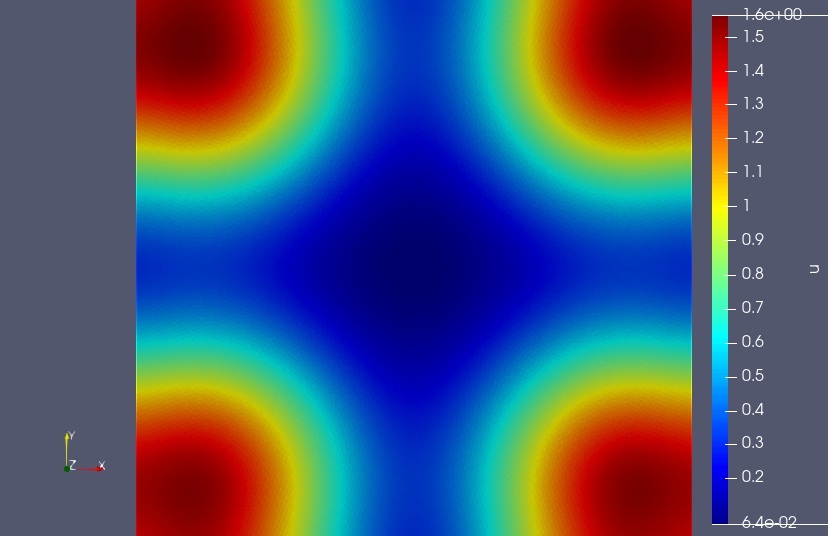} &
\includegraphics[scale=0.1125]{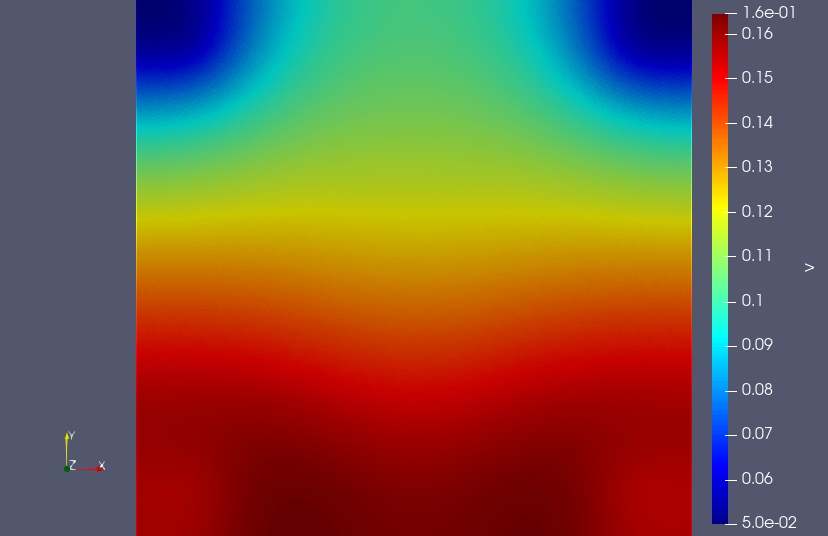} & 
\includegraphics[scale=0.1125]{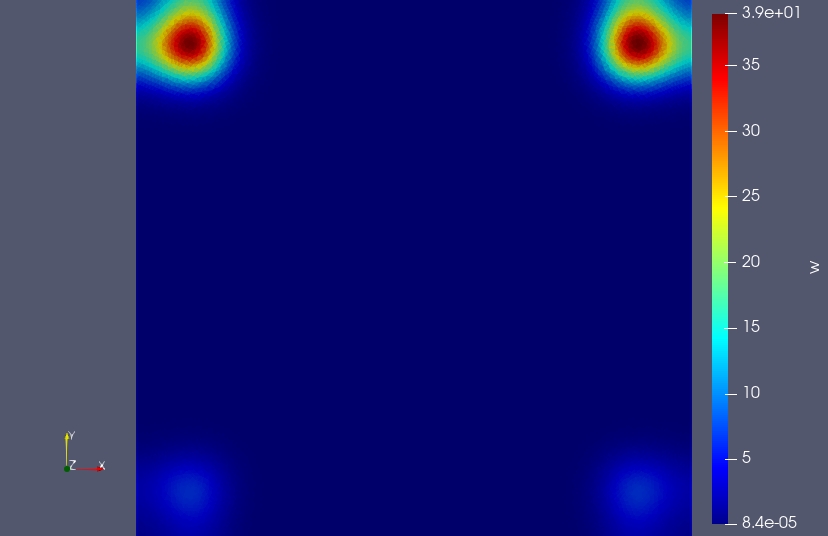} \\
($a_2$) $u,\quad t=4.0$ & ($b_2$) $v,\quad t=4.0$ & ($c_2$) $w,\quad t=4.0$ \\
\includegraphics[scale=0.1125]{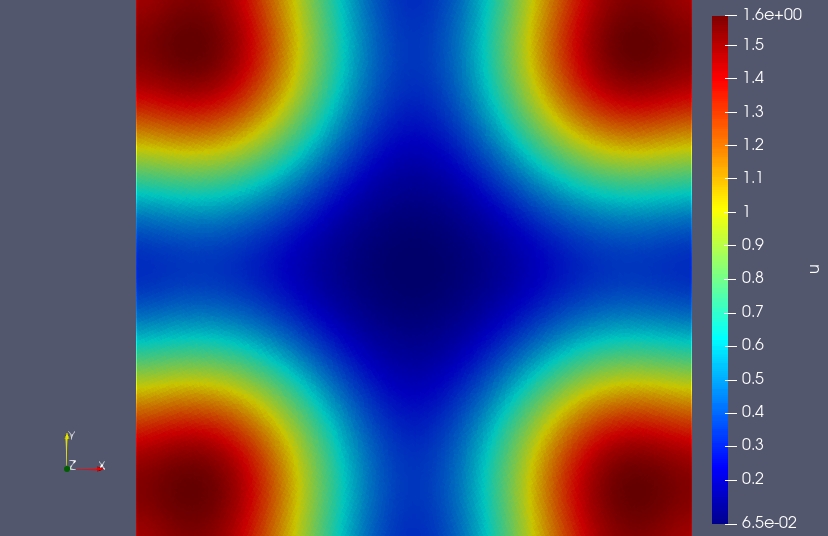} &
\includegraphics[scale=0.1125]{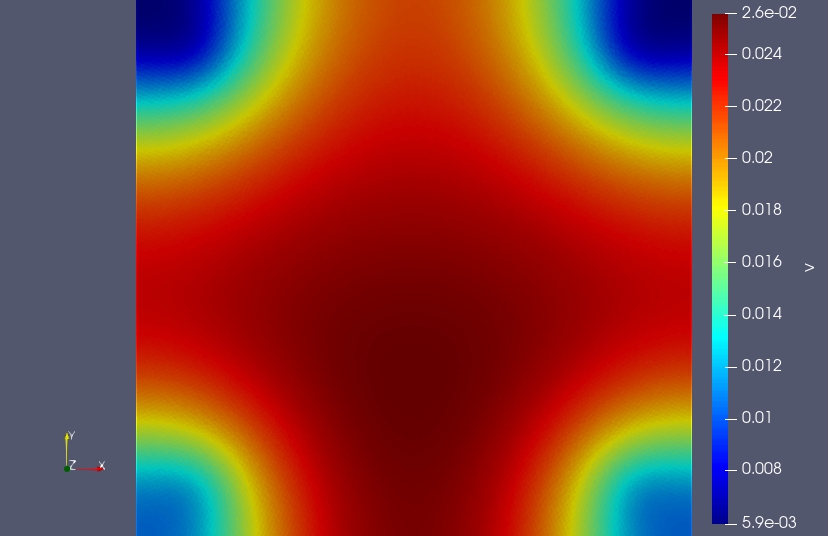} & 
\includegraphics[scale=0.1125]{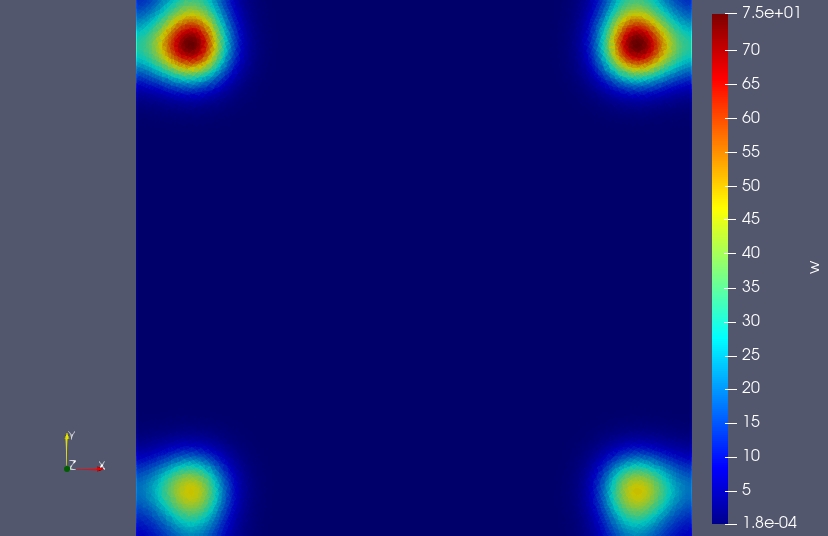} \\
($a_3$) $u,\quad t=20.0$ & ($b_3$) $v,\quad t=20.0$ & ($c_3$) $w,\quad t=20.0$ 
\end{tabular}
\caption{{\em Contour plots} of time evolution of the resource $u$,  mesopredador $v$ and top predador $w$ at different times. $ q=10.0$, $c=1.5$} \label{Figu15}
\end{figure}

\clearpage
It is worth to note that an increment of the predation rate $c$ not necessarily induces an increment on the predator population. In Figure \ref{Figu10} the predation rate is c=1.5, and  the predator population is lesser than the population showed in Figure \ref{Figu9} where the predaton rate is $c=1.0$. This is due, in part,  to the weak attraction of the resource on the individual  predators, as this allows  predators to remain randomly dispersed throughout space preventing the mesopredator population from reaching  a  level  high enough to support a large  population of predators. On the other hand, By comparing Figure \ref{Figu10} with Figure \ref{Figu11} we observe that the main effect on the increment of the attraction paremeter $q$ is on the spatial distribution of meso and top predators. For $q=1.0$ (Figure \ref{Figu11}), top predators tend to occupy the places most densely populated by the resource; in contrast, mesopredators occupy the places least densely populated by top predators. However, if the predation rate is large enough,  the mesopredator population is depleted and spatial complementarity is lost (see Figure \ref{Figu12}). This effect vanishes if he resource's attraction to top predators grows; in fact, for $q=10$, the separation of top and mesopredators habitats  is strengthened  for $c=1$ and all three species reach relatively large populations levels compared to $c=.1$ (See Figure (\ref{Figu13}-\ref{Figu14}). The coexistence of the three species requires a proper balance between the rate of predation and the attraction of predators to the resource population. In Figure \ref{Figu15}, we observe very low mesopredator population levels and a sharp concentration of top predators around the areas most populated by mesopredators.

\section{Conclusions}
With the aim to analyze the rol of migration and defensive mechanisms of the prey, in this work two variations of a  tritrophic model have been considered. According to Table (\ref{tb1}), if the three species remain in the same location (without migration), top predator would become  extinct since only the equliibrium point $P_2$ is stable.  In the first case, where a top predator is an active-search hunter it is assumed that as prey density  increases, searching intensity  decreases (Model (\ref{mod1f}) with $\chi_1(v,w)=e_1w-e_2 v$). Numerical  simulations show  that all three species coexist and both  resource and prey tend to be concentrated around sites $(x^{*},y^{*})\in \Omega$ where resource suitability is greatest; that is, sites $(x^{*},y^{*})$ where  $K(x^{*},y^{*})$ is the maximum. The spatial distribution of predator depends on the defensive capacity of the prey; for $e_2/e_1$ low enough, predators and prey have a similar distribution (see Figures (\ref{Figu3b}), (\ref{ff4})). However, if $e_2/e_1$ reaches a  large enough level, the resource and prey populations share the same space, but the predator occupies the locations less populated by prey(see Figures (\ref{ff2}), (\ref{fig12/3}), (\ref{fig5})).  
One second point of interest in this work is  how the attraction of enemies of my enemies influences the dynamics of a community. We analyzed this question with Model (\ref{mod2f}) where the predator moves toward the resource gradient  according to the sensivity function $\chi_2(u,w)=q u w$; that is, the higher  the population density  of the resource or the predator, the greater the tendency of the predator to move towards the resource. In some cases, the attraction activity is caused by volatiles emitted by the resource organisms. The numerical simulations of  Model (\ref{mod2f}) have focused to get some insight about the impact of the attraction that the resource exerted on the predator on the dynamics of the mesopredator-predator interaction. We observe that if the attraction is low enough, the dynamics is mainly determined  by the intensity of predation on the mesopredator population, and both mesopredators and predators tend to occupy  the sites  most populated  by the resource. The spatial distribution of the three especies shown in   Figure \ref{Figu9} ( $q=0.1, c=1.0$) is very similar to that shown in  Figure \ref{Figu10} ( $q=0.1, c=1.5$). Notice that the greatest population density of predators and mesopredators are closer to region where the resource is most abundant.   However, when the attraction of predators towards the resource increases to a reach a relatively large level,  predators follow th spatial distribution of the reosurce and mesopredators ocuppy zones where predators are scarce. This is shown in Figure \ref{Figu10} , Figure \ref{Figu11}  and Figure \ref{Figu15}
Hence, our numerical simulations provide evidence that migration favors coexistence and behavioral characteristics, such  as a defense mechanism or hunter strategies, can impact the spatial distribution of species. Furthermore, according with the simulations of our two models, we find that the distribution of prey follows a pattern similar to that of the resource, which tends to be distributed near the places of greatest suitability. The cost of a defense mechanism has been considered in \cite{xia} where the authors analyze how this cost impact on pattern distribution of predators and preys.  The role of predators on the spatial distribution has been studied from a experimental point of view in \cite{liv}, where preys do not present a defense against predators. They found that was not the patch type but the distribution of predators  that most strongly predicted the composition of the prey community.  The effect of diffussion on the spatial distribution has beeen analyzed in \cite{nit}.

%\section*{Acknowledgments} We would like to thank you for \textbf{following
%the instructions above} very closely in advance. It will definitely
%save us lot of time and expedite the process of your paper's
%publication.

\section*{Appendix A} \label{Ap A}

System (\ref{mod3f}) has the following equilibrium points

$i)$ $P_{1}\left( 0,0,0\right) $

$ii)$ $P_{2}\left( K,0,0\right) $

$iii)$ $P_{3}\left( \frac{a\mu }{b\gamma -\mu },\frac{a\alpha \gamma \left(
b\gamma K-\mu (a + K)\right) }{K\left( b\gamma -\mu \right) ^{2}},0\right) .$\\
Under appropriate conditions, this system posses two equilibrium points $P_{4}\left( u_{1},v_{1},w_{1}\right) $ and $P_{5}\left(
u_{2},v_{2},w_{2}\right) $ with positive coordinates given by

$u_{1}=\frac{1}{2}\left( -a+K-\sqrt{\frac{c\alpha \beta \left( a+K\right)
^{2}-\left( 4bdK+\left( a+K\right) ^{2}\alpha \right) \nu }{\left( c\beta
-\nu \right) \alpha }}\right) $

$v_{1}=\frac{d\nu }{c\beta -\nu }$

$w_{1}=\frac{\left( d+v_{1}\right) \left( b\gamma u_{1}-\left(
a+u_{1}\right) v_{1}\mu \right) }{c\left( a+u_{1}\right) }$

$u_{2}=\frac{1}{2}\left( -a+K+\sqrt{\frac{c\alpha \beta \left( a+K\right)
^{2}-\left( 4bdK+\left( a+K\right) ^{2}\alpha \right) \nu }{\left( c\beta
-\nu \right) \alpha }}\right) $

$v_{2}=\frac{d\nu }{c\beta -\nu }$

$w_{2}=\frac{\left( d+v_{2}\right) \left( b\gamma u_{2}-\left(
a+u_{2}\right) v_{2}\mu \right) }{c\left( a+u_{2}\right) }.$

\noindent Point $P_{1}$ is always unstable;  $P_{2}$ is locally asymptotically stable if $bK\gamma -a\mu -K\mu <0$ and unstable if $bK\gamma -a\mu
-K\mu >0$; $P_{3}$ is stable if $bK\gamma -a\mu -K\mu >0$ and $b\gamma
a>bK\gamma -a\mu -K\mu $ and unstable if $bK\gamma -a\mu -K\mu >0$ and $b\gamma
a<bK\gamma -a\mu -K\mu $.

If $bK\gamma -a\mu -K\mu <0$, point $(K,0,0)$ is a stable equilibrium point of system \ref{mod3f} (see appendix XX). In the following theorem, we prove that  stability of this point is also preserved in the system \ref{mod1f}.
Let ~$0=\mu _{1}<\mu _{2}<\mu _{3}...$ be the eigenvalues of the operator $%
-\Delta $ on $\Omega $ with Neumann boundary conditions and let 
$E\left( \mu _{i}\right) $ be the eigenspace corresponding to $\mu _{i}$ in 
$C^{1}\left( \overline{\Omega }\right) .$

\medskip
% The information below will be filled in by AIMS editorial staff
Received xxxx 20xx; revised xxxx 20xx; early access xxxx 20xx.
\medskip

\end{document}